\documentclass[10pt]{amsart}

\usepackage{amssymb, amsmath, amsthm, amsfonts}
\usepackage{mathrsfs,comment}
\usepackage{graphicx}
\usepackage{placeins}
\usepackage{todonotes}
\usepackage{rotating}
\usepackage{tikz}
\usepackage{float}
\usepackage{hvfloat}
\usepackage{caption}
\usepackage{xypic}
\usepackage{mathscinet}

\usepackage{pdflscape}

\usepackage{url}
\usepackage[all,arc,2cell]{xy}
\UseAllTwocells
\usepackage{enumerate}
 \usepackage{blindtext}
\usepackage{verbatim}
\usepackage{sseq}

\usepackage{hyperref}
  \definecolor{dark-red}{rgb}{0.6,0.15,0.15}
   \definecolor{dark-blue}{rgb}{0.15,0.15,0.6}
   \definecolor{medium-blue}{rgb}{0,0,0.5}

\hypersetup{
    colorlinks, 
    linkcolor=dark-red,
    citecolor=dark-blue, urlcolor=medium-blue
}

\usepackage[nameinlink,capitalise,noabbrev]{cleveref}
\usepackage[normalem]{ulem}
\newcommand{\stkout}[1]{\ifmmode\text{\sout{\ensuremath{#1}}}\else\sout{#1}\fi}

\numberwithin{equation}{section}
\theoremstyle{plain}
\newtheorem{thm}[equation]{Theorem}
\newtheorem{cor}[equation]{Corollary}

\newtheorem{prop}[equation]{Proposition}
\newtheorem{lem}[equation]{Lemma}
\newtheorem*{linhypo}{Linearization Hypothesis}

\newtheorem*{thmintro1}{\Cref{thm:dual-en}}
\newtheorem*{thmintro2}{\Cref{thm:a-big-one}}
\newtheorem*{thmintro3}{\Cref{thm:whod-have-believed-it-1}}
\newtheorem*{propintro}{\Cref{prop:recog-princ} and \Cref{prop:recog-princ-bis}}
\newtheorem*{thmintro4}{\Cref{thm:dualn=p=3} and \Cref{thm:dualn=p=2}}
\newtheorem*{thmintro5}{\Cref{thm:dualn=p-1}}
\newtheorem*{thmintro6}{\Cref{thm:pic-exotic}}

\theoremstyle{definition}
\newtheorem{defn}[equation]{Definition}
\newtheorem{notation}[equation]{Notation}

\newtheorem{exam}[equation]{Example}

\newtheorem{rem}[equation]{Remark}

\def\quickop#1{\expandafter\newcommand\csname #1\endcsname{\operatorname{#1}}}
\quickop{Hom} \quickop{End} \quickop{Aut} \quickop{Tel} \quickop{Mic} \quickop{map}
\quickop{Ext} \quickop{Tor} \quickop{Cotor} \quickop{Id} \quickop{Coker} \quickop{Ker}
\quickop{Lim} \quickop{Colim} \quickop{Holim} \quickop{Hocolim}
\quickop{id} \quickop{tel} \quickop{mic} \quickop{coker}
\quickop{colim} \quickop{holim} \quickop{hocolim} \quickop{im}

\newcommand{\F}{\mathbb{F}}
\newcommand{\R}{\mathbb{R}}

\newcommand{\G}{\mathbb{G}}

\newcommand{\Q}{\mathbb{Q}}
\newcommand{\smsh}{\wedge}

\newcommand{\cC}{\mathcal{C}}
\newcommand{\cU}{\mathcal{U}}

\newcommand{\ZZ}{\mathbb{Z}}
\newcommand{\FF}{\mathbb{F}}
\newcommand{\RR}{\mathbb{R}}

\DeclareFontFamily{OMS}{rsfs}{\skewchar\font'60}
\DeclareFontShape{OMS}{rsfs}{m}{n}{<-5>rsfs5 <5-7>rsfs7 <7->rsfs10 }{}
\DeclareSymbolFont{rsfs}{OMS}{rsfs}{m}{n}
\DeclareSymbolFontAlphabet{\scr}{rsfs}

\def\makeop#1{\expandafter\def\csname #1\endcsname{\mathop{\mathrm{#1}}\nolimits}}

\makeop{Gal}
\makeop{id}
\makeop{Mod}
\makeop{Tot}
\makeop{gr}
\makeop{Out}
\makeop{Hom}
\makeop{Ext}
\makeop{End}
\makeop{Aut}
\makeop{Tor}
\makeop{ev}
\makeop{Pic}
\def\FF{\mathbb{F}}
\def\QQ{\mathbb{Q}}
\def\GG{\mathbb{G}}
\def\SS{\mathbb{S}}
\def\WW{{{\mathbb{W}}}}
\def\ZZ{{{\mathbb{Z}}}}
\def\Z{{{\mathbb{Z}}}}
\def\Ext{\mathrm{Ext}}
\def\longr{{{\longrightarrow\ }}}
\def\sPic{{\mathcal{P}\mathrm{ic}}}
\def\Pic{\mathrm{Pic}}
\def\uPic{{{\underline{\Pic}}}}
\def\Rep{{\mathrm{Rep}}}

\definecolor{darkspringgreen}{rgb}{0.09, 0.45, 0.27}

   \def\cE{{{\mathcal{E}}}}
   \def\Gl{{\mathrm{Gl}}}
\def\gg{\mathfrak{g}}
\def\cO{{{\mathcal{O}}}}
\def\cG{{{\mathcal{G}}}}
\def\cU{{{\mathcal{U}}}}
\def\RR{{{\mathbb{R}}}}
\def\ig{{I_{\cG}}}

\def\tr{{\mathrm{tr}}}
\def\pic{\mathfrak{pic}}
\def\Pic{\mathrm{Pic}}

\newcommand{\ko}{\mathrm{ko}}
\newcommand{\ku}{\mathrm{ku}}
\newcommand{\kon}[1]{\ko\langle #1\rangle}

\newcommand{\ion}[1]{I \langle #1\rangle}
\newcommand{\ionp}[1]{I_p \langle #1\rangle}
\newcommand{\iontwo}[1]{I_2 \langle #1\rangle}
\newcommand{\ionthree}[1]{I_3 \langle #1\rangle}
\def\CC{{{\mathbb{C}}}}

\def\LTE{{{\mathbf{E}}}}
\def\LTK{{{\mathbf{K}}}}
\def\igg{{I_{\GG}}}
\def\mm{{\mathfrak{m}}}
\newcommand{\MJ}[1]{\mathrm{M}_{J(#1)}}
\def\Func{{{F_c}}}

\def\Sp{{\mathbf{Sph}}}
\def\haut{{\mathrm{Gl_1}}} 
\def\GL1{{\mathrm{Gl_1}}}
\def\SL1{{\mathrm{Sl_1}}}
\def\cK{{{\mathcal{K}}}}
\def\gl{{{\mathrm{gl}}}}
\def\Sq{{\mathrm{Sq}}}

\def\pic{{\mathfrak{pic}}}

\def\res{\mathrm{res}}
\def\tr{\mathrm{tr}}

\def\LTE{{{\mathbf{E}}}}
\def\LTK{{{\mathbf{K}}}}
\def\igg{{I_{\GG}}}
\def\cE{{{\mathcal{E}}}}

\def\sPic{{\mathcal{P}\mathrm{ic}}}
\def\uPic{{{\underline{\Pic}}}}
\def\Rep{{\mathrm{Rep}}}

\def\tr{{\mathrm{tr}}}

\def\maps{\mathbf{map}} 
\def\mapsp{\mathfrak{map}} 

\def\gGa{{\prescript{g}{}{\Gamma}}}
\def\gMa{{\prescript{g}{}{M}}}

\def\fMf{{f_\ast M^f}}
 \def\FP{{\mathbf{FP}}}
\def\pair{{\langle -, -\rangle}}
\def\barE{{\overline{E}}}

\def\bone{{\Sigma gl_1}}

\def\HH{{\mathbb{H}}}
\def\su2{\mathfrak{su}(2)}

\usepackage{mathtools}

\title[Dualizing spheres]{Dualizing spheres for compact $p$-adic analytic groups and duality in chromatic homotopy}
\date{\today}

\author[A. Beaudry]{Agn\`es Beaudry}
\address{Department of Mathematics\\ University of Colorado Boulder}

\author[P. Goerss]{Paul G. Goerss}
\address{Department of Mathematics\\ Northwestern University}

\author[M.J. Hopkins]{Michael J. Hopkins}
\address{Department of Mathematics\\Harvard University}

\author[V. Stojanoska]{Vesna Stojanoska}
\address{Department of Mathematics\\ University of Illinois at Urbana-Champaign}

\thanks{This material is based upon work supported by the National Science Foundation under grants No.~DMS--1510417, DMS--1812122 and DMS--1906227. The authors also thank the Isaac Newton Institute for Mathematical Sciences for support and hospitality during the program Homotopy Harnessing Higher Structures. This work was supported by EPSRC Grant Number EP/R014604/1.
}

\begin{document}

\begin{abstract}
The primary goal of this paper is to study Spanier-Whitehead duality in the $K(n)$-local category. One of the key players in the $K(n)$-local category is the Lubin-Tate spectrum $E_n$, whose homotopy groups classify deformations of a formal group law of height $n$, in the implicit characteristic $p$. It is known that $E_n$ is self-dual up to a shift; however, that does not fully take into account the action of the Morava stabilizer group $\mathbb{G}_n$, or even its subgroup of automorphisms of the formal group in question. In this paper we find that the $\mathbb{G}_n$-equivariant dual of $E_n$ is in fact $E_n$ twisted by a sphere with a non-trivial (when $n>1$) action by $\mathbb{G}_n$. This sphere is a dualizing module for the group $\mathbb{G}_n$, and we construct and study such an object $I_{\mathcal{G}}$ for any compact $p$-adic analytic group $\mathcal{G}$. If we restrict the action of $\mathcal{G}$ on $I_{\mathcal{G}}$ to certain type of small subgroups, we identify $I_{\mathcal{G}}$ with a specific representation sphere coming from the Lie algebra of $\mathcal{G}$. This is done by a classification of $p$-complete sphere spectra with an action by an elementary abelian $p$-group in terms of characteristic classes, and then a specific comparison of the characteristic classes in question. The setup makes the theory quite accessible for computations, as we demonstrate in the later sections of this paper, determining the $K(n)$-local Spanier-Whitehead duals of $E_n^{hH}$ for select choices of $p$ and $n$ and finite subgroups $H$ of $\mathbb{G}_n$.
\end{abstract}
\subjclass[2000]{55P92; 55R25; 55R50; 55U30; 20E18; 22E41}

\maketitle

\setcounter{tocdepth}{1}
\tableofcontents

\numberwithin{equation}{section}


\section{Introduction}

This project began with a contemplation of duality in the $K(n)$-local stable homotopy category, so let us start with
explaining why this is an interesting topic.
One of the standard approaches to stable homotopy theory emphasizes complex oriented cohomology theories, those with
a natural theory of Chern classes. Any such cohomology theory determines 
a smooth $1$-parameter formal group, and the algebraic geometry of formal groups can be used to organize
calculations and the search for large scale phenomena. Over an algebraically closed field of characteristic $p$, 
formal groups are classified up to isomorphism by a single invariant, the height. Fix a prime $p$ and a height $n$, and choose
a representative $F_n$ for this isomorphism class. It is convenient to assume $F_n$ is defined over
some finite field $\FF_q$ of characteristic $p$, and usually $q=p^n$. There is a complex oriented cohomology theory
$ \LTK^\ast:=K(n)^\ast  = K(\FF_q,F_n)^\ast$ determined by the pair $(\FF_q,F_n)$. Let $\LTK$ be the representing spectrum; this is a Morava $K$-theory. 
One standard approach is to then study the homotopy theory of $\LTK$-local spectra in the sense of Bousfield \cite{bousfield_locspectra} and,
perhaps later, assemble that information into a more global picture. 

In many ways, the $\LTK$-local category is much better behaved than the full stable homotopy category. To begin with there
is a very effective computational tool based on Lubin-Tate (or Morava) $E$-theory.
The universal deformation of $F_n$ determines a Landweber
exact complex oriented cohomology theory $ \LTE^\ast:=E^\ast_n  = E(\FF_q,F_n)^\ast$. Let $\LTE$ be the representing spectrum.
If we let $\GG$ be the automorphism group of the pair $(\FF_q,F_n)$, then $\GG$ is a profinite group and acts on $\LTE$. If $X$
is any spectrum we define
\[
\LTE_\ast X = \pi_\ast L_{\LTK}(\LTE \wedge X). 
\]
Then $\LTE_\ast X$ comes equipped with an $\LTE_\ast$-module structure and a compatible continuous action of $\GG$; we say
$\LTE_\ast X$ is a {\it Morava} module.  The $\LTK$-local Adams-Novikov Spectral
Sequence based on $\LTE$ reads
\[
H^s(\GG,\LTE_t X) \Longrightarrow \pi_{t-s}L_{\LTK}X.
\]
Here we are using continuous group cohomology. If $p$ is large with respect to $n$, this spectral sequence collapses
for some important examples, such as $X = S^0$. Even when it does not collapse the Morava module
$\LTE_\ast X$ is a very sensitive and informative algebraic invariant of $X$. 

The group $\GG$ is not simply a profinite group; it is a compact $p$-adic analytic group. It is a classical observation that the category of 
continuous modules over such groups behaves very much like the category of quasi-coherent sheaves on a very nice
projective scheme; for example, there is a very good notion of Grothendieck-Serre duality. See \cite{CohoGal} for a classical
source and especially \cite{SymWei} for a thorough modern treatment. The assignment $X \mapsto \LTE_\ast X$ is not
an equivalence of categories in general. (But see \cite{Piotr} for more on this point.) Nonetheless,
the structure of continuous $\GG$-modules is a very good indicator of what might be true for the $\LTK$-local category.
In particular, there are rich theories of duality in the $\LTK$-local category that can be glimpsed by studying Morava modules. The investigation of these dualities has been an important aspect of research in  chromatic homotopy theory. 

Here, we will concentrate on Spanier-Whitehead Duality in the $\LTK$-local category. Any spectrum $X$ has a $\LTK$-local
Spanier-Whitehead dual defined by the function spectrum
\[
DX = F(X,L_{\LTK}S^0).
\]
The spectrum $L_{\LTK}S^0$ is the unit for the natural symmetric monoidal structure on the $\LTK$-local category, so this is a very basic
duality. There is another duality, more closely related to Grothendieck-Serre duality, known as Gross-Hopkins (or $\LTK$-local Brown-Comenetz)
duality. Without going into detail, if $I_n(-)$ denotes Gross-Hopkins duality, then $D(X) \wedge I_n \simeq I_n(X)$ for a certain invertible
object $I_n$ in the $\LTK$-local category, so the two dualities are closely related. For more on this point, see \cite{StrickGH}
and \cite{666}.

Spanier-Whitehead duality in the global (i.e. unlocalized) stable homotopy category is often extremely hard to compute and behaves in strange and surprising ways when $X$ is not finite. However,
it is much better behaved in the $\LTK$-local category. For example, quite a few years ago, the third author noted that
$D\LTE$ is essentially $\LTE$; in fact, since $\GG$ is a compact $p$-adic analytic group of rank (or dimension) $n^2$ and
since $\LTE^\ast \LTE$ is isomorphic, as a graded $\GG$-module to the profinite group ring $\LTE^\ast[[\GG]]$, there is an isomorphism of continuous Morava modules
\begin{equation}\label{eq:intro0}
\pi_\ast D\LTE \cong \pi_\ast \Sigma^{-n^2}\LTE .
\end{equation}
See \cite{StrickGH} for details. Note that we could equally write $\pi_\ast D\LTE \cong \pi_\ast F(S^{n^2},\LTE)$.
We could then hope that this isomorphism is induced by a $\GG$-equivariant equivalence 
$D\LTE \simeq \Sigma^{-n^2}\LTE$, 
but this is false. In fact, it is not true even when restricted to
certain key finite subgroups of $\GG$. 
Indeed, Behrens \cite{Beh44} and Bobkova \cite{BobkovaSW} have computed $D(\LTE)^{hH}$ for maximal finite subgroups $H$ of $\GG$ at $n=2$, when $p=3$ and $p=2$, respectively, and their results, in particular, show that $D\LTE $ is not $\GG$-equivariantly equivalent to $\Sigma^{-n^2} \LTE$.
Our first goal was to understand this subtlety. 

Ultimately, it turns out that this is not a question about $\LTE$, but a question
about $\GG$. In fact, it is fairly formal, if technically formidable, to arrive at the following result.

\begin{thmintro1}
There is a $p$-complete spectrum $I_{\GG}$ with an action of $\GG$, whose underlying
$p$-complete homotopy type is $S^{n^2}$, such that there is a $\GG$-equivariant equivalence
\[
D\LTE \simeq F(I_{\GG},\LTE).
\]
\end{thmintro1}

By combining \Cref{rem:when-finite-coh} and \Cref{lem:top-action} we also have that the action of $\GG$ on
$H_{n^2}I_{\GG} \cong \ZZ_p$ is trivial. These results imply the formula \eqref{eq:intro0}.

The question then is to understand the equivariant homotopy type of $I_{\GG}$. To do so, it is helpful to generalize 
and replace $\GG$ by an arbitrary compact $p$-adic analytic group $\cG$. Any such $\cG$ has an exhaustive nested sequence of open subgroups, all normal in $\cG$,
\[
\cdots \subseteq \Gamma_{i+1} \subseteq \Gamma_i \subseteq \cdots \subseteq \Gamma_1 \subseteq \Gamma_0 = \cG,
\]
with the property that for large enough $i$, each $\Gamma_i$ is a particularly nice $p$-profinite group known as a uniformly powerful group. We review this theory in \Cref{sec:p-adicBasics}.
We have that $\Gamma_{i+1} \subseteq \Gamma_i$ is of finite index and $\cG$ acts on the classifying space
$B(\Gamma_i/\Gamma_j)$ by conjugation. We first define the $\cG$-space 
\[
B\Gamma_i = \holim_j B(\Gamma_i/\Gamma_j),
\]
where the limit is over the projections $\Gamma_i/\Gamma_{j+1} \to \Gamma_i/\Gamma_j$ and then we define
a $p$-complete $\cG$-spectrum
\[
\ig = \left( \hocolim_i \Sigma_+^\infty B\Gamma_i \right)^\wedge_p,
\]
where $\Sigma_+^\infty$ denotes the suspension spectrum after adding a disjoint basepoint. The colimit is over the transfer maps
\[
\tr \colon  \Sigma_+^\infty B\Gamma_i \to  \Sigma_+^\infty B\Gamma_{i+1},
\]
which are equivariant with respect to the $\cG$-action; see \Cref{rem:naturality-tr}.
Then $\ig$ has the underlying homotopy type of a $p$-complete $d$-sphere, where $d$ is the rank of $\cG$. 
See \Cref{lem:transfer-hom}. Even in the cases where $\cG = \GG$ or $\Gl_n(\ZZ_p)$ the resulting action of $\cG$ on $\ig$ is rich and potentially very non-linear, since the action on the building blocks is rich and highly non-linear.

There is a linear analog
of $\ig$ which is much simpler for computations. Because $\cG$ is a compact $p$-adic analytic group it has a $p$-adic Lie algebra
$\gg$; this is a free $\ZZ_p$-module of rank $d$, where $d$ is the rank of $\cG$.
The conjugation action makes $\gg$ into a $\cG$-module called the adjoint representation. We can define a linear $\cG$-sphere
\[
S^\gg = \left( \hocolim_i \Sigma_+^\infty B(p^i\gg) \right)^\wedge_p
\]
using this representation, where again the colimit is along transfers.
Then $S^\gg$ has the underlying homotopy type of a $p$-complete $d$-sphere and there is an isomorphism of $\cG$-modules
\[
H_d (\ig,\ZZ_p) \cong \Lambda^d \gg \cong H_d (S^\gg,\ZZ_p),
\]
as we recall in
\Cref{lem:top-action}.
Here $\Lambda^d \gg$ denotes the top exterior power of the adjoint representation. This is the classical
dualizing module for continuous $\cG$-modules. Again, see \cite{CohoGal} or \cite{SymWei} for details. 
In analogy with the classical algebraic duality statements discussed in these sources, one is led to the following conjecture.
\begin{linhypo} 
In an appropriate category of continuous $\cG$-spheres there is a
$\cG$-equivariant equivalence $S^\gg \simeq \ig$.
\end{linhypo}

This conjecture  appeared in the work of Clausen, \cite[\S 6.4]{Clausen}. Subsequently, in  \cite{OWS}, Clausen
announced a proof of this conjecture in full generality, including a discussion of
what category is a natural home for the spectrum $\ig$ with its $\cG$-action. Looking further back,
duality and dualizing modules in the $\LTK$-local category featured in work from the early '90s by the third
author and his collaborators. See \cite{DevHopAct}, especially the last section, and \cite{HopkinsGross}. 

\begin{rem}
The analogue of the Linearization Hypothesis for compact Lie groups was stated and proved by Klein in \cite{Klein}. In a different (genuinely equivariant) setting, the Wirthm\"uller isomorphism can also be cast as a duality statement for compact Lie groups as in Fausk-Hu-May \cite{FauskHuMay}.
\end{rem}

In this paper, we do not prove the Linearization Hypothesis in its full strength, but we show it holds when the action of $\cG$ is restricted to certain small subgroups. 
Here we write $Z(\cG)$ for the center of $\cG$.

\begin{thmintro2} Let $\cG$ be compact $p$-adic analytic group and let $H$ be a closed subgroup
of $\cG$ such that $H/H \cap Z(\cG)$ is finite. Suppose the $p$-Sylow subgroup of $H/H \cap Z(\cG)$
is an elementary abelian $p$-group. Then there is an 
$H$-equivariant equivalence
\[
\ig \simeq S^\gg.
\]
\end{thmintro2}

Although we do not prove the Linearization Hypothesis in its full strength, our methods are very different from Clausen's and are valuable in and of themselves: They allow us to access the result for specific computations, as we demonstrate in the latter sections of this paper. A major input in the proof are techniques from Lannes theory \cite{Lannes}, leading up to the following result, which we deduce directly from the work of Castellana \cite{Castellana}.

\begin{thmintro3} Let $F$ be an elementary abelian $p$-group and let $X$ be a
$p$-complete $F$-sphere of virtual dimension $k$. 
Then there is stable vector
bundle $\xi$ over $BF$ of virtual dimension $k$ and a $p$-equivalence of spectra
\[
M\xi \simeq EF_+ \wedge_F X.
\]
Furthermore there is an $F$-equivalence $X \simeq Y$ of $p$-complete $F$-spheres if and only if there
there is an isomorphism of modules over the Steenrod algebra
\[
H^\ast(EF_+ \wedge_F X) \cong H^\ast(EF_+ \wedge_F Y).
\]
Such an isomorphism uniquely determines the $F$-equivalence up to $F$-homotopy.
\end{thmintro3}

This result tells us that any stable $H$-sphere, and in particular either $\ig$ or $S^\gg$, is linearizable, and as such, determined by its 
characteristic classes. We have included the proof here in its entirety, as many of the details are missing in loc.cit. We then dive 
deeply into the constructions of $\ig$ and $S^\gg$ to show that their characteristic classes match.

The natural question that arises at this point is whether our methods for proving \Cref{thm:a-big-one} can be bootstrapped to the case 
when $H/H\cap Z(\cG)$ is a larger finite group. This is  tied to the problem of classifying stable spheres with an action of a finite group. 
This problem, as well as its unstable version, are well studied in homotopical representation theory, though the methods there
for using Lannes theory to obtain results about general groups do not apply here. Nevertheless, work of Jesper Grodal and Jeff 
Smith \cite{OWSGrodalSmith} may point the way toward generalizations.

The essence of linearizability is distilled in the following result, which makes \Cref{thm:a-big-one} accessible for computations. It gives 
a hold on the $H$-equivariant homotopy type of $S^\gg$. We prove this by implementing ideas from geometric topology. 
For simplicity, we limit our attention to finite subgroups of $\cG$, but this is not restrictive in practice.

\begin{propintro} Let $\cG$ be a compact $p$-adic analytic group and let
$F \subseteq \cG$ be a finite subgroup. Suppose there is
a finitely generated free abelian group $L \subseteq \gg$ with the properties that 
\begin{enumerate}

\item $L$ is stable under the adjoint action of $F$ on $\gg$, and

\item $L/pL \cong \gg/p\gg$ (or, more generally, $\QQ_p \otimes L \cong \QQ_p \otimes_{\ZZ_p} \gg$).
\end{enumerate}
Let $V = \RR \otimes L$ and let $S^V$ be the one-point compactification of $V$. Then there is
an $F$-equivariant map $S^V \to S^\gg$ which becomes a weak equivalence after completion at $p$. 
\end{propintro}

This result unlocks the applications of \Cref{thm:a-big-one} in the case of the Morava stabilizer group $\GG$ in chromatic
homotopy  theory.
Specifically, we develop a unified strategy to address the problem of determining $\LTK$-local Spanier-Whitehead
duals $D(\LTE^{hH})$ of spectra of the form $\LTE^{hH}$, for $H$ as in \Cref{thm:a-big-one}. For this application, we
need more than the above duality results applied to the profinite group $\GG$. Namely, we need a good understanding
of invertible equivariant $\LTE$-modules. We develop this is \Cref{sec:PicStuff}, where again the theory of characteristic 
classes plays a leading role. Further, we revisit the known examples  mentioned above but also explore new cases.  An 
important part of this paper develops this new strategy for addressing familiar questions in chromatic homotopy theory.

With that in hand, we apply the observation that if $H$ is finite, then Tate vanishing \cite{GreenSad} 
implies that there is an equivalence
\[
D(\LTE^{hH}) \simeq (D\LTE)^{hH}. 
\]
This (now) standard but subtle fact allows us to combine our results on the equivariant homotopy type of $D\LTE$ and the
results on equivariant invertible $\LTE$-modules to calculate $D(\LTE^{hH})$ for interesting $H$. 

As a first example, if $H \subset Z(\GG)$ is a finite subgroup of the center of $\GG$, then
\begin{align}\label{eq:dualCentral}
D(\LTE^{hH}) \simeq \Sigma^{-n^2} \LTE^{hH}.
\end{align}
This follows from \Cref{thm:dual-en} and the fact that $H$ acts trivially on $I_{\GG}$. The center of $\GG$ is not large; indeed,
$Z(\GG) \cong \ZZ_p^\times$, the units in the $p$-adic integers. Hence, this example has
the most impact when $p=2$ and 
\[H = C_2 = \{\pm 1\} \subseteq \Aut(F_n/\FF_{p^n}) \subseteq \GG.\]

The full force of \Cref{thm:dual-en}, \Cref{thm:a-big-one}, and \Cref{prop:recog-princ} can also be used to recover the 
results of Behrens \cite{Beh44} and Bobkova \cite{BobkovaSW}.
If $n$ is small with respect to $p$, then $\GG$ can contain $p$-torsion elements. For example
if $n=2$ and $p=2$, the maximal $2$-torsion subgroup is isomorphic to the quaternionic group $Q_8$ of order $8$.
If $n=2$ and $p=3$, the maximal $3$-torsion subgroup is a cyclic group $C_3$ of order $3$. In the result below, $\G=\Aut(\F_{p^2}, F_{C})$ where $F_C$ is the formal group law of certain supersingular elliptic curves. 

\begin{thmintro4} Let $n=2$ and $p=2$ or $3$. Let $F \subseteq G \subseteq \GG$ where $G$ is a maximal finite subgroup of $\GG$
containing the maximal $p$-torsion subgroup. Then
\[
D(\LTE^{hF}) \simeq \Sigma^{44} \LTE^{hF}.
\]
\end{thmintro4}

The spectrum $\LTE^{hG}$ is the $\LTK$-localization of the spectrum of topological modular forms. 
At either prime, the Morava module $\LTE_\ast \LTE^{hG}$ is $24$-periodic; from this
and \eqref{eq:intro0} one can prove $\LTE_\ast D(\LTE^{hG}) \simeq \Sigma^{24k-4}\LTE^{hG}$ for some integer $k$.
The difficulty is to understand why $k$ must be $2$.

For $F=G$, \Cref{thm:dualn=p=3}  ($p=3$) was originally proved by Behrens in \cite{Beh44} and \Cref{thm:dualn=p=2}  ($p=2$) by Bobkova in \cite{BobkovaSW}.
Both papers used delicate calculations based on the theory of topological resolutions from \cite{GHMR} and
\cite{BobkovaGoerss}. We have replaced that style of argument with one that depends ultimately on the representation
theory of $C_3$ and $Q_8$, respectively.

The case $p=3$ of \Cref{thm:dualn=p=3} can be extended to an arbitrary odd prime $p$ and height $n=p-1$. In that case, the maximal 
$p$-torsion subgroup of $\GG$ is $C_p$, a cyclic group of order $p$, and its representation theory helps us prove the following result. Here, $\GG=\Aut(\F_{p^n}, F_n)$ where $F_n$ is the Honda formal group law.

\begin{thmintro5} Let $p > 2$ and $n=p-1$. Let $G \subseteq \GG$ be a maximal finite 
subgroup containing  the maximal $p$-torsion subgroup of $\GG$. Then
\[
D(\LTE^{hG}) \simeq \Sigma^{-(p-1)^2(2p+1)} \LTE^{hG}.
\]
\end{thmintro5}

Note that  at $n=2$ and $p=3$, the spectrum $\LTE^{hG}$ is $72$-periodic, so \
$\Sigma^{-28}\LTE^{hG} \simeq \Sigma^{44}\LTE^{hG}$. Compare \Cref{thm:dualn=p=3} and \Cref{thm:dualn=p-1}. 

Finally, a direct application of \Cref{thm:dualn=p-1} gives us a new result about the $\LTK$-local Picard group $\mathrm{Pic}_n$. Namely, 
we compare our result on Spanier-Whitehead duality with results about Gross--Hopkins duals of the same spectra from \cite{BBS}. At 
the very end of this paper, we prove the following result.

\begin{thmintro6} Let $n=p-1$ and  $F=C_p$. Then, there is an element $P_n \in \Pic_n$  
such that $\LTE_*P_n \cong \LTE_*$ as Morava modules but
\[
P_n \smsh \LTE^{hF} \simeq \Sigma^{p^2+p} \LTE^{hF}.
\]
In particular, $P_n$ is a non-trival element of subgroup $\kappa_n$ of exotic elements in the Picard group $\mathrm{Pic}_n$.
\end{thmintro6}

In an analogous way, our main result in the case $n=2=p$ from \eqref{eq:dualCentral} is used by Heard-Li-Shi \cite{HeardEtal} in 
combination with their computation of $I_n(\LTE^{hC_2})$ to prove that there are exotic invertible elements in the $\LTK$-local 
category at $p=2$.

\subsection*{Organization of the paper}

The next three sections are a review of results found throughout the literature needed later in the paper.
In \Cref{sec:p-adicBasics}, we begin by reviewing some theory on compact $p$-adic analytic groups and introduce what will later be our main example, the Morava stabilizer group $\mathbb{G}$. In \Cref{sec:cohanapp}, we discuss properties of the mod $p$ cohomology of these groups. \Cref{sec:dualityandfrob} is a review of Poincar\'e duality and of Frobenius reciprocity for compact $p$-adic analytic groups.

\Cref{sec:TwoSpheres} begins our investigation of new results. There, we introduce the spheres $\ig$ and $S^\gg$ and state the Linearization Hypothesis. In \Cref{sec:resfinite}, we study spheres with actions of finite subgroups, reviewing the classical theory of Lannes and key facts about characteristic classes. This is where we deduce \Cref{thm:whod-have-believed-it-1} from the work of Castellana. We apply it to $\ig$ and $S^{\gg}$ in \Cref{thm:whod-have-believed-it}. This reduces the proof that $\ig \simeq_H S^{\gg}$ (\Cref{thm:a-big-one}) to a cohomological calculation. For this computation, we need to do a tricky analysis of Lyndon-Serre-Hochschild spectral sequences. We have opted to isolate these technical aspects to \Cref{sec:SStech}, which can easily be skipped on a first read. \Cref{sec:squeenrodstairs} is dedicated to the proof of \Cref{thm:a-big-one}. 

In \Cref{sec:anal-lin}, we get ready for our computation in the examples to follow. This is a study of Frobenius reciprocity for $G$-manifolds. We do not claim originality for the results of this section, but include them here because we did not find them conveniently gathered in literature. Based on these geometric results, we prove \Cref{prop:recog-princ} and \Cref{prop:recog-princ-bis} at the end of this section.

The last sections are applications of our results to chromatic homotopy theory. \Cref{sec:e-theory} and \Cref{sec:dualofe} study Lubin-Tate theory $\LTE$. In \Cref{sec:e-theory} we introduce technical background from chromatic homotopy theory and in \Cref{sec:dualofe} we prove \Cref{thm:dual-en}. In \Cref{sec:PicStuff}, we review equivariant techniques established by Hill-Hopkins-Ravenel that can be used study the Picard groups of categories of $R$-modules in $G$-spectra, as well as establish our general strategy for computing the duals $D(\LTE^{hF})$ for finite subgroups $F \subseteq \GG$ in the final parts of the paper. The last two sections are the computational applications of our theory. \Cref{sec:44} contains the proof of \Cref{thm:dualn=p=3} and \Cref{thm:dualn=p=2} and \Cref{sec:newxample} that of \Cref{thm:dualn=p-1} and \Cref{thm:pic-exotic}.

\subsection{Acknowledgments} 

This paper is a result of a long journey, and a great many people helped along the way. The project began in the spring of 2016
during an extended visit by Mike Hopkins to Northwestern University, as part of his residency after winning the
Nemmers Prize in  Mathematics. The initial question we set ourselves was to find some equivariant explanation for the
results of Behrens and  Bobkova mentioned above; that is, why $44$? We quickly understood how to prove
\Cref{thm:dual-en} and then, after blithely assuming the Linearization Hypothesis and some geometric topology, we began making the
calculations found in the final sections of this  paper. Thus, we first have to acknowledge and thank Northwestern University
and the Nemmers Committee for this opportunity for an extended period of collaborative research.

Coming to terms with the Linearization Hypothesis took some time, and some of that intellectual pilgrimage was documented
in a lecture by Hopkins at the conference in July 2017 ``Homotopy Theory: Tools and Applications" at the University of
Illinois Champaign-Urbana. This lecture, now known as the ``Mean Streets of Evanston" talk (``After 10PM, after the Whole
Foods closes, when only sinners are on the streets...'') has been widely viewed on YouTube \cite{BGHS}.

Lannes theory and the cohomology of the classifying space for sphere bundles first arose in conversation between
Mike Hopkins and Vesna Stojanoska at Harvard University. We picked up these ideas at another NSF supported conference 
``Chromatic Homotopy  Theory: Journey to the Frontier" in 2018  at the University of Colorado. Thus we also heartily thank 
the Massachusetts Institute of Technology, which sponsored Stojanoska's visit to Cambridge MA, the Universities 
of Illinois and Colorado, the other organizers of the various conferences, and the National Science Foundation.

However, the use of Lannes Theory went from speculation to concrete reality only after a sequence of pivotal conversations
with Jesper Grodal at the Isaac Newton Institute in 2018 during the semester ``Homotopy Harnessing Higher Structures". 
Jesper has unique and deep knowledge of group actions, Lannes Theory, and classifying spaces -- all crucial to this part of 
the project -- and we are extremely grateful that he was there to share his expertise.

The Newton Institute program brought together a remarkable range of mathematicians in homotopy theory, and we had
very fruitful conversations also with Tobias Barthel, Clark Barwick, Anna Marie Bohmann, John Greenlees,
Hans-Werner Henn, and Michael A. Hill. Thanks are due to all of them, for the inspiring conversations as well as email 
correspondence well after the program ended. We added thanks as well to the organizers of the semester program, and to 
the Newton Institute and the staff there.

As a profession, we would be nowhere without the anonymous and careful work of diligent referees; we are
especially grateful to the main referee of this paper, who gave the paper a very close and careful reading, and who
had a number of insightful comments and questions.

The Linearization Hypothesis is not new with this paper, as we explained above, and we are grateful to Dustin Clausen for 
sharing and explaining his ideas.

Finally, when it became apparent that we would make heavy use of Steenrod squares in some of our arguments,
someone repeated a hoary old joke about ``Squeenrod stares" and Vesna Stojanoska suggested to Anna Marie
Bohmann, who has a long track-record of writing quality light verse, that she could use ``Steenrod stairs" as the
basis of a poem with a good many words beginning st- and sq-.  The result is below.  We are touched and grateful that Anna Marie has
granted us permission to reproduce it here.

\begin{quote}
\textbf{Derivation}
\medskip

\noindent
To a stolid squalid attic

\noindent
Full of strident squeaking squirrels

\noindent
Dr. Squeenrod staggered squiffy up the stairs
\medskip

\noindent
When a tumble most dramatic

\noindent
Put his letters in a whirl: 

\noindent
Squeenrod's stumble gave the storied Steenrod squares!
\medskip

\hfill \emph{Anna Marie Bohmann}
\end{quote}


\section{Basics on compact $p$-adic analytic groups}\label{sec:p-adicBasics}

While our main applications of the results here are for the Morava stabilizer group and the $K(n)$-local
category, the initial analysis  of the dualizing sphere applies to a much wider class of groups $\cG$. The key feature
is that $\cG$ has a finite index subgroup which is a pro-$p$ group of a particularly nice type. Thus we begin by
fixing a prime $p$, and reviewing some basic definitions.

The following material is from Sections I.4 and  II.8 of \cite{profinite}. If $G$ is a topological group, let $G^n \subseteq G$ be the
closed normal subgroup obtained by taking the closure of the subgroup generated by the $n$th powers. Similarly, if
$H \subseteq G$ is a closed subgroup, we let $[H,G] \subseteq G$ be the closure of the commutator subgroup of $H$
with $G$, and if $H$ and $K$ are two subgroups, let $HK$ be the closure of the product subgroup.

In all the statements below, there are slight modifications for the prime $2$. 

\begin{defn}\label{def:uniformly powerful}
A pro-$p$ group $G$ is  \emph{uniformly powerful}  if 
\begin{enumerate}[(a)]
\item $G/G^p$ (or $G/G^4$ if $p=2$) is abelian;
\item $G$ is topologically finitely generated;
\item The lower $p$-series 
\[
\xymatrix{
G = G_1 \supseteq G_2 \supseteq \ldots
\supseteq G_i  \supseteq G_{i+1} = G_i^p[G_i, G]  \supseteq \ldots
}
\]
has the properties that the $p$th power map induces an isomorphism
\[
\xymatrix{
G_i/G_{i+1} \ar[rr]^-{(-)^p}_-{\cong} & & G_{i+1}/G_{i+2}
}
\]
and it is exhaustive: $\bigcap_i G_i = \{e\}$.
\end{enumerate}
\end{defn}

If $G$ is a uniformly powerful pro-$p$-group and $G/G^p \cong (\ZZ/p)^d$, then
$G$ has {\it rank} $d$. Note that each of the subgroups $G_j$ is normal is $G$,
and hence normal in $G_i$ with $i \leq j$.

\begin{rem}\label{rem:quotients} If $G$ is uniformly powerful and topologically generated by $a_1, \ldots, a_d$ 
\begin{align*} 
G/G^p  &\cong \Z/p\{a_1, \ldots, a_d\}  \\
G_i/G_{i+j} &\cong \Z/p^j\{a_1^{p^{i-1}}, \ldots, a_d^{p^{i-1}}\} \qquad\mathrm{for}\ j\leq i.
\end{align*}
\end{rem}

In \S II.8.2. of \cite{profinite}, the authors give an intrinsic definition of a compact $p$-adic analytic group. For our purposes
it is enough to have the following characterization. See Corollary 8.34 of \cite{profinite}.

\begin{thm}\label{thm:p-adic-anal}
Let $\cG$ be a topological group. Then the following are equivalent:
\begin{enumerate}[(1)]
\item $\cG$ is a compact $p$-adic analytic group;
\item $\cG$ is a profinite group with an open subgroup which is a topologically finitely generated pro-$p$ group;
\item $\cG$ has an open normal subgroup of finite index which is a uniformly powerful pro-$p$ group of finite rank.
\end{enumerate}
\end{thm}

For such a group $\G$, the {\it rank} of $\cG$ is defined to be the rank of $G$ where $G$ is any  
open normal subgroup of finite index which is a uniformly powerful pro-$p$ group. The rank is
independent of this choice.

\begin{defn}\label{def:lie-algebra} Let $\cG$ be a compact $p$-adic analytic group and let $G \subseteq \cG$ be a
an open uniformly powerful finite index normal subgroup of rank $d$. Then, from \Cref{def:uniformly powerful} we have that
$(-)^p:G_i \to G_{i+1}$ is a continous
bijection of sets. This allows us to form the {\it Lie algebra} $\mathfrak{g}$ of  $\cG$.
This is the set $G$ with addition given by
\[
x+y = \lim_{n} (x^{p^n}y^{p^n})^{p^{-n}}
\]
and Lie bracket by
\[
[x,y] = \lim_{n} [x^{p^n},y^{p^n}]^{p^{-2n}} .
\]
The Lie algebra $\gg$ is independent of the choice of $G$; see the beginning of \S II.9.5 of  \cite{profinite}.
A chosen set of topological generators for $G$ defines
a continuous isomorphism $\Z_p^{d} \cong \mathfrak{g}$ of compact abelian groups. See Theorem 4.17 of
\cite{profinite}.

\end{defn}

\begin{rem}[{\bf The adjoint representation}]\label{rem:adjoint-action} Again,  let $G \subseteq \cG$ be a
an open uniformly powerful finite index normal subgroup of rank $d$. Since $G$ is normal in $\cG$, there is a 
conjugation action of $\cG$  on $G$.
This gives a $\ZZ_p$-linear action of $\cG$ on $\mathfrak{g}$ called the adjoint action, and $\gg$ with this action
is the {\it adjoint representation} of $\cG$. For later purposes we define
\[
\mathfrak{g}_i = p^{i-1}\mathfrak{g}.
\]
These form a nested sequence of sub-representations and the equality of sets $\gg = G$ induces
an isomorphism of abelian groups
\begin{align}\label{eq:gg-gamma}
\gg_i/\gg_{i+j}\cong G_i/G_{i+j},\qquad j \leq i,
\end{align}
when $i \geq 1$ if $p > 2$ and $i \geq 2$ for $p=2$. 
This becomes an isomorphism of $\cG$-modules if we act on $G_i/G_{i+j}$ by conjugation.
\end{rem}

\begin{exam}\label{exam:exam0} The simplest example is $\cG = \ZZ_p^d$. Then $G = \cG = \gg$ and $
\Gamma_i = (p^{i-1}\ZZ_p)^d$. 
\end{exam}

\subsection{Groups of units in normed division algebras}\label{rem:normed}
We discuss two instances of compact $p$-adic analytic groups; the general linear group $\Gl_n(\ZZ_p)$, 
in \Cref{exmp:exam1}, and the Morava stabilizer group, in \Cref{exmp:exam2}. In both cases the group $\cG$ arises
as the group of units in a sub-algebra of a complete normed $\QQ_p$-algebra $(A,|\!| \cdot |\!|)$. The norm extends the 
standard norm on $\QQ_p$; hence, $|\!|p^i|\!| = p^{-i}$. If $A$ is of finite rank over $\QQ_p$, we define
\[
A_i = \{x \in A\ |\ |\!|x|\!| \leq p^{-i}\ \}.
\]
Then $A_0 \subseteq A$ is a sub-$\ZZ_p$-algebra, each $A_i$ is an ideal in $A_0$, and $pA_i \subseteq A_{i+1}$. 
In our examples, this inclusion is an equality, so for convenience we assume this. The most basic example is
$A = \QQ_p$, $A_0 = \ZZ_p$ and $A_i = p^i\ZZ_p$.
 
We let $\cG = A^\times$ be the group of units in $A_0$ and define closed normal subgroups
\[
\Gamma_i = 1 + p^iA_0 \subseteq \cG.
\] 
In both our examples $\cG$ is a compact $p$-adic analytic group with a uniformly powerful subgroup 
\begin{align*}
\Gamma_1&=1+A_1=1+pA_0 \text{ \ \  if $p >2$ or} \\
\Gamma_2 &= 1+A_2=1+4A_0 \text{ \ \  if $p=2$.}
\end{align*} 
In fact, more is true. The following is an exercise in definitions. 

\begin{lem}\label{lem:missing-lem} Let $A$ be a complete normed $\Q_p$-algebra. With the notation as established above, if $p A_i = A_{i+1}$, we have the following conclusions.
\begin{enumerate} 
\item If $p > 2$, the subgroup $\Gamma_{i+1} \subseteq \Gamma_1$ is the $i$th term in 
the  lower $p$-series for $\Gamma_1$.
\item If $p =2$,  the subgroup $\Gamma_{i+2} \subseteq \Gamma_2$
is the $i$th term in the lower $p$-series of $\Gamma_2$.
\end{enumerate}
\end{lem}

It  follows from \Cref{lem:missing-lem}. that the rank of $\cG$ is the $\QQ_p$-rank of $A$. 

The standard exponential map
\begin{equation}\label{eq:exponential-one}
\exp:A_i \longr 1+A_i = \Gamma_i
\end{equation}
converges for $i\geq 1$ if $p>2$ and $i\geq 2$ if $p=2$. If we give $A_0$ the structure of a $\ZZ_p$-Lie algebra with 
bracket $[x,y] = xy-yx$, then this induces an isomorphism of Lie algebras
\begin{equation}\label{eq:exponential-two}
\xymatrix{
\exp:A_i = p^iA_0 \ar[r]^-\cong & \gg_i
}
\end{equation}
for $i\geq 1$ if $p>2$ and $i\geq 2$ if $p=2$. This is the content of Corollary II.7.14 of \cite{profinite}.

The group $\cG$ acts on $A_0$ by conjugation and the exponential
map also gives an isomorphism of $\cG$ representations
\begin{equation}\label{eq:exponential-three}
\xymatrix{
\exp:A_i/A_{i+j} \ar[r]^-\cong & (1+A_i)/(1+A_{i+j}) = \Gamma_i/\Gamma_{i+j}
}
\end{equation}
for $i\geq 1$ if $p>2$ (or $i\geq 2$ if $p=2$) and $j \leq i$. Note that since $A_iA_j \subseteq A_{i+j}$,
this exponential function has a very simple form: if $x \in A_i$, then
\[
\exp(x) = 1+x\qquad \mathrm{modulo}\  A_{2i}.
\]
In combination with \eqref{eq:gg-gamma} this establishes an isomorphism of representations
\[
A_i/A_{i+j} \cong \gg_i/\gg_{i+j}
\]
for $i\geq 1$ if $p>2$ (or $i\geq 2$ if $p=2$) and $j \leq i$.

\subsection{The main examples} We now give our two main examples.

\begin{exam}[{\bf The General Linear Group}]\label{exmp:exam1} Let
$\cG = \Gl_n(\ZZ_p)$ be the group of invertible $n\times n$ matrices with entries in the $p$-adic integers. Then $\Gl_n(\ZZ_p)$
exactly fits the rubric of \Cref{rem:normed}: we can let $A=M_n(\QQ_p)$, the $\QQ_p$-algebra
of $n\times n$ matrices. Then $A_0 = M_n(\ZZ_p)$ and $\Gl_n(\ZZ_p)=M_n(\ZZ_p)^\times$. Further,
$A_i =p^iM_n(\ZZ_p)$.

For the moment  assume $p>2$. Then $\Gl_n(\ZZ_p)$ is a compact $p$-adic analytic group with uniformly powerful 
subgroup $\Gamma_1=1+pM_n(\ZZ_p)$. The group $\Gamma_i$ is  the kernel of the map $\Gl_n(\ZZ_p) \to \Gl_n(\ZZ/p^i)$;
thus,
\[
\Gamma_i = 1 + p^iM_n(\ZZ_p).
\]
The group $\Gamma_1$ is of rank $n^2$. Furthermore, by \Cref{rem:quotients} we see that
if $j \leq i$ then $\Gamma_i/\Gamma_{i+j}$  is a free $\ZZ/p^j$-module of rank $n^2$.

If $\gg=\mathfrak{gl}_n$ is the Lie algebra of $\Gl_n(\ZZ_p)$ then \Cref{rem:normed} gives
an isomorphism of Lie algebras
\begin{equation}\label{eq:alg-to-liealg}
\xymatrix{
M_n(\ZZ_p) \ar[r]^-\cong & \mathfrak{gl}_n\ .
}
\end{equation}
If we let $\Gl_n(\ZZ_p)$ act by conjugation on $M_n(\ZZ_p)$, this gives an isomorphism of representations
from $M_n(\ZZ_p)$ to the adjoint representation. As in \Cref{rem:adjoint-action}, we can filter $\gg$  by powers 
of $p$ and the isomorphism of \eqref{eq:alg-to-liealg} descends to an isomorphism of $\Gl_n(\ZZ_p)$-modules
\begin{equation}\label{eq:subone}
p^{i}M_n(\ZZ_p)/p^{i+j}M_n(\ZZ_p) \cong \gg_i/\gg_{i+j} \cong \Gamma_i/\Gamma_{i+j},\qquad j \leq i.
\end{equation}
Here the action of $\Gl_n(\ZZ_p)$ on $ \Gamma_i/\Gamma_{i+j}$ is again by conjugation.

If $p=2$, $\Gl_n(\ZZ_2)$ is a compact $p$-adic analytic group with uniformly powerful subgroup $\Gamma_2 = 1 + 4M_n(\ZZ_2)$.
The rest of the remarks go through, with the evident changes on bounds. 

This example extends without much change to $\Gl_n(\WW)$ with $\WW=W(k)$ the $p$-typical Witt vectors
on a finite algebraic extension $k$ of $\FF_p$. The rank of $\Gl_n(\WW)$ is now $n^2[k:\FF_p]$.
\end{exam}

\begin{exam}[{\bf The Morava Stabilizer Group}]\label{exmp:exam2} This is our main example, and it has a number
of variants, all of which fit the setup of \Cref{rem:normed}.

To be concrete let $F=F(x,y)$ be the Honda formal group law of height $n \geq 1$ over $\FF_{p}$ and let
$\cO_n$ be the endomorphism ring of $F$ over $\FF_{p^n}$; thus, an element of $\cO_n$ is a power series
$\varphi(x) \in \FF_{p^n}[[x]]$ so that $\varphi(F(x,y))= F(\varphi(x),\varphi(y))$. Since $F$ is defined over $\FF_p$, the power series $S = x^p$ is an endomorphism in $\cO_n$, and there is an isomorphism of $\ZZ_p$-algebras
\[
\WW\langle S \rangle/(S^n-p) \to \cO_n
\]
where $\WW = W(\FF_{p^n})$ is the Witt vectors. The source is a non-commuting truncated polynomial ring:
if $a \in \WW$ then $Sa = a^\sigma S$ where we write $\sigma$ in exponent to indicate the action of the Frobenius
on $\WW$. 

Note that $\cO_n$ is free of rank $n^2$ over $\ZZ_p$. The $\QQ_p$-algebra $A = \QQ_p \otimes_{\ZZ_p} \cO_n$
is a complete normed $\QQ_p$-algebra with $|\!|S|\!| = p^{-1/n}$. Then $A_0 = \cO_n$ and $A_i = p^i\cO_n$. Define
\[
\SS_n = \cO_n^\times.
\]
This is the small Morava stabilizer group. The Galois group $\Gal(\FF_{p^n}/\FF_p)$ acts on $\cO_n$ through
the action on $\WW$; the full Morava stabilizer group is the semi-direct product
\[
\GG_n =  \SS_n \rtimes \Gal(\FF_{p^n}/\FF_p).
\]

Let $p > 2$. The groups $\SS_n$ and $\GG_n$ are compact $p$-adic analytic groups with uniformly powerful subgroup 
$\Gamma_1 = 1 +p\cO_n$. Then
\[
\Gamma_i = 1 +p^i\cO_n.
\]

\Cref{rem:normed} again applies and we have an isomorphism of Lie algebras
\begin{equation}\label{eq:alg-to-liealg--bis}
\xymatrix{
\cO_n \ar[r]^-\cong & \mathfrak{g}.
} 
\end{equation}
Furthermore $\SS_n$ acts by conjugation on $\cO_n$ and we have an isomorphism of representations
from $\cO_n$ to the adjoint representation. This isomorphism extends to one of $\GG_n$-modules.
As in \Cref{rem:adjoint-action}, we can filter $\gg$  by powers 
of $p$ and the isomorphism of \eqref{eq:alg-to-liealg} descends to an isomorphism of $\GG_n$-modules
\begin{equation}\label{eq:subone-bis}
p^{i}\cO_n/p^{i+j}\cO_n \cong \gg_i/\gg_{i+j} \cong \Gamma_i/\Gamma_{i+j},\qquad j \leq i.
\end{equation}
Here the action of $\GG_n$ on $ \Gamma_i/\Gamma_{i+j}$ is again by conjugation.

If $p=2$, the group $\GG_n$ is a compact $p$-adic analytic group with uniformly powerful subgroup $\Gamma_2 = 1 + 4\cO_n$.
The rest of the remarks go through, with the evident changes on bounds. 

Again this could be generalized to the example where $F$ is a formal group of height $n$ over an
algebraic extension of $\FF_p$. We can go further. The algebra $A = \QQ_p \otimes_{\ZZ_p} \cO_n$ is
a central division algebra of Hasse invariant $1/n$ over $\QQ_p$. The paradigm of \Cref{rem:normed} extends to the case
when $\cG$ is the group of units in the maximal order of any finite dimensional
central division algebra over $\QQ_p$.\end{exam}

\begin{notation}\label{rem:in-practice-gam} Driven by these examples, we will adopt the following
conventions for the rest of the paper. We will equip our 
$p$-adic analytic groups $\cG$ with a nested sequence of open normal subgroups of finite index
\[
\xymatrix{
\cG \supseteq \Gamma_1 \supseteq \Gamma_2 \supseteq \ldots \supseteq \Gamma_i  \supseteq
\Gamma_{i+1}  \supseteq \ldots\ 
}
\]
so that either $\Gamma_1$ (if $p  > 2$) or $\Gamma_2$ (if $p=2$) is a uniformly powerful
pro-$p$ group. If $p > 2$ and $i \geq 1$ then 
\begin{equation}\label{eq:lower-specific}
\Gamma_{i+1} = \Gamma_i^p[\Gamma_i, \Gamma_1].
\end{equation} 
If  $p=2$, then $\Gamma_{i+1} = \Gamma_i^p[\Gamma_i, \Gamma_2]$ for $i \geq 2$. 
Thus the remaining terms  for the sequence are the lower $p$-series for $\Gamma_1$ or $\Gamma_2$,
depending on the prime. See \Cref{lem:missing-lem}. 

Note that if $p=2$, there is no theoretical stipulation on $\Gamma_1$, although in practice it will be very concrete. 

The rank of $\cG$ will be the rank $\Gamma_i$ for $i$ large. 
\end{notation}

\section{Some group cohomology, with applications}\label{sec:cohanapp}

\begin{notation}
In this section as well as the remainder of the paper, we are concerned with continuous group cohomology. \emph{We will not decorate 
the cohomology notation to specify continuity; it is to be understood that all cohomology is continuous whenever that makes sense.} If 
$G$ is any compact $p$-analytic group (or simply a finite group), we
write $H_1(G,\FF_p)$ for $G/G_1$, where $G_1$ is the closure of $G^p[G,G]$. We have
that $H^1(G,\FF_p) \cong H_1(G,\FF_p)^\ast$, where $(-)^\ast$ denotes the $\FF_p$-linear dual.
 If we write $H^\ast(\hbox{--})$ we mean $H^\ast(\hbox{--},\FF_p)$; we use a similar
convention for homology.
\end{notation}

Following \cite{SymWei} we could define
\[
H_\ast (G,M) = \Tor_\ast^{\ZZ_p[[G]]}(\ZZ_p,M)
\]
where $\ZZ_p[[G]]$ is the completed group ring, $M$ lies in some category of continuous $G$-modules,
and $\Tor_\ast$ denotes the derived functors of a completed tensor product. The equation $H_1(G,\FF_p) = G/G_1$
is then a lemma, rather a definition. We won't use the greater generality, so we won't need to make these
ideas precise. 

Now let $\cG$ be a fixed compact $p$-adic analytic group with uniformly powerful subgroup $\Gamma_1$ if $p>2$ or
perhaps $\Gamma_2$ if $p=2$. See \Cref{rem:in-practice-gam}. We are interested in the continuous
cohomology  $H^\ast(\Gamma_i,M)$ for various $i$ and various coefficients $M$.

For any $i$ and all $j > i$, the quotient maps $\Gamma_i \to \Gamma_i/\Gamma_j \to \Gamma_i/\Gamma_{i+1}$
induce isomorphisms
\[
H_1(\Gamma_i) \cong H_1(\Gamma_i/\Gamma_j )\ \cong \Gamma_i/\Gamma_{i+1}.
\]
Let $V_i = H^1 (\Gamma_i) \cong  (\Gamma_i/\Gamma_{i+1})^\ast$. Then
\begin{equation}\label{eq:first-coho}
H^1(\Gamma_i/\Gamma_{j}) \cong H^1(\Gamma_i) = V_i, \qquad j > i. 
\end{equation}

If $V$ is a vector space over a field $k$, let $\Lambda(V)$ denote the graded exterior algebra on $V$,
and write $\Lambda^r(V) \subseteq \Lambda(V)$ for the homogeneous elements of degree $r$.  
Below we will also have a symmetric (polynomial) algebra $P(W)$ on a vector space $W$. 
The following basic result is the key to much of what follows.

\begin{thm}\label{thm:coh-gammai} (1) The natural inclusion $V_i = H^1(\Gamma_i) \hookrightarrow H^\ast(\Gamma_i)$
induces an isomorphism of graded commutative algebras
\[
\Lambda(V_i) \longr H^\ast(\Gamma_i).
\]

(2) The inclusion $\Gamma_{i+1} \to \Gamma_i$ induces the zero map on $H^k(-)$ for $k>0$.
\end{thm} 

\begin{proof} The first statement is proved  by combining Theorem 5.1.5 of \cite{SymWei} with Theorem 3.6 of 
\cite{profinite}. The second statement follows from the first and the fact that
\[
H_1(\Gamma_{i+1},\FF_p) \cong \Gamma_{i+1}/\Gamma_{i+2} \to \Gamma_i/\Gamma_{i+1}
\cong H_1(\Gamma_{i},\FF_p) 
\]
is the zero map.
\end{proof}

We now turn to the cohomology of $\Gamma_i/\Gamma_j$, with $j > i$, working our way to \Cref{lem:basic-coh}.
To establish some notation we contemplate the exact sequence of groups
\begin{align}\label{eq:a-simple-ses}
\xymatrix{
1 \ar[r] & \Gamma_{j-1}/\Gamma_j \ar[r]^f & \Gamma_i/\Gamma_j \ar[r]^-q & \Gamma_i/\Gamma_{j-1} \ar[r] & 1,
}
\end{align}
where $f$ is the inclusion and $q$ is the projection. Let $W_{j-1} \subseteq H^2(\Gamma_{j-1}/\Gamma_j)$
be the image of the Bockstein operation on $V_{j-1} \cong H^1(\Gamma_{j-1}/\Gamma_j)$. By \Cref{rem:quotients},
the group $ \Gamma_{j-1}/\Gamma_j$ is an elementary abelian $p$-group; therefore, there are isomorphisms
\begin{align}\label{eq:the-base-case}
\Lambda(V_{j-1}) \otimes P(W_{j-1}) &\cong H^\ast ( \Gamma_{j-1}/\Gamma_j ),\qquad p > 2;\\
P(V_{j-1}) &\cong H^\ast ( \Gamma_{j-1}/\Gamma_j ),\qquad p = 2.\nonumber
\end{align}
In the case $p=2$, $W_{j-1} \subseteq H^2 ( \Gamma_{j-1}/\Gamma_j )$ is the subvector space given by 
squares of the elements in degree $1$. 

In \cref{lem:basic-coh}, we extend \eqref{eq:the-base-case} to $H^\ast(\Gamma_i/\Gamma_j)$. The bounds on $i$ and $j$ in
this result are not optimal, but certainly are good enough for later applications and relieve us of the duty of making
special statements at $p=2$. Note that part (1) gives a splitting result, but some care is needed in interpreting that statement. We will add further comments below in \Cref{rem:basic-coh-1}.

\begin{thm}\label{lem:basic-coh} 
\begin{enumerate}
\item\label{thm:part1} For all $i \geq 3$ and $j > i+1$, there is an exact sequence of vector spaces
\[
\xymatrix{
0 \ar[r] & \Lambda^2 V_i \ar[r]^-{q^\ast} & H^2(\Gamma_i/\Gamma_j) \ar[r]^-{f^\ast} & W_{j-1} \ar[r] & 0
}
\]
and any splitting of this short exact sequence defines an isomorphism of algebras
\begin{align}\label{eq:split-temp-1}
\Lambda(V_i) \otimes P(W_{j-1}) \cong H^\ast (\Gamma_i/\Gamma_j). 
\end{align}

\item For all $i \geq 3$ and $j > i+1$, the image of $H^\ast(\Gamma_i/\Gamma_j) \to H^\ast(\Gamma_i/\Gamma_{j+1})$ is
$\Lambda(V_i)$.
\end{enumerate}
\end{thm}

\begin{proof} We will do the case when $p > 2$. There are evident modifications needed for $p=2$ to
account for the fact that $H^\ast(\Gamma_{j}/\Gamma_{j+1})$ is a polynomial algebra if $p=2$.

Part (1) is by induction on $j$, using that $\Gamma_i/\Gamma_{i+1}$ is elementary abelian
for the base case. Compare \eqref{eq:ss-setup}. The induction step is completed
using the Lyndon-Hochschild-Serre Spectral Sequence for the exact sequence of \eqref{eq:a-simple-ses}.
To compute the differentials we extend that exact sequence to a diagram
\begin{equation}\label{eq:ss-setup}
\xymatrix{
1 \ar[r] & \Gamma_{j-1}/\Gamma_j \ar[r] \ar[d]_=& \Gamma_{j-2}/\Gamma_j \ar[r] \ar[d]&
\Gamma_{j-2}/\Gamma_{j-1} \ar[d] \ar[r] & 1\\
1 \ar[r] & \Gamma_{j-1}/\Gamma_j \ar[r]^f & \Gamma_i/\Gamma_j \ar[r]^q & \Gamma_i/\Gamma_{j-1} \ar[r] & 1.
}
\end{equation}
Since $j > 4$, \Cref{rem:quotients} implies that the top exact sequence is non-canonically isomorphic to 
\[
\xymatrix{
0 \ar[r] & (\ZZ/p)^d \ar[r]^-{\times p} & (\ZZ/p^2)^d \ar[r] & (\ZZ/p)^d \ar[r] & 0.
}
\]

The diagram of \eqref{eq:ss-setup} gives a diagram of Lyndon-Hochschild-Serre Spectral Sequences
\[
\xymatrix{
\barE_2^{p,q} \cong H^p(\Gamma_i/\Gamma_{j-1},H^q(\Gamma_{j-1}/\Gamma_j))  \ar@{=>}[r]  \ar[d]
&H^p(\Gamma_i/\Gamma_{j})\ar[d]\\
E_2^{p,q} \cong H^p(\Gamma_{j-2}/\Gamma_{j-1},H^q(\Gamma_{j-1}/\Gamma_j))  \ar@{=>}[r]  &H^p(\Gamma_{j-2}/\Gamma_{j}).
}
\]
Note that we have decorated the top spectral sequence with an overbar to later help distinguish between the two. 

Since $\Gamma_1$ (or $\Gamma_2$ if $p=2$) is uniformly $p$-powerful, we have by \Cref{def:uniformly powerful} that 
$[\Gamma_i,\Gamma_j] \subseteq \Gamma_{i+j}$ for all $j \geq i \geq 1$ (or $i\geq 2$ if $p=2$). Thus in
both spectral sequences the action of the base group on $H^\ast(\Gamma_{j-1}/\Gamma_j)$ is trivial.

By \Cref{rem:quotients} the bottom of these spectral sequences is completely known: indeed, there are isomorphisms
\begin{align*}
E_2^{\ast,0} &\cong H^\ast(\Gamma_{j-2}/\Gamma_{j-1}) \cong \Lambda(V_{j-2}) \otimes P(W_{j-2}),\\
 E_2^{0,\ast} &\cong H^\ast(\Gamma_{j-1}/\Gamma_{j}) \cong \Lambda(V_{j-1}) \otimes P(W_{j-1}),
\end{align*}
and $E_2^{\ast,\ast} \cong  E_2^{\ast,0} \otimes  E_2^{0,\ast}$. The only non-zero differential is $d_2$. This is determined
by the isomorphism $d_2 \colon V_{j-1} \cong W_{j-2}\subseteq E_2^{2,0}$ and the multiplicative structure of the spectral sequence. We conclude
\[
E_\infty^{\ast,\ast}  \cong E_3^{\ast,\ast} \cong \Lambda(V_{j-2}) \otimes P(W_{j-1})
\]
with $V_{j-2} = E_\infty^{1,0}$ and $W_{j-1} = E_\infty^{0,2}$. 

We now turn to the top spectral spectral sequence. The base case of $j=i+2$ is covered by the above. By the induction hypothesis we have a short
exact sequence
\[
\xymatrix{
0 \ar[r] & \Lambda^2 V_i \ar[r]^-{q^\ast} & H^2(\Gamma_i/\Gamma_{j-1}) \ar[r]^-{f^\ast} & W_{j-2} \ar[r] & 0.
}
\]
We will make a useful choice of splitting of this exact sequence to obtain an isomorphism
$\Lambda(V_i) \otimes P(W_{j-2}) \cong H^\ast (\Gamma_i/\Gamma_{j-1})$ to aid in completing the induction
step.

We have isomorphisms
\begin{align*}
\barE_2^{\ast,0} &\cong H^\ast(\Gamma_{i}/\Gamma_{j-1}) \\
\barE_2^{0,\ast} &\cong H^\ast(\Gamma_{j-1}/\Gamma_{j}) \cong \Lambda(V_{j-1}) \otimes P(W_{j-1})
\end{align*}
and $\barE_2^{\ast,\ast} \cong  \barE_2^{\ast,0} \otimes  \barE_2^{0,\ast}$. 
By the induction hypothesis we have an isomorphism
\[
\barE_2^{\ast,0} \cong H^\ast(\Gamma_{i}/\Gamma_{j-1}) \cong \Lambda(V_i) \otimes P(W_{j-2}) \ .
\]
By the naturality of the spectral
sequences and the calculation just completed we have that the composition
\[
\xymatrix@C=15pt{
V_{j-1} = H^1(\Gamma_{j-1}/\Gamma_j) \cong \barE_2^{0,1} \ar[r]^-{d_2} &
\barE_2^{2,0} \cong H^2(\Gamma_i/\Gamma_{j-1}) \ar[r] & E_2^{2,0} \cong H^2(\Gamma_{j-2}/\Gamma_{j-1})
}
\]
has image exactly $W_{j-2}\subseteq E_2^{2,0}$. 
So $d_2$ maps $V_{j-1} = \barE^{0,1}_2$ isomorphically onto $W_{j-2} \subseteq \barE_2^{2,0}$. 
We then have
\[
 \overline E_3^{\ast,\ast} \cong \Lambda(V_{i}) \otimes P(W_{j-1})
\]
with $V_{i} = \barE_3^{1,0}$ and $W_{j-1} = \barE_3^{0,2}$. 

It remains to show the spectral sequence collapses at $\barE_3$. For this it is sufficient to show
$d_3$ vanishes on $W_{j-1} = \barE_2^{0,2}$. The target of this differential is $\Lambda^3(V_i) \cong
\barE_{3}^{3,0}$. However, we know from \Cref{thm:coh-gammai} that the composition
\[
\Lambda(V_i) \to H^\ast(\Gamma_i/\Gamma_j) \to H^\ast(\Gamma_i)
\]
is an isomorphism, so the map induced by the edge homomorphism
\[
\Lambda^3(V_i) \to \barE_\infty^{3,0} \to H^3(\Gamma_i/\Gamma_j)
\]
must be an injection. Thus $d_3(W_{j-1})=0$ as needed. 

For part (2) we examine the diagram from part (1)
\[
\xymatrix{
0 \ar[r] & \Lambda^2 V_i \ar[r]^-{q^\ast} \ar[d]_=& H^2(\Gamma_i/\Gamma_j) \ar[r]^-{f^\ast} \ar[d] &
W_{j-1} \ar[r] \ar[d]& 0\\
0 \ar[r] & \Lambda^2 V_i\ar[r]^-{q^\ast} & H^2(\Gamma_i/\Gamma_{j+1}) \ar[r]^-{f^\ast} & W_{j} \ar[r] & 0
}
\]
Since the map $H^2(\Gamma_{j-1}/\Gamma_{j}) \to H^2(\Gamma_{j}/\Gamma_{j+1})$ is zero, the map
$W_{j-1} \to W_j$ is zero. It follows that the
kernel of the map $H^2(\Gamma_i/\Gamma_j) \to H^2(\Gamma_i/\Gamma_{j+1})$ is isomorphic
to $W_{j-1}$ and, in fact, this gives a splitting of the top $f^\ast$. The result follows.
\end{proof}

\begin{rem}\label{rem:basic-coh-1} Part (1) of \Cref{lem:basic-coh} says that we can choose an
isomorphism of graded algebras \eqref{eq:split-temp-1}
\[
\Lambda(V_i) \otimes P(W_{j-1}) \cong H^\ast(\Gamma_i/\Gamma_j).
\]
If $j \leq 2i$ then $\Gamma_i/\Gamma_{j} \cong (\ZZ/p^{j-i})^d$ and we can choose this isomorphism
to be an isomorphism of unstable algebras over the Steenrod algebra, and we can even add the further stipulation that
the appropriate higher Bockstein defines an isomorphism from $V_i$ to $W_{j-1}$. If $j > 2i$, it is far less 
clear how the higher Bocksteins behave, and it is no longer possible to be so explicit.

For example,
if $\cG=\SS_n$ is the Morava stabilizer group of \Cref{exmp:exam2} then we know by work
of Lazard and Morava (see Remark 2.2.5 of \cite{MoravaAnnals}) that there is an isomorphism
\[
\Lambda_{\QQ_p}(x_1,x_3,\cdots,x_{2n-1}) \cong H^\ast(\SS_n,\ZZ_p) \otimes_{\ZZ_p} \QQ_p
\]
where $x_{2k-1}$ is in degree $2k-1$. This implies that the structure
of the higher order Bocksteins must be rich. 

As a further remark, note that if we choose a splitting as in \eqref{eq:split-temp-1}, then the map
\[
\Lambda(V_i) \otimes P(W_{j-1}) \cong H^\ast(\Gamma_i/\Gamma_j) \to H^\ast \Gamma_i
\]
gives an isomorphism $\Lambda(V_i) \cong H^\ast \Gamma_i$.  However, we have not proved
that $W_{j-1}$ maps to zero. 
\end{rem} 

If $G$ is a profinite group, it is not immediately obvious how to define its classifying space.  In particular, 
there is a long history, going back to Artin and Mazur, of regarding the classifying space as a pro-object.
We will need nothing that sophisticated. For our purposes, the following will suffice. 

\begin{defn}\label{def:class-pro-finite} Let $G = \lim_j G_j$ be a profinite group. Then the classifying space $BG$ is defined by
\[
BG = \holim_j BG_j.
\]
\end{defn}

In particular, if $\cG$ is our compact $p$-adic analytic group with uniformly powerful subgroup $\Gamma_1$  we have
$B\Gamma_i = \holim_j B(\Gamma_i/\Gamma_{i+j})$. It is immediate that 
\[
\pi_n B\Gamma_i \cong 
\begin{cases}
\Gamma_i, & n = 1;\\
0, & n \ne 1.
\end{cases}
\]
We remark, however, that this isomorphism does not recover the topology on $\Gamma_i$.  

The next result now follows from \Cref{thm:coh-gammai}, \Cref{lem:basic-coh}, \Cref{prop:inv-to-homology} 
below,
and the fact that if  $G$ is a finite $p$-group, then $\Sigma^\infty BG$ is already $p$-complete. 
It is slightly surprising, since suspension and limits to do not formally commute.

\begin{prop}\label{prop:it-doesn't-matter} The natural map
\[
\Sigma^\infty_+ B\Gamma_i \to \holim_j  \Sigma^\infty_+ B(\Gamma_i/\Gamma_{i+j})
\]
is an equivalence after $p$-completion. Furthermore, we have isomorphisms
\[
\Lambda(V_i) \cong \colim_j H^\ast(\Gamma_i/\Gamma_{i+j}) \cong
 H^\ast(\Gamma_i) \cong H^\ast(B\Gamma_i).
\]
\end{prop}

The proof of \Cref{prop:it-doesn't-matter} is completed with this general result.

\begin{prop}\label{prop:inv-to-homology} Let X be a bounded below spectrum and let
\[
\xymatrix@C=15pt{
X \rto & \cdots \rto & X_k \rto & X_{k-1} \rto & \cdots \rto & X_1
}
\]
be a tower of bounded below spectra under $X$ with the property that $H_\ast X \to H_\ast X_k$ induces an isomorphism
\[
H_\ast X \cong \mathrm{Image}\{H_\ast X_{k+1} \to H_\ast X_k\}.
\]
Then $X_p \to \holim_k (X_k)_p^\wedge$ is an equivalence and the induced maps
\begin{align*}
H_\ast X\ \longr& \lim_k H_\ast X_k\\
\mathop{\colim}_k H^\ast X_k\ &\longr H^\ast X
\end{align*}
are isomorphisms. 
\end{prop}

\begin{proof} If $Z$ is any spectrum, write $H\FF_p^{\bullet+1} \smsh Z$ for the 
standard cosimplicial cobar complex defining the Adams Spectral Sequence. This spectral sequence reads
\[
\Ext_{A_\ast}^s(\Sigma^t\FF_p,H_\ast Z) \cong \pi^s\pi_t H\FF_p^{\bullet+1} \smsh Z 
\Longrightarrow \pi_{t-s} \holim_\Delta H\FF_p^{\bullet+1} \smsh Z. 
\]
Here $A_\ast$ is the dual Steenrod algebra. If $Z$ is bounded below then the map
\[
Z \longr \holim_\Delta H\FF_p^{\bullet+1} \smsh Z
\]
is $p$-completion. By construction, there is an isomorphism, natural in $Z$,
\[
\pi_\ast H\FF_p^{s+1} \smsh Z \cong A_\ast ^{\otimes s} \otimes H_\ast Z.
\]

Now turn to the tower under $X$. The hypothesis on the homology of the tower implies that for all $s$ the
natural map
\[
\pi_\ast H\FF_p^{s+1} \smsh X \to \pi_\ast \mathop{\holim}_k H\FF_p^{s+1} \smsh X_k
\cong \lim_k \pi_\ast H\FF_p^{s+1} \smsh X_k
\]
is an isomorphism. Thus we have equivalences
\begin{align*}
X_p \simeq \holim_\Delta H\FF_p^{\bullet+1} \smsh X &\simeq \holim_\Delta   \holim_k H\FF_p^{s+1} \smsh X_k\\
&\simeq \holim_k \holim_\Delta  H\FF_p^{s+1} \smsh X_k\\
& \simeq \holim_k (X_k)_p^\wedge. \qedhere
\end{align*}
\end{proof}

\section{Duality and Frobenius Reciprocity}\label{sec:dualityandfrob}

This section collects a great deal of relatively standard material about group cohomology, with the wrinkle that we need
these results for compact $p$-adic analytic groups. Much of this can be collected from \cite{CohoGal} and \cite{SymWei}.

Define the completed group ring of $\cG$ to be
\begin{equation}\label{eq:grp-ring-def}
\ZZ_p[[\cG]] = \lim_{n,i} \ZZ/{p^n}[[G_i]]
\end{equation}
where $\cG \cong \lim G_i$ is any presentation of $\cG$ as an inverse limit of finite groups.

\begin{notation}\label{not:keep-it-str} In this section $\cG$ will be a compact $p$-adic analytic group unless otherwise
stated. All modules we consider will be continuous $\ZZ_p$-modules and most examples
will be either $p$-profinite or discrete $p$-torsion. The following conventions and definitions hold or continue to hold. 
\begin{enumerate}

\item $\Hom$ always means continuous homomorphisms and tensor products between profinite modules
are always completed, and taken over $\Z_p$ if unadorned.

\item  If $M$ is
a continuous right $\cG$-module and $N$ is a continuous left $\cG$-module we will also write
$M \otimes_\cG N$ for $M \otimes_{\ZZ_p[[\cG]]} N$. and we will write $\Hom_{\cG}(M,N)$ for $\Hom_{\ZZ_p[[\cG]]}(M,N)$.
We will similarly abbreviate the $\Tor$ and $\Ext$ decorations.

\item If $M$ and $N$ are two continuous left $\cG$-modules we give $M \otimes N$ the structure of a left $\cG$-module
with the diagonal $\cG$-action: $g(x \otimes y) = gx \otimes gy$. 

\item Define $\Hom^\Delta(M,N)$ to be the  abelian group of continuous homomorphisms with the left
$\cG$-action given $(g\varphi)(a) = g\varphi(g^{-1}a)$.
\end{enumerate}

The action on $\Hom^\Delta$ has the twin features that $\Hom_\cG(A,B) = \Hom^\Delta(A,B)^\cG$ and the evaluation map
\[
\Hom^\Delta(A,B) \otimes A \longr B
\]
sending $\varphi \otimes a \to \varphi(a)$ is a morphism of left $\cG$-modules. 
\end{notation}

\subsection{Group cohomology basics}

Let $\cG$ be a compact $p$-adic analytic group and $M$ a continuous $\cG$-module. 
Following Serre \cite{CohoGal}, we let $C^s(\cG,M)$ be the set of continuous maps
\[
\phi: \cG^s \longr M.
\]
If $s=0$, then $C^s(\cG,M) = M$. 
We will also call these continuous cochains. The collection $C^\bullet(\cG,M)$ is a cosimplicial abelian group with 
coface operators given by
\[
d^i\phi(x_1,\ldots,x_{s+1}) = 
\begin{cases}
x_1\phi(x_1,\ldots,x_{s+1}),&i=0;\\
\phi(x_0,\ldots,x_{i}x_{i+1},\ldots, x_{s+1}),&1 \leq i \leq s;\\
\phi(x_1,\ldots,x_{s}),&i=s+1.
\end{cases}
\]
Codegenerarcy operators are obtained by inserting identities; we won't need them. We then have
\[
H^s(\cG,M) = H^s(C^\bullet(\cG,M),\partial)
\]
where $\partial = \sum_{i=0}^s (-1)^i d^i$. As explained in  Sections 3.1 and 3.2 of \cite{SymWei},
these are the right derived functors of $H^0(\cG,M) = M^{\cG}$; indeed, there is a natural isomorphism
\[
H^s(\cG,M) \cong \Ext^s_\cG(\ZZ_p,M)
\]
and the cochain complex above is the standard cobar construction for calculating the $\Ext$-groups. We also define
\begin{equation}\label{eq:hom-defined-g}
H_s(\cG,M) = \Tor^s_\cG(\ZZ_p,M).
\end{equation}

Let $f:\cG_1 \to \cG_2$ be a continuous map of compact $p$-adic analytic groups and $M$ a $\cG_2$-module. Then $M$ becomes
a $\cG_1$-module by restriction: if $x \in \cG_1$ and $a \in M$, then $x\cdot a = f(x)a$. Write this module as
$f_\ast M$. Then we get a map
\[
f^\ast:C^s(\cG_2,M) \to C^s(\cG_1,f_\ast M)
\]
given by sending $\phi:\cG_2^{s} \to M$ to a cochain $f^\ast\phi:\cG_1^{s} \to f_\ast M$ with
\[
(f^\ast\phi)(x_1,\cdots,x_s) = \phi(f(x_0),\ldots,f(x_s)). 
\]
The new action on $M$ is needed so that $f^\ast(d^0\phi)=d^0f^\ast\phi$. We then get a map
\[
f^\ast \colon H^\ast(\cG_2, M) \to H^\ast(\cG_1, f_\ast M).
\]

\begin{exam}\label{ex:conj-action}
For example, suppose $f = c_g:\cG \to \cG$ is given by conjugation and $M$ is a left $\cG$-module; then we write $\gMa$ for
$f_\ast M$ and we get a map 
\[c_g^\ast:H^\ast(\cG,M) \to H^\ast(\cG,\gMa).\] The action on $\gMa$ is given
by $x\cdot a = (gxg^{-1})a$.

More generally, let $\Gamma \subseteq \cG$ be a closed subgroup. If $g\in \cG$ let $ \gGa$ denote the conjugate 
subgroup $g^{-1}\Gamma g$ and write $c_g(-) = g(-)g^{-1}:\gGa \to \Gamma$ for the conjugation homomorphism. The
conventions are chosen so that $c_{gh} = c_g \circ c_h$. If $M$ is a $\cG$-module, we then get a homomorphism
\[
c_g^\ast:H^\ast(\Gamma,M) \to H^\ast(\gGa,\gMa).
\]
\end{exam}

\subsection{Induced modules and transfer} 
Let $\cG$ be a compact $p$-adic analytic group and let $\Gamma \subseteq \cG$ be a closed subgroup.
If $M$ is a continuous $\Gamma$-module, define the {\it coinduced} $\cG$-module to be the module
\[
M_\Gamma^\cG = \map_\Gamma(\cG,M)
\]
of {\it continuous} $\Gamma$-equivariant functions $\varphi: \cG \to M$ with $\cG$ action given by the formula
\[
(g\varphi)(x) = \varphi(xg).
\]
The functor $M \mapsto M_\Gamma^\cG$ is right adjoint to the forgetful functor to $\Gamma$-modules. The
following result is {\it Schapiro's Lemma}.

\begin{prop}\label{prop:schapiro-lemma} Let $M$ be a $\Gamma$-module which is either discrete $p$-torsion or
profinite. Then we have a natural isomorphism
\[
H^\ast(\cG,M_\Gamma^\cG) \cong H^\ast(\Gamma,M). 
\]
\end{prop}

\begin{proof} The discrete case is covered in \S I.2.5 of \cite{CohoGal}. The profinite case follows from the discrete case and the
fact that if $M = \lim M_i$, then
\[
\map_\Gamma(\cG,M) \cong \lim \map_\Gamma(\cG,M_i).\qedhere
\]
\end{proof}

If $M$ is a continuous $\cG$-module, then there is a map
\[
\eta_M:M \to M_\Gamma^\cG = \map_\Gamma(\cG,M)
\]
adjoint to the identity $M \to M$ regarded as a $\Gamma$-module map. The induced map
\[
\res = \res_\Gamma^\cG:H^\ast(\cG,M) \longr H^\ast(\cG,M_\Gamma^\cG) \cong H^\ast(\Gamma,M)
\]
is the restriction map. 

Now we specialize to the case where $\Gamma \subseteq \cG$ is open and, hence, closed and of finite index. 
There is a $\cG$-equivariant averaging map
\begin{align*}
M_\Gamma^\cG =& \map_\Gamma(\cG,M) \longr \map_{\cG}(\cG,M) \cong M\\
\varphi(-) &\longmapsto \sum_{g\Gamma \in \cG/\Gamma} g\varphi(g^{-1}-).
\end{align*}
The induced map
\begin{equation}\label{eq:coh-trans-def}
\tr = \tr_{\Gamma}^\cG: H^\ast(\Gamma,M) \cong H^\ast(\cG,M_\Gamma^\cG) \longr H^\ast(\cG,M)
\end{equation}
is the {\it transfer} map. The next result follows from the definitions. 

\begin{lem}\label{lem:res-trans-natural} We have the following naturality statements for restriction and transfer. 

(1) Let $f:\cG_1 \to \cG_2$ be a continuous homomorphism of compact $p$-adic analytic groups. Suppose
$\Gamma_i \subseteq \cG_i$ is a closed subgroups and suppose $f(\Gamma_1) \subseteq \Gamma_2$. Let $M$ be a
$\cG_2$-module. Then the following diagram commutes.
\[
\xymatrix{
H^\ast(\cG_2,M) \ar[r]^-{\res} \ar[d]_{f^\ast}& H^\ast(\Gamma_2,M) \ar[d]^{f^\ast}\\
H^\ast(\cG_1,f_\ast M) \ar[r]_-{\res} & H^\ast(\Gamma_1,f_\ast M). 
}
\]

(2) Suppose further that $\Gamma_2$ is of finite index in $\cG_2$ and $f$ induces an isomorphism
$\cG_1/\Gamma_1 \cong \cG_2/\Gamma_2$. Then the following diagram commutes.
\[
\xymatrix{
H^\ast(\Gamma_2,M) \ar[r]^{\tr} \ar[d]_{f^\ast}& H^\ast(\cG_2,M) \ar[d]^{f^\ast}\\
H^\ast(\Gamma_1,f_\ast M) \ar[r]_{\tr} & H^\ast(\cG_1,f_\ast M). 
}
\]
\end{lem}

Still assuming $\Gamma \subseteq \cG$ is open, we have that  the functor $M \mapsto M_\Gamma^\cG$ is also isomorphic
to the left adjoint to the forgetful functor. The natural map of $\Gamma$-modules 
\[
M_\Gamma^\cG = \map_\Gamma^c(\cG,M) \longr M
\]
sending $\varphi$ to $\varphi(e)$ has a natural $\Gamma$-module splitting
that extends to an isomorphism of $\cG$-modules
\begin{equation}\label{eq:induced-vs-coninduced}
\ZZ_p[[\cG]] \otimes_{\ZZ_p[[\Gamma]]} M \cong M_\Gamma^\cG.
\end{equation}
Thus we equally call $M_\Gamma^\cG$ the {\it induced} $\cG$-module.

We have a commutative diagram of $\cG$-modules
\begin{equation}\label{eq:trans-commutes}
\xymatrix{
M \ar[r]^-\psi \ar[d]_= & \ZZ_p[[\cG]] \otimes_{\Gamma} M \ar[r]^-m \ar[d]^\cong & M\ar[d]^=\\
M \ar[r]_-{\eta_M} &  M_\Gamma^\cG \ar[r]_-{\tr} & M
}
\end{equation}
where $\psi(a) = \sum_{g\Gamma \in \cG/\Gamma} g \otimes g^{-1}a$ and $m$ is induced by the action of $\cG$ on
$M$. The map $\psi$ induces the transfer map in homology
\begin{align*}
\tr_\ast \colon H_s(\cG,M)\ \longr H_s(\Gamma,M). 
\end{align*}
More generally, if $N$ is a right $\cG$-module we can use the isomorphism 
\[
N \otimes_{\ZZ_p[[\cG]]} \ZZ_p[[\cG]] \otimes_{\ZZ_p[[\Gamma]]} M \cong N \otimes_{\ZZ_p[[\Gamma]]} M
\]
and the map $\psi$ to obtain a transfer map 
\begin{align*}
\tr_\ast \colon \Tor_s^\cG(N,M) \longr \Tor_s^\Gamma(N,M). 
\end{align*}
Likewise, $m$ induces the map $\res_\ast: \Tor_s^\Gamma(N,M) \longr \Tor_s^\cG(N,M)$ arising from the ring
homomorphism $\ZZ_p[[\Gamma]] \to \ZZ_p[[\cG]]$. There is a naturality statement analogous to \Cref{lem:res-trans-natural}. 

\subsection{Frobenius Reciprocity and cohomology}  Let $M$ and $N$ be two continuous $\cG$-modules and let $M \otimes N$ be
their tensor product with the diagonal action. Then we have a cup product pairing
\[
H^m(\cG,M) \otimes H^n(\cG,N) \longr H^{m+n}(\cG,M \otimes N).
\]
If $\phi \in C^m(\cG,M)$ and $\psi\in C^n(\cG,N)$, then
\[
(\phi\cdot \psi)(x_1,\dots,x_{n+m}) = \phi(x_1,\ldots,x_m) \otimes (x_1\cdots x_m)\psi(x_{m+1},\ldots,x_n).
\]
The following formulas (1) and (2) are in \S I.2.6a of \cite{CohoGal}. Following Serre we leave them
as an exercise. 

\begin{enumerate} 
\item Restriction preserves products: for all $x \in H^m(\cG,M)$ and all $y \in H^n(\cG,N)$
\[
\res^\cG_\Gamma(x\cdot y) = \res^\cG_\Gamma(x)\cdot \res^\cG_\Gamma(y);
\]

\item Frobenius Reciprocity holds: for $x \in H^m(\Gamma,M)$ and $y \in H^n(\cG,N)$
\[
\tr^\cG_\Gamma(x)\cdot y = \tr^\cG_\Gamma(x\cdot\res^\cG_\Gamma(y))
\]
and for all $x \in H^m(\cG,M)$ and $y \in H^n(\Gamma,N)$
\[
x\cdot \tr^\cG_\Gamma(y) = \tr^\cG_\Gamma(\res^\cG_\Gamma(x)\cdot y).
\]

\item Let $\Gamma \subseteq \cG$ be a closed subgroup and $g \in \cG$. Then for all $x \in H^m(\Gamma,M)$ and
$y \in H^n(\Gamma,N)$
\[
c_g^\ast(x \cdot y) = c_g^\ast(x) \cdot c_g^\ast(y) \in H^{m+n}(\gGa,\gMa \otimes \prescript{g}{}N).
\]
Note $\gMa \otimes \prescript{g}{}N \cong \prescript{g}{}(M \otimes N)$. 
\end{enumerate}

\subsection{Conjugation actions} 

As above, let $c_g(x) = g^{-1}xg$ denote conjugation by $g \in \cG$.  Let $M$ be a $\cG$-bimodule;
that is, $M$ is both a left and right $\cG$-module and the two actions  commute. Our main example is
$M = \ZZ_p[[\cG]]$ of \eqref{eq:grp-ring-def}. As $M$ is a left $\cG$-module, we have a map
$c_g^\ast:H^\ast(\cG,M) \to H^\ast(\cG,\gMa)$. But the right multiplication of $\cG$ on $M$ defines a map
$r_g^\ast :H^\ast(\cG,M) \to H^\ast(\cG,M)$. We relate these two maps. 

If  $\phi:\cG^{s} \to M$ is a continuous cochain, define
\[
\chi_g(\phi):\cG^s \to M
\]
by
\begin{equation}\label{eq:conj-in-cochains}
\chi_g(\phi)(x_1,\ldots,x_s) = g^{-1}\phi(c_g(x_1),\ldots,c_g(x_n))g.
\end{equation}
We check that this induces a map on cochain complexes and hence a map
\[
\chi_g^\ast:H^\ast(\cG,M) \to H^\ast(\cG,M).
\]
Note we have $\chi_{gh}^\ast = \chi_h^\ast \circ \chi_g^\ast$, so the assignment $g \mapsto \chi_g^\ast$ defines a right
action on $H^\ast(\cG,M)$. The following result says the action is very simple.

\begin{prop}\label{prop:conjugation-is-right} The map $\chi_g^\ast : H^\ast(\cG,M) \to H^\ast(\cG,M)$ is given
by the right action of $g\in \cG$ on $M$.
\end{prop}

\begin{proof} The right action of $\cG$ on $H^\ast(\cG,M)$ is given by multiplication of $\cG$ on the right on
$M$; on cochains we have
\[
r_g(\phi)(x_1,\ldots,x_s) = \phi(x_1,\ldots,x_s)g.
\]

We construct a cochain homotopy $T_g:C^{s+1}(\cG,M) \to C^{s}(\cG,M)$
with
\[
\partial T_g + T_g \partial = \chi_g - r_g.
\]
The result will follow.

Define $T_g^i:C^{s+1}(\cG,\ZZ_p[[\cG]]) \to C^{s}(\cG,\ZZ_p[[\cG]])$, $0 \leq i \leq s$,
by
\[
T_g^i(\phi)(x_1,\ldots,x_s) = \phi(x_1,\ldots,x_{i},g^{-1},c_g(x_{i+1}),\ldots,c_g(x_s))g.
\]
Thus $g^{-1}$ is in the $(i+1)$st-slot. Note
\begin{align*}
T_g^0(d^0\phi)(x_1,\ldots,x_s) &= \chi_g(\phi)(x_1,\ldots,x_s)
\end{align*}
and
\begin{align*}
T_g^s(d^{s+1}\phi)(x_1,\ldots,x_s) &= r_g(\phi)(x_1,\ldots,x_s).
\end{align*}
and also
\[
T_g^j d^i = 
\begin{cases}
d^iT_g^{j-1},& i < j;\\
T_g^{i-1}d^i,& i=j \ne 0;\\
d^{i-1}T_g^j,& i > j+1.
\end{cases} 
\]
Then we set
\[
T_g = \sum_{i=0}^s (-1)^i T_g^i:C^{s+1}(\cG,M) \to C^{s}(\cG,M).\qedhere
\]
\end{proof} 

\begin{rem} In particular if $M$ is a trivial module, we can make it a bimodule by giving it the trivial right module structure.
Then we get the standard argument that conjugation of $g$ induces  the trivial action on $H^\ast(\cG,M)$.
\end{rem}

If $M$ is bimodule then we get an isomorphism $M \cong \gMa$ of left $\cG$-modules that sends $a$ to $gag^{-1}$. 
This gives a commutative diagram
\begin{equation}\label{eq:all-conj-all-time}
\xymatrix@R=10pt{
&H^\ast(\cG,M) \ar[dd]^\cong\\
H^\ast(\cG, M) \ar[ur]^{\chi_g^\ast=r_g^\ast}  \ar[dr]_{c_g^\ast}\\
&H^\ast(\cG,\gMa).
}
\end{equation}
Thus all possible ways of defining the conjugation action determine each other. 

\subsection{Dualizing modules and duality} 

We review the concepts of Poincar\'e duality, relying especially on Section 4.4. of \cite{SymWei}. Our emphasis
will be on $p$-adic analytic groups. 

\begin{defn}\label{def:PD-groups} Let $\cG$ be a profinite group. Then we have the following concepts.

\begin{enumerate}

\item The group $\cG$ has {\it finite cohomological $p$-dimension} $cd_p(\cG)$ if there is some integer $m$ so that for
all $p$-torsion $\cG$-modules $M$ and all $s > m$
\[
H^s(\cG,M) = 0.
\]
Then $cd_p(\cG) = n$ if $n$ is minimal among the  integers $m$ for which this condition holds. 

\item The group $\cG$ is of {\it type $p$-$\FP$} if the trivial $\cG$-module $\ZZ_p$ has a finite resolutions
\[
0 \longr P_k \longr \cdots \longr P_1 \longr P_0 \longr \ZZ_p
\]
where each $P_i$ is a finite direct sum of copies of $\ZZ_p[[\cG]]$. Call $k$ the {\it length} of the resolution.

\item The group $\cG$ is a {\it Poincar\'e duality group} of dimension $n$ if $cd_p(\cG) = n$, $\cG$ is of type $p$-$\FP$ and 
\[
H^s(\cG,\ZZ_p[[\cG]]) \cong
\begin{cases}
\ZZ_p,&s=n;\\
0,&s \ne n.
\end{cases}
\]
\end{enumerate} 
\end{defn}

The following is a consequence of Proposition 4.1.1 of \cite{SymWei}.
\begin{lem}\label{lem:length-precise} Let $\cG$ be a Poincar\'e duality group of dimension $n$. Then
\begin{enumerate}

\item $H^s(\cG,M) = 0$ for $s > n$ for all $p$-profinite $\cG$-modules $M$; and

\item the trivial $\cG$-module $\ZZ_p$ has a resolution as in \Cref{def:PD-groups}.2 of exactly length $n$. 

\end{enumerate}
\end{lem}

Let $\cG$ be a compact $p$-adic analytic group
of rank $d$ with an exhaustive system
\[
\cdots \subseteq \Gamma_{i+1} \subseteq \Gamma_i \subseteq \cdots \subseteq \Gamma_1 
\subseteq \cG
\]
of uniformly powerful open normal subgroups. We then have.

\begin{prop}\label{prop:when-g-pd} For any compact $p$-adic analytic group $\cG$ of rank $d$ we have
\[
H^s(\cG,\ZZ_p[[\cG]]) \cong
\begin{cases}
\ZZ_p,&s=d;\\
0,&s \ne d.
\end{cases}
\]
Furthermore, for all $i \geq 1$ (or $i \geq 2$ if $p=2$), the uniformly powerful open normal subgroup $\Gamma_i$
is a Poincar\'e duality group of dimension $d$. 
\end{prop} 

\begin{proof} It is a theorem, going back to Serre (Proposition I.4.5 of \cite{CohoGal}) and Lazard
(Th\'eor\`eme V. 2.5.8 of \cite{lazard})  
that $\Gamma_i$ is a Poincar\'e duality group
of dimension $d$. The statement about $H^s(\cG,\ZZ_p[[\cG]])$ follows from Schapiro's Lemma \Cref{prop:schapiro-lemma}
and that fact if $\Gamma \subseteq \cG$ is of finite index then $\ZZ_p[[\Gamma]]_{\Gamma}^{\cG} \cong \ZZ_p[[\cG]]$. 
See \eqref{eq:induced-vs-coninduced}. 
\end{proof}

The group ring $\ZZ_p[[\cG]]$ is a $\cG$-bimodule, using the left
and right actions of $\cG$ on itself. 

\begin{rem}\label{rem:when-finite-coh-bis} Our main examples of compact $p$-adic analytic groups, such as $\Gl_n(\ZZ_p)$ or
$\GG_n$, need not be of finite cohomological dimension. 
\end{rem}

In calculating $H^s(\cG,\ZZ_p[[\cG]])$ we use the left action of $\cG$ on the group ring.
This continuous cohomology group retains a residual action on the right by $\cG$.

\begin{defn}\label{defn:dualizng-mod} Let $\cG$ be a compact $p$-adic analytic group of rank $d$. Then the {\it compact
dualizing module} $D_p(\cG)$ of $\cG$ is the right-$\cG$ module
\[
D_p(\cG) = H^d(\cG,\ZZ_p[[\cG]]).
\]
\end{defn}

The following version of Poincar\'e Duality can be found as Proposition 4.5.1 of \cite{SymWei}. 

\begin{thm}\label{thm:pd-forg} Let $\cG$ be a compact $p$-adic analytic group of rank $d$ 
with dualizing module $D_p(\cG)$. Suppose further that $\cG$ is a Poincar\'e duality group.
Then there is a natural homomorphism
\begin{equation*}\label{eq:pdf-two}
\Tor_s^\cG(D_p(\cG),M) \longr  H^{d-s}(\cG,M).
\end{equation*}
This homomorphism is an isomorphism if $M$ is either a $p$-profinite module or a discrete $p$-torsion module.
\end{thm}

\begin{proof} The paper \cite{SymWei} covers a great deal of ground and it takes a while to pull together
the proof of \Cref{thm:pd-forg} from that source. Here is a summary.

For any continuous left $\cG$-module $P$ let $\Hom_\cG(P,\ZZ_p[[\cG]])$ be the module of left $\cG$-module maps.
This is a right $\cG$-module using the right $\cG$-module structure on $\ZZ_p[[\cG]]$. Choose a projective resolution 
$P_\bullet \to \ZZ_p$ of the trivial $\ZZ_p[[\cG]]$-module $\ZZ_p$. Let $M$ be a continuous left $\cG$-module. 
We then we have a pairing
\[
\Hom_\cG(P_s,\ZZ_p[[\cG]]) \otimes_\cG M \longr \Hom_\cG(P_s,M)
\]
sending $\phi \otimes a$ to the function $x \mapsto \phi(x)a$. This passes to an induced pairing 
\[
D_p(\cG) \otimes_\cG M = H^d(\cG,\ZZ_p[[\cG]]) \otimes_\cG M \longr H^d(\cG,M). 
\] 
By \Cref{lem:length-precise} $\cG$ has the property that $H^s(\cG,M) = 0$ for $s > d$ for all $p$-profinite module $M$, we find
this is a natural transformation between right exact functors in $M$. Thus we get a natural transformation of left derived functors
and, since the $s$th left derived functor of $H^d(\cG,-)$ is $H^{d-s}(\cG,-)$ we have the homomorphism
\begin{equation}\label{eq:pdf-one-pre}
\xymatrix{
\Tor_s^{\cG}(D_p(\cG),M) \ar[r] & H^{d-s}(\cG,M).
}
\end{equation}
If $M \cong \ZZ_p[[\cG]]$ this map is an isomorphism when $s=0$ and both source
and target of this homomorphism  vanish if $s\ne 0$. From this we deduce that \eqref{eq:pdf-one-pre}
is an isomorphism when $M$ is $p$-profinite. In particular, it is true when $M$ is finite.
Since any discrete $p$-torsion module is the colimit of its finite submodules, we have the result in that case as well. 
\end{proof}

\begin{rem} In \cite{SymWei} the authors define $D_p(\cG)$ to be the {\it left} $\cG$-module obtained from
$H^d(\cG,\ZZ_p[[\cG]])$ by $g\cdot x = xg^{-1}$. This allows them to write the Poincar\'e Duality isomorphism of \Cref{thm:pd-forg} 
using homology. If we write $D_p(\cG)^\ell$ for this left module structure, then we have, under the hypotheses
of \Cref{thm:pd-forg}
\[
H_s(\cG,D_p(\cG)^\ell \otimes M) \cong \Tor_s^\cG(D_p(\cG),M) \cong H^{d-s}(\cG,M).
\]
There are times when this point of view is convenient. See the proof of \Cref{prop:dual-via-trans-alg}.
\end{rem}

\begin{rem}\label{rem:orientable}  The Poincar\'e duality group $\cG$ is called {\it orientable} if the action on
$D_p(\cG)$ is trivial. Then \cref{thm:pd-forg} implies
\[
H_{s}(\cG,M) \cong  H^{d-s}(\cG,M)
\]
and $H_d(\cG,\ZZ_p) \cong \ZZ_p$. 
\end{rem}

\begin{rem}\label{rem:when-finite-coh} Our main examples are orientable, using the following argument adapted from
the proof of Proposition 5 of \cite{StrickGH}.

By Corollary 5.2.5 of \cite{SymWei} (see \Cref{prop:ipg-is-top-ext} for more
explanation) there is an isomorphism of right $\cG$-modules
\[
D_p(\cG) \cong \Lambda^d \mathfrak{g}^\ast
\]
where $ \Lambda^d \mathfrak{g}^\ast$ is the top exterior power of the dual of the adjoint representation.
In the case where $\cG = \Gl_n(\ZZ_p)$ or $\GG_n$, the group $\cG$ is the group of units in a sub-algebra $A_0$ of a
complete normed $\QQ_p$-algebra and $\gg$ is $A_0$ with the conjugation action of $\cG$. See \Cref{exmp:exam1} and
\Cref{exmp:exam2}. If $f:\gg^\ast \to \gg^\ast$ is any linear transformation,
then $\Lambda^df$ is multiplication by the determinant of $f$. Thus if $g \in \cG$ acts by conjugation on $\gg = A_0$,
it must act trivially $\Lambda^d \gg^\ast$. 
\end{rem}

There is a variant of Poincar\'e Duality which 
is closer to the Serre-Grothendieck duality of algebraic geometry. Let $\ZZ/p^\infty = \colim \ZZ/p^n$; this
is an injective abelian group.  Define the {\it discrete dualizing module} for $\cG$ as the left $\cG$-module
\begin{equation}\label{rem:dualize-the-dualize}
I_p(\cG) = \Hom(D_p(\cG),\ZZ/p^\infty). 
\end{equation}
The action of $\cG$ is given by the formula $g\phi(x) = \phi(xg)$. Evaluation gives an isomorphism\[
D_p(\cG) \otimes_\cG I_p(\cG) \cong \ZZ/p^\infty.
\]
If $H^s(\cG,-) = 0$ for $s > d$, then \Cref{thm:pd-forg} gives an isomorphism
\[
\xymatrix{
\epsilon\colon H^d(\cG,I_p(\cG)) \ar[r]^-\cong & \ZZ/p^\infty.
}
\]
If $M$ is a continuous $\cG$ module define
\[
M^\vee = \Hom^\Delta(M,I_p(\cG))
\]
to be the  abelian group of continuous homomorphisms with the left
$\cG$-action given $(g\varphi)(a) = g\varphi(g^{-1}a)$. See (4) of \Cref{not:keep-it-str}. The evaluation
map $M^\vee \otimes M \to I_p(\cG)$ and cup product give us a pairing
\begin{equation}\label{eq:pd-pairing-infty}
\xymatrix{
H^s(\cG,M^\vee) \otimes H^{d-s}(\cG,M) \ar[r] & H^d(\cG,I_p(\cG)) \ar[r]^-{\epsilon} & \ZZ/p^\infty.
}
\end{equation}
Then we have, as in \S I.3.4 of \cite{CohoGal}:

\begin{thm}\label{thm: pd-is-sg} Let $\cG$ be a Poincar\'e duality group of dimension $d$ and suppose
$M$ is either $p$-profinite or $p$-torsion. Then
\eqref{eq:pd-pairing-infty} is a perfect pairing
which identifies $H^s(\cG,M^\vee)$ with the Pontryagin dual of $H^{d-s}(\cG,M)$. 
\end{thm}

\subsection{Frobenius reciprocity and duality}\label{sec:Frob-resp}

Poincar\'e Duality flips transfer and restriction. To make this precise we begin with the following result.

\begin{prop}\label{rem:trans-fund} Let $\cG$ be a compact analytic group of rank $d$ and
$\Gamma \subseteq \cG$ an open subgroup. Then the compact dualizing module $D_p(\Gamma)$ for $\Gamma$
is isomorphic to the compact dualizing module $D_p(\cG)$ for $\cG$ with action restricted to $\Gamma$.
\end{prop} 

\begin{proof} Since we have an isomorphism of right $\Gamma$-modules
\[
\ZZ_p[[\Gamma]]_\Gamma^\cG = \ZZ_p[[\cG]] \otimes_\Gamma \ZZ_p[[\Gamma]] \cong \ZZ_p[[\cG]]
\]
Schapiro's Lemma \Cref{prop:schapiro-lemma} gives
\[
D_p(\Gamma) = H^d(\Gamma,\ZZ_p[[\Gamma]]) \cong H^d(\cG,\ZZ_p[[\cG]]) = D_p(\cG)
\]
and these isomorphisms respect the right action by $\Gamma$. 
\end{proof} 

Note that it follows that the discrete dualizing module $I_p(\Gamma)$ for $\Gamma$ is 
isomorphic to the discrete dualizing module $I_p(\cG)$ for $\cG$ with action restricted to $\Gamma$.
This is Proposition 18 of \S I.3.5 of \cite{CohoGal}.

We now have the following result. Note that $D_p(\cG)$ has a canonical structure as $\Gamma$-module by
\Cref{rem:trans-fund}. 

\begin{prop}\label{prop:frob-in-all-glory} Let $\cG$ be compact $p$-adic analytic group of rank $d$
and suppose $H^s(\cG,-) = 0$ for $s > d$. Then we have commutative diagrams 
\begin{equation}\label{eq:trans-res-1a}
\xymatrix{
\Tor_s^\cG(D_p(\cG),M) \ar[r] \ar[d]_{\tr_\ast} & H^{d-s}(\cG,M) \ar[d]^\res\\
\Tor_s^\Gamma(D_p(\cG), M) \ar[r] & H^{d-s}(\Gamma, M).
}
\end{equation}
and
\begin{equation}\label{eq:trans-res-1b}
\xymatrix{
\Tor_s^\Gamma(D_p(\cG), M) \ar[r] \ar[d]_{\res_\ast} & H^{d-s}(\Gamma, M) \ar[d]^\tr\\
\Tor_s^\cG(D_p(\cG),M) \ar[r] & H^{d-s}(\cG, M).
}
\end{equation}
\end{prop}

\begin{proof} We give an argument only for \eqref{eq:trans-res-1a}, as that is what we will use later. The argument
for \eqref{eq:trans-res-1b} is very similar.

The map $\Tor_s^\cG(D_p(\cG),M) \to H^{d-s}(\cG,M)$ was constructed in the proof of \Cref{thm:pd-forg} by first defining it for
$s=0$ and then using naturality to extend to the higher left derived functors. Built into this construction is
the assertion that if
\[
0 \longr K \longr M \longr N \to 0
\]
is a short exact sequence of continuous $\cG$-modules then we have a commutative diagram of long exact sequences
\[
\xymatrix@C=12pt{
\Tor_{s+1}^\cG(D_p(\cG), N) \ar[d]\ar[r] & \Tor_s^\cG(D_p(\cG), K) \ar[d]\ar[r] &
\Tor_s^\cG(D_p(\cG), M) \ar[d]\ar[r] & \Tor_s^\cG(D_p(\cG), N) \ar[d]\\
H^{d-s-1}(\cG,N) \ar[r] & H^{d-s}(\cG,K) \ar[r] & H^{d-s}(\cG,M) \ar[r] & H^{d-s}(\cG,N)
}
\]
Thus if we can prove that \eqref{eq:trans-res-1a} commutes at $s=0$ then it will follow for all $s \geq 1$ as well. 

Choose a projective resolution  $P_\bullet \to \ZZ_p$ of the trivial $\ZZ_p[[\cG]]$-module $\ZZ_p$. The map
$D_p(\cG)\otimes_\cG M \to H^{d}(\cG,M)$ is defined at the chain level by the pairing
\[
\Hom_\cG(P_s,\ZZ_p[[\cG]]) \otimes_\cG M \longr \Hom_\cG(P_s,M)
\]
sending $\phi \otimes a$ to the function $x \mapsto \phi(x)a$. Let
\begin{align*}
\psi: M &\longr \ZZ_p[[\cG]] \otimes_\Gamma M\\
\psi(a) &= \sum_{g\Gamma \in \cG/\Gamma} g \otimes g^{-1}a
\end{align*}
be the map of \eqref{eq:trans-commutes} used to define $\tr_\ast$.
Since $\ZZ_p[[\cG]] \cong  \big(\ZZ_p[[\Gamma]]\big)_\Gamma^\cG$ we have a commutative diagram
\[
\xymatrix{
\Hom_\cG(P,\ZZ_p[[\cG]]) \otimes_\cG M \ar[r] \ar[d]_{1 \otimes \psi} & \Hom_\cG(P,M)\ar[dd]\\
\Hom_\cG(P,\ZZ_p[[\cG]]) \otimes_\cG  \ZZ_p[[\cG]] \otimes_\Gamma M \ar[d]_\cong\\
\Hom_\Gamma(P,\ZZ_p[[\Gamma]]) \otimes_\Gamma M \ar[r]  &\Hom_\Gamma(P,M)
}
\]
where the right vertical map is inclusion. Passing to cohomology gives the needed commutative diagram
\[
\xymatrix{
H^d(\cG,\ZZ_p[[\cG]]) \otimes_\cG M \ar[r] \ar[d]_{\tr_\ast} & H^d(\cG,M) \ar[d]^\res\\
H^d(\Gamma,\ZZ_p[[\Gamma]]) \otimes_\Gamma M \ar[r] & H^d(\Gamma,M). 
}
\]
\end{proof}

\subsection{Restricting to subgroups} The compact dualizing module $D_p(\Gamma)$ for $\Gamma$
is isomorphic to the compact dualizing module $D_p(\cG)$ for $\cG$ with action restricted to $\Gamma$. 
See \Cref{rem:trans-fund}. We now discuss how to recover the right-$\cG$ module structure on $D_p(G)$ from
$D_p(\Gamma)$ using the subgroup structure of $\cG$.

If $g \in \cG$ and $\Gamma \subseteq \cG$ is a subgroup, let $ \gGa$ denote the conjugate 
subgroup $g^{-1}\Gamma g$ and $c_g(-) = g(-)g^{-1}:\gGa \to \Gamma$ the conjugation homomorphism.
If $\phi:\Gamma^{s+1} \to \ZZ_p[[\Gamma]]$
is a continuous cochain, define $\chi_g(\phi): \gGa^{s+1} \to \ZZ_p[[\gGa]]$ by
\[
\chi_g(\phi)(x_0,x_1,\ldots,x_s) = g^{-1}\phi(c_g(x_0),\ldots,c_g(x_n))g.
\]
This extends to subgroups the definition given for $\cG$ itself in \eqref{eq:conj-in-cochains}. 
We check that this induces a map on cochain complexes and hence a map
\[
\chi_g^\ast:H^\ast(\Gamma,\ZZ_p[[\Gamma]]) \to H^\ast(\gGa,\ZZ_p[[\gGa]]).
\]

\begin{prop}\label{prop:conj2ba} Suppose $\cG$ is a compact $p$-adic analytic group of rank $d$ and
$\Gamma$ is an open subgroup. Suppose $H^s(\Gamma,-) = 0$
for $s > d$. The map
\[
\xymatrix{
D_p(\Gamma) = H^d(\Gamma,\ZZ_p[[\Gamma]]) \ar[r]^-{\chi_g^\ast} & D_p(\gGa) = H^d(\gGa,\ZZ_p[[\gGa]])
}
\]
is isomorphic to the map
\[
\xymatrix{
r_g = {(-)g}:D_p(\cG) \ar[r] & D_p(\cG).
}
\]
\end{prop}

\begin{proof} A chain level calculation shows that the following diagram commutes
\[
\xymatrix{
H^d(\Gamma,\ZZ_p[[\Gamma]]) \ar[r]^{\chi_g^\ast} \ar[d]_{\cong} & H^d(\gGa,\ZZ_p[[\gGa]]) \ar[d]^\cong\\
H^d(\cG,\ZZ_p[[\cG]])\ar[r]_{\chi_g^\ast} & H^d(\cG,\ZZ_p[[\cG]]).
}
\]
By \Cref{prop:conjugation-is-right} the bottom map is given by the standard right $\cG$-action on $H^\ast(\cG,\ZZ_p[[\cG]])$. 
\end{proof}

\begin{rem}
In particular, if $\Gamma$ is normal in $\cG$, the conjugation action of $\cG$ on $\Gamma$ and $\ZZ_p[[\Gamma]]$
recovers the right action of $\cG$ on $H^d(\cG,\ZZ_p[[\cG]])$ from the conjugation action on $H^d(\Gamma,\ZZ_p[[\Gamma]])$.
\end{rem}

\begin{rem} The conjugation map $c_g$ defines an isomorphism of left $\ZZ_p[[\gGa]]$-modules $\ZZ_p[[\gGa]] \cong
\prescript{g}{}\ZZ_p[[\Gamma]]$ and we have a diagram mapping to the diagram of \eqref{eq:all-conj-all-time} 
\begin{equation*}\label{eq:all-conj-all-time-bis}
\xymatrix@R=10pt{
&H^\ast(\gGa,\ZZ_p[[\gGa]]) \ar[dd]^\cong\\
H^\ast(\Gamma,\ZZ_p[[\Gamma]]) \ar[ur]^{\chi_g^\ast}  \ar[dr]_{c_g^\ast}\\
&H^\ast(\gGa,\prescript{g}{}\ZZ_p[[\Gamma]]).
}
\end{equation*}
Thus we could have worked with $c_g^\ast$ (\Cref{ex:conj-action}) as well. 
\end{rem} 

\Cref{prop:frob-in-all-glory} described how duality interacted with
transfer and restriction. We'd also like to know how duality interacts with conjugation. We gain conceptual
clarity by proving a more general result; we then specialize. The main result is \Cref{prop:conj2b}

Let $\cG$ be a compact $p$-adic analytic group, $N$ a continuous left $\cG$-module and $M$ a continuous $\cG$-bimodule.
We have a pairing
\[
\pair: C^s(\cG,M) \otimes_\cG N \longr C^s(\cG,M \otimes_\cG N)
\]
sending $\phi \otimes a$ to $\psi$ where
\[
\psi(x_1,\ldots,x_s) = \phi(x_1,\cdots,x_s) \otimes a.
\]
We now explore the naturality of this pairing. Let $f:\cG_1 \to \cG_2$ be a continuous homomorphism
of compact $p$-adic analytic groups, $N$ a continuous left $\cG_1$-module and $M$ a continuous $\cG_2$-bimodule. For simplicity
we assume that $f$ is an injection onto a subgroup of finite index, so that we can form the induced module
\[
N_{\cG_1}^{\cG_2} \cong \ZZ_p[[\cG_2]] \otimes_{\ZZ_p[[\cG_1]]} N =: \cG_2 \times_{\cG_1} N.
\]
If $M$ is a right $\cG_2$ module, let $M^f$ denote the \emph{right} $\cG_1$-module obtained by restriction. Recall, from
Section 4.1, that $f_*M$ denotes the \emph{left} $\cG_1$-module obtained by restriction.
The following can be checked on cochains.

\begin{lem}\label{lem:conj1a} Let $N$ be a continuous left $\cG_1$-module, $M$ a continuous $\cG_2$-bimodule and $f \colon \cG_1 \to \cG_2$  a continuous homomorphism
of compact $p$-adic analytic groups which is an
injection onto a subgroup of finite index. Then the pairing $\pair$ induces a commutative diagram
\[
\xymatrix{
H^s(\cG_2,M) \otimes_{\cG_2} (\cG_2 \times_{\cG_1} N) \ar[r]\ar[d]_\cong &
H^s(\cG_2,M \otimes_{\cG_2} (\cG_2 \times_{\cG_1} N))\ar[d]^\cong\\
H^s(\cG_2,M^f) \otimes_{\cG_1} N \ar[r] \ar[d]_{f_\ast} & H^s(\cG_2,M^f \otimes_{\cG_1} N)\ar[d]^{f_\ast}\\
H^s(\cG_1,\fMf) \otimes_{\cG_1} N \ar[r] & H^s(\cG_1,\fMf \otimes_{\cG_1} N) \ .
}
\]
\end{lem} 
\bigskip

We now specialize. Let $\Gamma \subseteq \cG$ be an open subgroup and $g \in \cG$. We let $f$
be the conjugation map $c_g:\gGa \to \Gamma$. For $M$ a $\Gamma$-bimodule, let $ \prescript{g}{}M^g$ be $f_*M^{f}$, the $\gGa$-
bimodule obtained by restriction along $f=c_g$ for both the right and left actions.
Then $c_g$ defines an isomorphism of $\gGa$-modules
\[
\xymatrix{
 \ZZ_p[[\gGa]] \ar[r]^\cong & \prescript{g}{}\ZZ_p[[\Gamma]]^g
}
\]
and if $N$ is a $\gGa$-module
\begin{align*}
\Gamma \times_{\gGa} N &\cong \prescript{g^{-1}}{}N
\end{align*}
where $ \prescript{g^{-1}}{}N$ is the module obtained by restriction along $c_{g^{-1}} \colon \Gamma \to \gGa$.
Thus if we replace $N$ by $\prescript{g}{}N$ in \Cref{lem:conj1a} we have a diagram 
\[
\xymatrix{
H^s(\Gamma,\ZZ_p[[\Gamma]]) \otimes_{\Gamma} N \ar[r]\ar[d]_{c_g^\ast} &
H^s(\Gamma,\ZZ_p[[\Gamma]] \otimes_{\Gamma} N)\ar[d]^{c_g^\ast}\\
H^s(\gGa,\ZZ_p[[\gGa]]) \otimes_{\gGa} \prescript{g}{}N \ar[r] & H^s(\gGa,\ZZ_p[[\gGa]] \otimes_{\gGa} \prescript{g}{}N)\\
}
\]
After using the module multiplications on the coefficients, we get a diagram
\begin{equation}\label{eq:conj1b}
\xymatrix{
H^s(\Gamma,\ZZ_p[[\Gamma]]) \otimes_{\Gamma} N \ar[r]\ar[d]_{c_g^\ast  \otimes N} &
H^s(\Gamma,N)\ar[d]^{c_g^\ast}\\
H^s(\gGa,\ZZ_p[[\gGa]]) \otimes_{\gGa} \prescript{g}{}N \ar[r] & H^s(\gGa,\prescript{g}{}N)\\
}
\end{equation}
The right hand morphism $c_g^\ast$ sends $\phi \otimes a$ to the cocyle
\[
(x_1,\ldots,x_s) \longmapsto g^{-1}\phi(c_g(x_1),\ldots,c_g(x_s))g \otimes a.
\]

In the next result we use that the compact dualizing module $D_p(\Gamma)$ for $\Gamma \subseteq \cG$ open
is isomorphic to the compact dualizing module $D_p(\cG)$ for $\cG$ with action restricted to $\Gamma$. 
See \Cref{rem:trans-fund}.

\begin{prop}\label{prop:conj2b} Suppose $\cG$ is a compact $p$-adic analytic group of rank $d$ and
$\Gamma$ is an open subgroup and a Poincar\'e duality group. Let $M$ be a left $\Gamma$-module. Then the map
\[
\xymatrix{
D_p(\Gamma) = H^d(\Gamma,\ZZ_p[[\Gamma]]) \ar[r]^-{\chi_g^\ast} & D_p(\gGa) = H^d(\gGa,\ZZ_p[[\gGa]])
}
\]
is isomorphic to the map
\[
\xymatrix{
r_g = {(-)g}:D_p(\cG) \ar[r] & D_p(\cG)
}
\]
and we have a commutative diagram 
\[
\xymatrix{
\Tor_s^\Gamma(D_p(\cG),M) \ar[r]^-\cong\ar[d]_{(r_g)_\ast} & H^{d-s}(\Gamma,M)\ar[d]^{c_g^\ast}\\
\Tor_s^\gGa(D_p(\cG),\gMa) \ar[r]^-\cong & H^{d-s}(\gGa,\gMa).
}
\]
\end{prop}

\begin{proof} The first statement is a repeat of \Cref{prop:conj2ba}. For the second, we can proceed as in the proof of
\Cref{prop:frob-in-all-glory} and note it is sufficient to check the case $s=0$. This is exactly \eqref{eq:conj1b}. The horizontal 
isomorphisms are Poincar\'e Duality. See \Cref{thm:pd-forg}.
\end{proof} 

\begin{cor}\label{cor:how-use-this} Suppose $\cG$ is a compact $p$-adic analytic group of rank $d$ and
$\Gamma$ is an open subgroup and a Poincar\'e duality group.
Suppose $M$ is a trivial $\cG$-module and $\Gamma \subseteq \cG$ is normal.
Then \Cref{prop:conj2b} gives a diagram
\[
\xymatrix{
D_p(\cG) \otimes_\Gamma M \ar[r]^-\cong\ar[d]_{(r_g)_\ast \otimes M} & H^{d}(\Gamma,M)\ar[d]^{c_g^\ast}\\
D_p(\cG) \otimes_\Gamma M \ar[r]^-\cong & H^{d}(\Gamma,M)
}
\]
where the left column is given by the right action of $g$ on $D_p(\cG)$ and the right column is given
by the conjugation action of $g$ on $\Gamma$. 
\end{cor}

\subsection{The dualizing module revisited}

We can now give a formula for the discrete dualizing module $I_p(\cG)$ due to Serre. The module
$I_p(\cG)$ was defined in \eqref{rem:dualize-the-dualize}. Let $\cG$ be a compact $p$-adic analytic group
of rank $d$ with an exhaustive system
\[
\cdots \subseteq \Gamma_{i+1} \subseteq \Gamma_i \subseteq \cdots \subseteq \Gamma_1 
\subseteq \cG
\]
of uniformly powerful open normal subgroups. The following gives exactly Serre's formula in
\S 3.5 of \cite{CohoGal}. 

\begin{prop}\label{prop:dual-via-trans-alg} There is an isomorphism of left $\cG$-modules
\begin{equation}\label{eq:dual-via-trans-alg-0}
I_p(\cG) \cong \colim_{\tr} H_d(\Gamma_i,\ZZ/p^\infty)
\end{equation}
where $\cG$ acts on $H_\ast(\Gamma_i,\ZZ/p^\infty)$ through conjugation on $\Gamma_i$. 
\end{prop}

\begin{proof} Let $D_p(\cG)^\ell$ be left-$\cG$-module obtained by transposing the right action; that is $g\cdot a = ag^{-1}$. 
Then if $M$ is any continuous $\cG$-module we have a natural isomorphism 
\[
D_p(\cG) \otimes_\cG M \cong \ZZ_p \otimes_{\cG} (D_p(\cG)^\ell \otimes M)
\]
where $D_p(\cG)^\ell \otimes M$ has the diagonal action. We thus get a natural isomorphism
\[
H_s(\cG,D_p(\cG)^\ell \otimes M) \cong \Tor^\cG_s(D_p(\cG),M). 
\]
There is an isomorphism of left $\cG$-modules
\[
D_p(\cG)^\ell \otimes I_p(\cG) \cong \ZZ/p^\infty
\]
where the tensor product has the diagonal $\cG$-structure and the action on $\ZZ/p^\infty$ is trivial. 
For $i$ large Frobenius reciprocity \eqref{eq:trans-res-1a} and Poincar\'e Duality \Cref{thm:pd-forg} give
a commutative square
\[
\xymatrix{
H_s(\Gamma_i,\ZZ/p^\infty) \ar[r]^-\cong \ar[d]_{\tr_\ast} & H^{d-s}(\Gamma_i,I_p(\cG)) \ar[d]^\res\\
H_s(\Gamma_{i+1},\ZZ/p^\infty) \ar[r]^-\cong & H^{d-s}(\Gamma_{i+1},I_p(\cG)) \ .
}
\]
Thus, if we set $s=d$ and $M = \ZZ/p^\infty$ we get an isomorphism 
\[
\big[ I_p(\cG) \big]^{\Gamma_i}\cong H_d(\Gamma_i,\ZZ/p^\infty).
\]
This is a $\cG$-equivariant isomorphism by \Cref{prop:conj2b}.
If $M$ is a discrete $\cG$-module then the natural map
\[
\colim_i M^{\Gamma_i} \to M
\]
is an isomorphism; now using the above commutative diagram, we get an isomorphism of $\cG$-modules
\begin{align*}
I_p(\cG) &\cong \colim_{\tr_\ast} H_d(\Gamma_i,\ZZ/p^\infty).\qedhere
\end{align*}
\end{proof}

The following is a direct consequence of Corollary 5.2.5
of \cite{SymWei}, although it would take some translation to get from that statement to ours. The adjoint
representation was defined in \Cref{rem:adjoint-action}

\begin{prop}\label{prop:ipg-is-top-ext} Let $\cG$ be a compact $p$-adic analytic group of dimension $d$ and let
$\gg$ be the the adjoint representation. Then there is a natural isomorphism of right $\cG$-modules
\[
\Lambda^d \mathfrak{g}^\ast \cong D_p(\cG).
\]
where
\[
\gg^\ast = \Hom(\gg,\ZZ_p)
\]
with right $\cG$-action given by $(\phi g)(x) = \phi(gx)$ and $\Lambda^d \mathfrak{g}^\ast$ is the top exterior power.
\end{prop}

\begin{proof} Since $\Lambda^d \mathfrak{g}^\ast$ and $D_p(\cG)$ are free modules of rank $1$ over $\ZZ_p$ it
is sufficient to produce a surjective map of $\cG$-modules
\[
\Lambda^d \mathfrak{g}^\ast \longr D_p(\cG)/p^j
\]
for all $j \geq 1$. Recall from \Cref{rem:adjoint-action} that $\gg_i = p^{i-1}\gg$ and, since $\gg$ is $p$-torsion free, that
$p^{i-1} \colon \gg \to \gg_i$ is an isomorphism. 

Fix $j$. Then $D_p(\cG)/p^j$ is a finite $\cG$-module, so we may choose $i \geq j$ so that $\Gamma_i$ acts trivially on
$D_p(\cG)/p^j$. Then, using \eqref{eq:gg-gamma}, we have a map of left $\cG$-modules
\[
\xymatrix{
\gg \ar[r]^-{p^{i-1}} & \gg_i/\gg_{i+j}\ar[r]^-\cong & \Gamma_i/\Gamma_{i+j} 
}
\]
where $\cG$ acts on the target by conjugation. This gives an isomorphism of right $\cG$-modules
\[
\xymatrix{
H^1(\Gamma_i/\Gamma_{i+j},\ZZ/p^j) \cong \Hom(\gg,\ZZ/p^j).
}
\]
We then obtain a map of right $\cG$-modules
\[
\gg^\ast = \Hom(\gg,\ZZ_p) \longr \Hom(\gg,\ZZ/p^j) \cong H^1(\Gamma_i/\Gamma_{i+j},\ZZ/p^j) \cong H^1(\Gamma_i,\ZZ/p^j).
\]
Using the cup product we get a map of right $\cG$-modules
\[
\Lambda^d \gg^\ast \longr H^d(\Gamma_i,\ZZ/p^j)
\]
which is surjective by \Cref{thm:coh-gammai}. By Poincar\'e Duality and the fact that $D_p(\Gamma_i) \cong D_p(\cG)$
with action restricted to $\Gamma_i$ \Cref{cor:how-use-this} gives an isomorphism of right $\cG$-modules
\[
H^d(\Gamma_i,\ZZ/p^j) \cong H^d(\cG,\ZZ_p[[\cG]]) \otimes_{\Gamma_i} \ZZ/p^j  \cong H^d(\cG,\ZZ_p[[\cG]])/p^j.
\]
The last isomorphism follows from the fact we have chosen $i$ so that $\Gamma_i$ acts trivially on
\[
D_p(\cG)/p^j = H^d(\cG,\ZZ_p[[\cG]])/p^j.
\]
The composition
\[
\Lambda^d \gg^\ast \longr H^d(\Gamma_i,\ZZ/p^j) \cong D_p(\cG)/p^j
\]
is the surjection we need. 
\end{proof}

\begin{rem}\label{rem:dual-via-trans-alg-two} We can combine \Cref{prop:dual-via-trans-alg} and \Cref{prop:ipg-is-top-ext}
to obtain an isomorphism of left $\cG$-modules
\[
I_p(\cG)/p^j \cong \colim_{\tr_\ast} H^d(\Gamma_i,\ZZ/p^j) \cong \Lambda^d\gg/p^j.
\]
This follows from the fact that $I_p(\cG) = \Hom(D_p(\cG),\ZZ/p^\infty)$ and that the dual of $\Lambda^d \gg$ is
$\Lambda^d \gg^\ast$. 
\end{rem} 


\section{Two $\cG$-spheres}\label{sec:TwoSpheres}

In this section, we produce the two $p$-complete $\cG$-spheres which are the main characters of this paper.
They are produced by the same method, so we will give complete details for only one. Here as usual
$\cG$ is a compact $p$-adic analytic group with
an exhaustive system $\cdots \subseteq \Gamma_{i+1} \subseteq \Gamma_i \subseteq \cdots \Gamma_1 
\subseteq \cG$ of uniformly powerful open normal subgroups.

\begin{notation}\label{rem:p-complete} Here and hereafter
we will work in the category of $p$-complete spectra unless we specify otherwise. In particular
we will write
\[
\Sigma^\infty_+ B\Gamma_i \qquad\mathrm{for} \qquad \big(\Sigma^\infty_+ B\Gamma_i\big)_p\ .
\]
This is necessary as we will want to use \Cref{prop:it-doesn't-matter}. 
\end{notation}

The spectra
$\Sigma^\infty_+B\Gamma_i/\Gamma_j$ and $\Sigma^\infty_+B\Gamma_i$ have a $\cG$-action induced from the 
conjugation action of $\cG$ on $\Gamma_i$. Let $j > 1$; then we have an inclusion of a normal subgroup of finite index
$\Gamma_{i+1}/\Gamma_{i+j} \to \Gamma_{i}/\Gamma_{i+j}$
and hence a $\cG$-equivariant transfer map
\[
\Sigma^{\infty}_+ B(\Gamma_{i}/\Gamma_{i+j}) \longr \Sigma^\infty_+ B(\Gamma_{i+1}/\Gamma_{i+j}).
\]
By taking homotopy inverse limits, we then get an induced $\cG$-equivariant transfer map
\begin{align}\label{eq:tr-gammas}
\tr: \Sigma^{\infty}_+ B\Gamma_{i} \longr \Sigma^\infty_+ B\Gamma_{i+1}.
\end{align}

\begin{rem}\label{rem:naturality-tr} To produce the transfer map of \eqref{eq:tr-gammas} we use the following
naturality property of the transfer. Let $K \subseteq H \subseteq G$ be a nested sequence of 
subgroups with all three inclusions normal. Then the following diagram commutes
\[
\xymatrix{
\Sigma_+^\infty BG \ar[r]^{\tr} \ar[d] & \Sigma_+^\infty BH \ar[d]\\
\Sigma_+^\infty B(G/K) \ar[r]^\tr & \Sigma_+^\infty B(H/K).
}
\]
\end{rem}

The following construction is motivated by Serre's formula from \Cref{prop:dual-via-trans-alg}. 

\begin{defn}\label{def:ig-defined} Let $\ig$ be the $\cG$-spectrum
\begin{equation}\label{eq:the-big-kahuna}
\ig = \hocolim \Sigma^\infty_+ B\Gamma_{i}
\end{equation}
where the colimit is over the transfer maps and in the category of $p$-complete spectra; that is, we are taking
the $p$-completion of the ordinary homotopy colimit. 
\end{defn}

\begin{prop}\label{lem:transfer-hom} We have the following calculations in homology.

\begin{enumerate}
\item The map induced in homology by the transfer
\[
\tr_\ast:H_k (B\Gamma_{i},\FF_p) \longr H_k (B\Gamma_{i+1},\FF_p)
\]
is zero if $k \ne d$ and an isomorphism if $k=d$.

\item $H_k \ig = 0$ unless $k =d$ and $H_d\ig \cong \FF_p$. In particular, $\ig$ has the homotopy type of a 
$p$-complete $d$-sphere. 
\end{enumerate}
\end{prop}

\begin{proof} Using \Cref{prop:frob-in-all-glory} the map induced by the
transfer is isomorphic, via Poincar\'e Duality, to the map induced in cohomology by restriction
\[
H^{k-d} (B\Gamma_{i},\FF_p) \longr H^{k-d} (B\Gamma_{i+1},\FF_p).
\]
By Part (2) of \Cref{thm:coh-gammai} the map $H^{\ast} (B\Gamma_{i},\FF_p) \to H^{\ast} (B\Gamma_{i+1},\FF_p)$ is
zero in positive degrees, and the result follows from the Whitehead Theorem for connective spectra.
\end{proof}

We can use these same methods to build a $p$-complete $\cG$ sphere from the Lie algebra $\gg$ of $\cG$.
(The Lie algebra and its properties were discussed in \Cref{def:lie-algebra} and material following that
definition.) Specifically, we use the abelian groups $p^i\gg \subseteq \gg$ in place of $\Gamma_i$, and form the $p$-complete spectrum
\begin{equation}\label{eq:sg-defined}
S^\gg =  \hocolim \Sigma^\infty_+ B(p^i\gg).
\end{equation}
As in \Cref{def:ig-defined}, the colimit is taken over transfers in the category of $p$-complete spectra.
The action of $\cG$ on $S^\gg$ is through the adjoint representation. See \Cref{rem:adjoint-action}.

The following is the analog of \Cref{lem:transfer-hom}.

\begin{lem}\label{lem:transfer-hom-ad} We have the following calculations in homology.

\begin{enumerate}
\item The map induced in homology by the transfer
\[
\tr_\ast:H_k (Bp^i\gg,\FF_p) \longr H_k (Bp^{i+1}\gg,\FF_p)
\]
is zero if $k \ne d$ and an isomorphism if $k=d$.

\item $H_k S^\gg = 0$ unless $k =d$ and $H_d S^\gg\cong \FF_p$. In particular, $S^\gg$ has the homotopy type of a 
$p$-complete $d$-sphere. 
\end{enumerate}
\end{lem}

The next result identities the $\cG$ action on $H_\ast (\ig,\ZZ_p)$ and $H_\ast (S^\gg,\ZZ_p)$. As these
are both $p$-complete $d$-spheres, we have $H_\ast (\ig,\ZZ_p) \cong \ZZ_p \cong H_\ast (S^\gg,\ZZ_p)$. 

\begin{prop}\label{lem:top-action} There are canonical isomorphisms of $\cG$-modules
\[
H_d (\ig,\ZZ_p) \cong \Lambda^d \gg \cong H_d (S^\gg,\ZZ_p),
\]
where $\Lambda^d \gg$ is the top exterior power of the adjoint representation.
\end{prop}

\begin{proof} For $\ig$ we have 
\[
H_d (\ig,\ZZ_p) \cong \lim_j H_d(\ig,\ZZ/p^j)  \cong \lim_j \colim_\tr H_d(\Gamma_i,\ZZ/p^j).
\]
Now apply \Cref{rem:dual-via-trans-alg-two}. 

For $S^\gg$ we argue slightly differently. By \Cref{lem:transfer-hom-ad} we have
\[
H_d (S^\gg,\ZZ_p) \cong H_d (B\gg,\ZZ_p) \cong \Lambda^d \gg,
\]
as $H_1(B\gg,\ZZ_p) \cong \gg$.
\end{proof}

\Cref{lem:top-action} suggests the following conjecture.
\bigskip

\noindent{\bf Linearization Hypothesis:} There is a $\cG$-equivariant 
equivalence $S^\gg \simeq \ig$.
\bigskip

This conjecture appeared in the work of Clausen \cite[\S 6.4]{Clausen}. Furthermore, \cite{OWS} contains an outline of an argument for showing this conjecture in full generality, including a discussion of what category is a natural home for this equivalence.

If we restrict to a finite subgroup $F \subseteq \cG$, we can regard $\ig$ and $S^\gg$ as objects in
some standard category of spectra with an $F$-action, such as the functor category from the classifying space $BF$ to spectra, or the localization of (any model of) genuine $F$-spectra at the underlying equivalences. Then we have the following conjectural statement we can attack with neoclassical techniques, as we do in the next few sections. 
\bigskip

\noindent{\bf Finite Linearization Hypothesis:} For every finite subgroup $F \subseteq \cG$ there is an $F$-equivariant 
equivalence $S^\gg \simeq \ig$.

\section{Restricting to finite subgroups}\label{sec:resfinite}

In this section we reduce the problem of producing an $F$-equivariant equivalence $\ig \to S^\gg$ for a finite subgroup of $\cG$ to a 
calculation of modules over the Steenrod algebra. The main result is \Cref{thm:whod-have-believed-it}.

\begin{notation}
Both $p$-complete and non-$p$-complete spectra will appear in this section, so we will be suitably decorating the $p$-complete 
objects. For spaces $X$ and $Y$, we will write $\maps(X,Y)$ for the unpointed mapping space, and $\maps_*(X,Y)$ for the pointed 
mapping space (if $X,Y$ are based). The mapping space of spectra will be denoted by $\mapsp(-,-)$. For spaces
we will write
\[
[X,Y] = \pi_0\maps_\ast(X,Y)
\]
for {\it based} homotopy classes of maps. 
\end{notation}

\begin{rem}\label{rem:point-or-not}
Recall that if $X$ and $Y$ are pointed spaces then evaluation at the basepoint of $Y$ gives a fiber sequence
\[
\maps_\ast (X,Y) \longrightarrow \maps(X,Y) \longrightarrow Y
\]
with a section given by the constant maps. When $Y$ is a loop space, this gives a splitting $\maps(X,Y) \simeq
\maps_*(X,Y) \times Y$. Therefore, if $Y$ is connected, then
\[
[X,Y] = \pi_0\maps_\ast(X,Y) \cong \pi_0\maps(X,Y).
\]
\end{rem}
 
\subsection{$F$-spheres} Let $F$ be a finite group. An $F$-{\it sphere} is an $F$-spectrum $X$ so that
the underlying spectrum is equivalent to a $p$-complete  sphere $S^k_p$, $k \in \ZZ$. We call $k$ the virtual dimension of $X$.
An $F$-equivariant map $X \to Y$ of two $F$-spheres is an $F$-{\it equivalence} if the underlying map of spectra is an 
equivalence.  Let $\Sp_F$ be the set of $F$-equivariant equivalence classes of $F$-spheres. This a group under smash 
product in the $p$-complete stable category.

We have the following classical result.

\begin{lem}\label{lem:f-spheres-1} Let $\haut(S^0_p)$ be the space of self homotopy equivalence of the $p$-complete sphere 
spectrum $S^0_p$. There is an isomorphism of abelian groups
\begin{align*}
\Sp_F &\cong \pi_0\maps(BF, \Z \times B\haut(S^0_p))
\end{align*}
so that an $F$-sphere corresponds to an unbased map $BF \to \ZZ \times B\haut(S^{0}_p)$, which in turn classifies a stable,
fiberwise $p$-complete  spherical fibration over $BF$. We also have
\begin{align*}
\Sp_F &\cong [BF_+,\ZZ \times B\haut(S^0_p)]\\
& \cong \ZZ\times [BF, B\haut(S^0_p)].
\end{align*}
\end{lem}

\begin{proof} The first isomorphism goes back to Theorem 1.1 of \cite{CookeAction} and has been 
extensively  explored and generalized in the follow-up literature. See, for example, Theorem 1.2 of \cite{DDKProc}.\footnote{
In later developments, an action of $F$ on $X$ is {\it defined} to be a map from $BF$ to the classifying space
of self-equivalences of $X$. The connection to our definition can be found in Proposition 4.2.4.4 of \cite{HigherTopos}; see 
also Remark 1.2.6.2 of the same source.}

For the isomorphisms of the second equation use \Cref{rem:point-or-not} and the fact that $\ZZ \times B\haut(S_p^0)$
is an infinite loop space.
\end{proof}

\begin{rem} The isomorphism of \Cref{lem:f-spheres-1} allows for very explicit constructions. For example, if $X \in \Sp_F$, then the 
Thom spectrum of the associated spherical fibration is equivalent to the homotopy orbits $EF_+ \wedge_F X$. 
\end{rem}

As just noted, $\ZZ \times B\haut(S_p^0)$ is an infinite loop space; indeed, it is weakly equivalent to is the Picard space
$\sPic(S^0_p)$ of invertible $p$-complete spectra. Thus, there is a spectrum $\pic(S^0_p)$ and an equivalence
$\ZZ \times B\haut(S_p^0) \simeq \Omega^\infty \pic(S^0_p)$, yielding an isomorphism
\[
\Sp_F \cong [\Sigma^\infty_+BF,\pic(S^0_p)].
\]
Compare \Cref{rem:why-pic}. 

The spectrum $\pic(S^0_p)$ is not $p$-complete, but is nearly so, and it will be useful to have some control
over the difference. Note that
\[
\pi_0\haut(S^0_p) \cong \ZZ_p ^\times \cong
\begin{cases}
\{1+p\ZZ_p\} \times C_{p-1}&\mathrm{if}\ p > 2;\\
\\
\{1+4\ZZ_2\} \times \{\pm 1\} &\mathrm{if}\ p = 2.
\end{cases}
\]
where $C_{p-1}$ is a cyclic group of order $p-1$. There are isomorphisms $1+p\ZZ_p \cong \ZZ_p$ and $1+4\ZZ_2 \cong \ZZ_2$
using a logarithm.

\begin{lem}\label{lem:p-failure} If $p =2$, then $B\GL1(S_2^0)$ is $2$-complete. If $p>2$ the canonical map
\[
B\GL1(S_p^0) \longr BC_{p-1}
\]
has $p$-complete homotopy fiber. 
\end{lem}

\begin{proof} By definition $\haut(X)$ is a set of components in the mapping space $\mapsp(X,X)$ and, hence, for
any basepoint and any $k \geq 1$ we have
\[
\pi_k\haut(X) \cong \pi_k\mapsp(X,X) \cong [\Sigma^kX,X].
\]
From this we can conclude that for $k \geq 1$ the map
\[
\pi_k \haut(S^0) \to \pi_k\haut(S_p^0)
\]
is $p$-completion. The result follows. 
\end{proof}

If $X$ is an $F$-sphere of virtual dimension $k$, then $F$ acts on $H_k(X,\ZZ) \cong \ZZ_p$. We write 
\[
\phi_X:F \longr \ZZ_p^\times
\]
for the resulting character. The map
\[
f_X:BF \longr \{k\} \times B\haut(S^0_p) \subseteq \ZZ \times B\haut(S^0_p)
\]
induces $\phi_X$ on the fundamental group. There is also an action on $H_k(X,\FF_p) \cong \FF_p$. If $p=2$, this
action is necessarily trivial, and if $p > 2$ we write 
\[
\psi_X:F \longr \FF_p^\times = C_{p-1}
\]
for this character. Note that $\psi_X$ is obtained from $\phi_X$ by composing with the quotient map $\ZZ_p^\times \to 
\FF_p^\times$. 

\begin{lem}\label{lem:red-to-sylow} Let $F$ be a finite group and $F_0 \subseteq F$ a $p$-Sylow subgroup. Then
two $F$-spheres $X$ and $Y$ are $F$-equivalent if and only if
\begin{enumerate}

\item $X$ and $Y$ are equivalent as $F_0$-spheres and,

\item if $p > 2$, $\psi_X= \psi_Y:F \to C_{p-1}$.
\end{enumerate}
\end{lem}

\begin{proof} One implication is clear: if $X$ is equivalent to $Y$ as an $F$-sphere, then the two listed conditions must 
hold. We work on the other implication.

Since $X$ and $Y$ are equivalent as $F_0$ spheres, they have the same virtual dimension. Because 
$\ZZ \times B\haut(S_p^0)$ is an infinite loop space, we can form the difference map
\[
g \colon = f_X - f_Y\colon BF \longr B\haut(S_p^0) = \{0\} \times B\haut(S_p^0) \subseteq \ZZ \times B\haut(S_p^0)
\]
and we need only show this map is null-homotopic.  As in \Cref{lem:p-failure}, let $C= C_{p-1}$ if $p>2$; let $C$ be trivial if $p=2$.
Let $A_0$ be the homotopy fiber of the map $B\haut(S_p^0) \to BC$. Then (2) implies that there is a (unique) factoring
of $g$ as in the diagram
\[
\xymatrix{   & BF \ar[d]^-g \ar[dl]_{g_1}& \\
A_0 \ar[r] & B\Gl_1(S_p^0) \ar[r] & BC.
}
\]
We will will show $g_1$ is null-homotopic. Since the order of $C$ is prime to $p$, we have $[BF_0,A_0] \cong [BF_0,B\Gl_1(S_p^0)]$.
Hence (1) implies that the composition
\[
\xymatrix{
BF_0 \ar[r] & BF \ar[r]^{g_1} & A_0
}
\]
is null-homotopic.

Let $\bone$ be the zero connected cover of $\pic(S_p^0)$; thus, $\Omega^\infty\bone \simeq B\haut(S_p^0)$.
Let $A$ be the fiber of the map $\bone \to \Sigma HC$; then $\Omega^\infty A \simeq A_0$ and $g_1$
is adjoint to a map $h_1:\Sigma_+^\infty BF \to A$. By \Cref{lem:p-failure}, $A_0$ is $p$-complete; hence $A$
is $p$-complete as well. Since $F_0$ is a $p$-Sylow subgroup of $F$ and $A$ is $p$-complete, the restriction map
\[
[\Sigma^\infty_+ BF,A] \longr [\Sigma^\infty_+BF_0,A]
\]
is an injection; the transfer for the inclusion $F_0 \subseteq F$ gives a splitting.
Since $h_1$ maps to zero under this map, we must have $h_1=0$, and hence $g_1$ is null as well. 
\end{proof} 

\begin{rem}\label{rem:these-hyp-hold} Our main examples of $F$-spheres are the two spheres constructed in \Cref{sec:TwoSpheres}. 
When trying to compare $\ig$ with $S^\gg$, the condition (2) of \Cref{lem:red-to-sylow} holds 
automatically. See \Cref{lem:top-action}.
\end{rem}

We now begin to specialize $F$. We let $\SL1(X)$ be the component of the identity in $\haut(X)$.

\begin{lem}\label{lem:comp-is-harmless} Let $F$ be a finite $p$-group. Then the map $\haut(S^0) \to \haut(S_p^0)$ induced from the completion $S^0 \to S^0_p$ induces a weak equivalence
\[
\maps_\ast (BF,B\haut(S^0)) \simeq \maps_\ast (BF,B\haut(S_p^0)).
\]

\end{lem}

\begin{proof} For $k \geq 1$ the map of homotopy groups based at the identity
\[
\pi_k \haut(S^0) \to \pi_k\haut(S_0^p)
\]
is $p$-completion. If we set $C_2  =  \{ \pm 1\}$, there is a short exact sequence
\[
0 \to \pi_0 \haut(S^0) \cong C_2 \to \pi_0\haut(S_0^p) \cong \ZZ_p^\times \to \ZZ_p^\times/C_2 \to 0.
\]
We have a diagram of fiber sequences
\begin{align*}
\xymatrix{
B\SL1(S^0) \ar[r]\ar[d]_i & B\haut(S^0) \ar[r]\ar[d]_j \ar@{}[dr]|{(\star)} & BC_2 \ar[d]\\
B\SL1(S^0_p) \ar[r] & B\haut(S^0_p) \ar[r] & B\Z_p^\times.
}
\end{align*}
This is a diagram of infinite loop spaces so we still have an analogous diagram of fiber sequences after $p$-completion. The space
$B\SL1(S^0_p)$ is already $p$ complete so the map $i$ defines an equivalence
\[
B\SL1(S^0)_p \simeq B\SL1(S^0_p).
\]
Hence, upon $p$ completion the square $(\star)$ becomes the homotopy pullback square
\[
\xymatrix{
B\haut(S^0)_p \ar[r]\ar[d]_j \ar@{}[dr]|{(\star\star)}& (BC_2)_p \ar[d]\\
B\haut(S^0_p)_p \ar[r] & (B\Z_p^\times)_p.
}
\]

At odd primes, $(BC_2)_p \simeq \ast$ and $B(\Z_p^\times)_p \simeq B\Z_p$, so $(\star\star)$ reduces to a fiber sequence
\begin{align}\label{eq:SLtoGLp}
B\haut(S^0)_p \longrightarrow B\haut(S^0_p)_p \longrightarrow B\Z_p.
 \end{align}
When $p=2$, the spaces $BC_2$ and $B\Z_2^\times$ are already $2$-complete. Since $(\star\star)$ is a pullback square, 
we get a fiber sequence \eqref{eq:SLtoGLp} at the prime $2$ as well.

At all primes, we have $\maps_\ast (BF,B\ZZ_p) \simeq \holim \maps_\ast (BF,B\ZZ/p^k)$. Hence, for $i\geq 0$,
\[
\pi_i \maps_\ast (BF,B\ZZ_p) \cong \lim_k \tilde{H}^{1-i}(BF,\ZZ/p^k) = 0,
\]
so $\maps_\ast (BF,B\ZZ_p)$ is contractible.

 We now take pointed maps from $BF$ to \eqref{eq:SLtoGLp}
to get an equivalence
\[
\maps_\ast(BF, B \haut(S^0)_p) \to \maps_\ast(BF, B \haut(S^0_p)_p).
\]
Finally, to get the claim, notice that since $BF$ has no $\FF_\ell$ homology for $\ell \ne p$, and no rational homology, we also have
\[
\maps_\ast(BF,X) \simeq \maps_\ast(BF,X_p)
\]
for any space $X$.
\end{proof}

\begin{rem}\label{rem:comp-is-harmless-1} As an addendum to \Cref{lem:comp-is-harmless} we examine
the role of the fundamental group of $B\haut(S^0)$. There is a split fiber sequence
\[
B\SL1(S^0)\ \longr B\haut(S^0)\ \longr BC_2 .
\]
Since $B\haut(S^0)$ is an infinite loop space, we obtain
a weak equivalence of spaces
\[
B\SL1(S^0) \times BC_2 \simeq B\haut(S^0).
\]
Hence, for any pointed  connected space $X$, we have an equivalence 
\begin{equation}\label{eq:split-for-oriented}
\maps_\ast(X,B\SL1(S^0)) \times \Hom(\pi_1X,C_2) \simeq \maps_\ast(X,B\haut(S^0)). 
\end{equation}
Here we use that $\maps_\ast(X,BC_2)$ has contractible components and  
\[\pi_0\maps_\ast(X,BC_2)\cong  \Hom(\pi_1X,C_2).\]
If $F$ is an elementary abelian $p$-group and $p > 2$ we then obtain
an equivalence
\[
\maps_\ast (BF,B\SL1(S^0)) \simeq \maps_\ast (BF,B\haut(S^0)).
\]
If $p=2$ then \eqref{eq:split-for-oriented} becomes
\[
\maps_\ast(BF,B\SL1(S^0)) \times \Hom(F,C_2) \simeq \maps_\ast(BF,B\haut(S^0)). 
\]
There is a similar decomposition for maps into $BO$ and, in fact, the map $BO \to B\haut(S^0)$ induces a commutative
diagram
\[
\xymatrix{
\maps_\ast(BF,BSO) \times \Hom(F,C_2)\ar[d] \ar[r]^-\simeq & \maps_\ast(BF,BO)\ar[d]\\
\maps_\ast(BF,B\SL1(S^0)) \times \Hom(F,C_2) \ar[r]^-\simeq & \maps_\ast(BF,B\haut(S^0)).
}
\]
In particular, we get that for a $p$-group $F$, 
\begin{align}\label{eq:p-groupSpheres}
\Sp_F \cong \Z \times \Hom(F,C_2) \times [BF, B\SL1(S^0)].
\end{align}
Thus we need to concentrate on computing $[BF, B\SL1(S^0)]$, which is isomorphic to the unpointed homotopy classes of maps
$\pi_0\maps(BF, B\SL1(S^0))$. See \Cref{rem:point-or-not}.
\end{rem}

\subsection{Interlude on characteristic classes and Steenrod operations}\label{rem:pont-steenrod}
The comparison of $F$-spheres and vector bundles on $BF$ will play a crucial role below, when we specialize $F$ even further to 
an elementary abelian $p$-group and reduce the calculation of $F$-spheres to a cohomology calculation. We will use the theory 
developed by Lannes \cite{Lannes}, giving us that it is sufficient to understand $H^\ast (B\SL1(S^0))/J$ where $J$ is the ideal
of nilpotent elements. As a prelude we pull together what we need from the literature on characteristic classes. 

Let $\xi$ be a stable real vector  bundle of virtual dimension $n$ over a space $X$ and let $M\xi$
be the Thom spectrum. If $p > 2$, we assume $\xi$ is oriented. Let $U \in H^n(M\xi,\FF_p)$
be the Thom class; then the Thom isomorphism is given by the cup product
\[
U \smallsmile (-) = U\cdot (-):H^k(X,\FF_p) \to H^{k+n}(M\xi,\FF_p).
\]
If $p=2$ one can \emph{define} the Steifel-Whitney classes $w_i(\xi)$ using Steenrod operations
\begin{equation}\label{eq:def-swi}
U\cdot w_i(\xi) = \mathrm{Sq}^iU.
\end{equation}
This is the approach in \cite[\S 4]{MS}. Then we have $H^\ast (BO,\FF_2) \cong \FF_2[w_1,w_2,\ldots]$
where $w_i = w_i(\gamma)$, with $\gamma$ the universal stable bundle over $BO$.

If $p > 2$, the relationship between Pontrjagin classes and Steenrod operations is more complicated. Recall that
if $\xi$ is a real vector bundle over $X$, then the Pontrjagin classes are defined in terms of Chern classes:
\[
p_i(\xi) = (-1)^i c_{2i}(\xi \otimes_\RR \CC) \in H^{4i}(X,\ZZ_{(p)}).
\]
Then $H^\ast (BSO, \ZZ_{(p)}) = \ZZ_{(p)}[p_1,p_2,\ldots]$, where $p_i $ is the Pontrjagin class of the universal
bundle over $BSO$. 

If $\gamma_1$ over $BU(1) = BSO(2) = \CC\mathrm{P}^\infty$ is the universal complex line bundle, then
$H^\ast (BU(1),\ZZ_{(p)}) \cong \ZZ_{(p)}[c]$, where $c= c_1(\gamma_1)$. We have $p_1(\gamma_1) = c^2$, and we write
\[
t = c^2 \in H^4(BSO(2),\ZZ_{(p)}).
\]
We then get a map $BSO(2)^n \to BSO$ classifying $\gamma_1^{\times n}$, which on cohomology factors as
\[
\xymatrix{
H^\ast (BSO,\ZZ_{(p)}) \ar[r] & \ZZ_{(p)}[t_1,\cdots,t_n]^{\Sigma_n} \ar[r]^-\subseteq & H^\ast (BSO(2)^n, \ZZ_{(p)}).
}
\]
The first of these maps is an isomorphism in degrees $\ast \leq 4n$; this initiates the classical analysis of Pontrjagin classes using
symmetric polynomials. By the Cartan formula for Pontrjagin classes, each $p_i$ maps to the $i$th elementary symmetric
polynomial in the $t_j$. 

If $\xi$ is an oriented real $n$-bundle over a space $X$, we define a cohomology class
\[
q_n(\xi) \in H^{2n(p-1)}(X,\FF_p)
\]
by
\begin{equation}\label{eq:def-qi}
U\cdot q_n(\xi) = P^nU
\end{equation}
where $P^n$ is the $n$th odd-primary Steenrod operation. By considering the universal case over $X=BSO$, we see that
$q_n(\xi)$ must be a polynomial in the Pontrjagin classes $p_i(\xi)$. More specifically, we have
the following result.

\begin{prop}\label{eq:pont-and-steen-1} Let $p = 2r+1$, so $4r= 2(p-1)$. Then there is a congruence
\begin{equation}\label{eq:pont-and-steen}
q_n(\xi) \equiv (-1)^{n(r+1)} rp_{rn}(\xi)\qquad \mathrm{modulo}\qquad I_{rn}
\end{equation}
where $I_{rn}$ is the ideal generated by the Pontrjagin classes $p_i(\xi)$, $i < rn$. 
\end{prop}

\begin{proof} This is originally due to Wu \cite{WuII}. However the result can be deduced by putting together
various ideas from Milnor-Stasheff \cite{MS}; we now go through this exercise. As usual, it is sufficient to do the universal
example.

First, we assert
\begin{equation}\label{eq:symm-in-q}
1+ q_1 + q_2 + \cdots + q_n + \cdots = (1+t_1^r)\cdots (1+t_{rn}^r)
\end{equation}
in $\FF_p[t_1,\ldots,t_{rn}]^{\Sigma_{rn}} = \FF_p[p_1,\ldots,p_{rn}]$. This is in the proof of
\cite[Theorem 19.7]{MS}, explicitly credited there to Wu.
It can also be seen by observing that since the Steenrod operations have a Cartan formula, we have that
\[
q_i(\xi \times \zeta) = \sum_{j+k=i} q_j(\xi) \times q_k(\zeta).
\]
Then to prove \eqref{eq:symm-in-q} we need only note that $q_1(\gamma_1) = t^r \in H^{4r}(BSO(2),\FF_p)$ and
that $q_i(\gamma_1) = 0$ for $i > 1$. 

Note that the right hand side of \eqref{eq:symm-in-q}
is the reduction of an obvious integral polynomial and we define integral lifts $\overline{q}_i$ of the classes 
$q_i$ by the formula
\[
1+ \overline{q}_1 + \overline{q}_2 + \cdots + \overline{q}_n + \cdots = (1+t_1^r)\cdots (1+t_{rn}^r)
\in \ZZ_{(p)}[t_1,\ldots,t_{rn}]^{\Sigma_{rn}}.
\]
We may equally regard these as elements in  $R := \QQ_{(p)}[t_1,\ldots,t_{rn}]^{\Sigma_{rn}}$. 

To finish the argument, we use a variant of Girard's formula, as in \cite[Problem 16A]{MS}. Apply the logarithm to the formula
\[
1+ p_1 + p_2 + \cdots + p_n = (1+t_1)\cdots (1+t_n)
\]
to get that $p_k \equiv (-1)^{k+1} \sum t_i^k/k$ modulo decomposables in the $p_i$ in $R$. If we then apply the logarithm
to the formula
\[
1+ \overline{q}_1 + \overline{q}_2 + \cdots + \overline{q}_n + \cdots = (1+t_1^r)\cdots (1+t_{rn}^r)
\]
we get $\overline{q}_k \equiv (-1)^{k+1} \sum_{i=1}^{rn} t_i^{rk}/k$ modulo decomposables in the $\overline{q}_i$, so also
modulo decomposables in the $p_i$. This then gives
\[
(-1)^{k(r+1)} r p_{rk} \equiv (-1)^{k(r+1)} (-1)^{rk+1}  r \sum_{I=1}^{rn} t_i^{rk} /{rk} \equiv \overline{q}_k.
\]
This is an integral formula, so we can reduce modulo $p$ to get \eqref{eq:pont-and-steen}
\end{proof}

We now give a construction to realize the classes $q_i$ as generators for the cohomology of a space. This
is a variation on the discussion of \cite[Lecture 4]{AdamsBattelle}. 
After we complete at an odd prime $p$, we have stable Adams operations
\[
\psi^k\colon BU_p \longr BU_p,\qquad k \in \ZZ_p^\times.
\]
The Adams summand of $BU_p$ can be realized as the homotopy fixed points $BU_p^{hC_{p-1}}$ of the cyclic group
$C_{p-1}= \FF_p^\times \subseteq \ZZ_p^\times$. There is an equivalence
\[
BU_p^{hC_{p-1}} \simeq BX(p-1)
\]
where $BX(p-1)$ is the colimit of classifying spaces of certain $p$-compact groups. This space is discussed in detail
by Castellana \cite{Castellana}.

\begin{prop}\label{prop:bxp-random} We have the following calculations in cohomology.
\begin{enumerate}
\item The map $BX(p-1)\simeq BU_p^{hC_{p-1}} \to BU_{p}$ induces an isomorphism
\[
H^\ast (BU,\FF_p)/I \cong H^\ast BX(p-1) \cong \FF_p[c_{(p-1)},c_{2(p-1)},\ldots]
\]
where $I \subseteq H^\ast(BU)$ is the ideal generated by the Chern classes $c_n$ with
$n \not\equiv 0$ modulo $(p-1)$. 

\item Let $f:BSO \to BU$ classify the complexification of the universal bundle. Then
\[
f^\ast: H^\ast (BU,\ZZ_{(p)})   \cong \ZZ_{(p)}[c_1,c_2,\cdots]\to H^\ast BSO \cong \ZZ_{(p)}[p_1,p_2,\cdots]
\]
sends $c_{2n+1}$ to zero and $c_{2n}$ to $(-1)^np_n$. Furthermore $f$ induces an equivalence
\[
BSO_p \simeq BU_p^{hC_2}.
\]

\item There is an equivalence $BX(p-1) \simeq BSO_p^{hC_{(p-1)/2}}$ and an isomorphism of unstable
algebras over the Steenrod algebra
\[
\FF_p[q_1,q_2,\ldots] \cong H^\ast BX(p-1)
\]
where $q_i \in H^{2i(p-1)}BSO$  is the class of \eqref{eq:def-qi} defined using Steenrod operations. 
\end{enumerate}
\end{prop}

\begin{proof} If $X$ is a connected  infinite loop space and $G$ is a finite group of order prime to $p$ acting on $X$, then the map
$X^{hG} \to X$ induces an isomorphism
\begin{equation}\label{eq:coho-orbit-infinite}
H^\ast (X,\FF_p)/I(G) \cong H^\ast (X^{hG},\FF_p)
\end{equation}
where $I(G)$ is the ideal generated by elements of the form $g_\ast x - x$ with $x \in H^\ast (X,\FF_p)$. If found nowhere else,
this can be deduced using the techniques of \cite{Goe}.\footnote{What can be proved easily from \cite{Goe} is that 
$D_\ast (H_\ast X^{hG}) \cong [D_\ast H_\ast X]^G$ where $D_\ast(-)$ is the graded Dieudonn\'e module 
functor. The assertion here
can be deduced from that.}

If $L$ is a line bundle, then $\psi^k(L) = L^{\otimes k}$. Hence, $\psi^k_\ast c_1(L) = kc_1(L)$. The splitting principle
then implies $\psi^k_\ast c_n = k^n c_n$. Thus if $k$ is any integer which generates $\FF_p^\times$, the ideal generated
by $(\psi^k_\ast -1)H^\ast BU$ in $H^\ast BU$ is generated by $c_n$ with $n \not\equiv 0$ modulo $(p-1)$. The
first statement follows. 

For $k = -1$, the ideal  $(\psi^{-1}_\ast -1)H^\ast BU$ in $H^\ast BU$ is generated by the odd-dimensional Chern classes
$c_{2n+1}$. Since $p_n = (-1)^n c_{2n}$, the second statement follows from \eqref{eq:coho-orbit-infinite}.

The final statement is the iteration of homotopy fixed points and \eqref{eq:pont-and-steen}.
\end{proof}

\subsection{Characteristic classes and $H^\ast B\haut(S^0)$}

The following result is crucial for us. At $p=2$ it can be found in \cite{MadsenMilgram}; for $p > 2$ it can 
be found in \cite{Tsuchiya}. Definitive references for both results are \cite[Theorems II.5.1 and II.5.2]{CLM}.

\begin{thm}\label{thm:coho-mod-nil} Let $J \subseteq H^\ast( B\haut(S^0),\FF_p)$ be the ideal of nilpotent elements.

\begin{enumerate}
\item If $p=2$, the map $BSO \to B\SL1(S^0)$ induces an isomorphism 
\[
H^\ast (B\SL1(S^0),\FF_2)/J \cong H^\ast (BSO,\FF_2).
\]

\item If $p > 2$, the composition
\[
BX(p-1)_p = BSO_p^{hC_{(p-1)/2}} \to BSO_p \to B\SL1(S^0)_p
\]
induces an isomorphism
\[
H^\ast (B\SL1(S^0),\FF_p)/J \longr H^\ast (BX(p-1),\FF_p).
\]
\end{enumerate}
\end{thm}

Below, we will use \Cref{thm:coho-mod-nil} as input along with the following result. Let $\cU$ be the category of unstable modules over the Steenrod algebra, and $\cK$ the category of unstable algebras over the Steenrod algebra.

\begin{lem}\label{lem:Lannes-Zarati} Let $g\colon A^\ast \to B^\ast$ be a surjective homomorphism of connected unstable algebras
of finite type over the Steenrod algebra and let $I^\ast$ be the kernel of $g$. If $I^\ast$ consists of nilpotent elements, then $g$ 
induces an isomorphism 
\[
\Hom_{\cK}(B^*,H^\ast BF) \cong \Hom_{\cK}(A^*,H^\ast BF)
\]
for all elementary abelian $p$-groups $F$.
\end{lem}

\begin{proof} It is sufficient to show that
$\Hom_{\cU}(I^\ast,H^\ast BF) = 0$. At the prime $p=2$ this is relatively easy to prove. Indeed, we have that the top Steenrod operation
\[
\Sq_0 = \Sq^n = (-)^2  \colon H^nBF \to H^{2n}BF 
\]
is an injection. On the other hand, for all $x \in I^\ast$ there is a $k$ so that $\Sq_0^k(x) = x^{2^k} = 0$.

If $p > 2$ the argument takes more technology because $H^\ast BF$ itself contains nilpotent elements. First we
have the Carlsson-Miller theorem that  $H^\ast BF$ is an injective object in the category $\cU$. See \cite{Carl}, \cite{Miller},
and \cite[Appendix A ]{LZ}. Second, $\widetilde{\Sigma}H^\ast BF = 0$; see the proof of \cite[Corollaire 7.2]{LZ}. Combining
these two facts with \cite[Proposition 6.1.4]{LZ} gives the result. 
\end{proof}

We now have the following remarkable facts; see \cite[Corollary 4.5]{Castellana}. In sum, we
conclude that if $F$ is an elementary abelian $p$-group, then $F$-spheres are in one-to-one correspondence with
stable  vector bundles over $BF$ and that any such vector bundle is completely determined by its Stiefel-Whitney classes or 
Pontraygin classes depending on the prime. See \Cref{rem:point-or-not}. 

\begin{thm}\label{thm:use-of-T-2} Let $p=2$ and let $F$ be an elementary abelian $2$-group. Then there
is a commutative diagram with all maps isomorphisms of sets 
\[
\xymatrix{
[BF,BO] \rto \dto & \Hom_\cK(H^\ast BO,H^\ast BF) \dto\\
[BF,B\haut(S^0)] \rto  & \Hom_\cK(H^\ast B\haut(S^0),H^\ast BF).
}
\]
\end{thm} 

\begin{thm}\label{thm:use-of-T-p} Let $p>2$ and let $F$ be an elementary abelian $p$-group. Then there
is a commutative diagram with all maps isomorphisms of sets
\[
\xymatrix{
[BF,BX(p-1)] \rto \dto & \Hom_\cK(H^\ast BX(p-1),H^\ast BF) \dto\\
[BF,B\GL1(S^0)] \rto  & \Hom_\cK(H^\ast B\GL1(S^0),H^\ast BF).
}
\]
\end{thm}

\begin{proof} The same argument, with variations, works for both \Cref{thm:use-of-T-2} and \Cref{thm:use-of-T-p}. 
We use \cite[Th\'eor\`eme 0.4]{Lannes}. This result says that if $Y$ is a simply connected space
with $H^\ast Y = H^\ast(Y,\FF_p)$ finite in each degree then the set of unbased homotopy classes of maps from $BF$ to $Y$ is in 
bijection with $\Hom_{\cK}(H^\ast Y, H^\ast BF)$. So, using \Cref{rem:point-or-not}, we have that 
for a simply connected infinite loop space $Y$, the natural map
\[
\xymatrix{
[BF,Y] \rto  & \Hom_\cK(H^\ast Y,H^\ast BF)
}
\]
is a bijection.  The spaces in \Cref{thm:use-of-T-2} and \Cref{thm:use-of-T-p} are not simply connected, so a little more care is 
required. 

When $p$ is odd, we use \cite[Th\'eor\`eme 0.4]{Lannes} with $B\SL1(S^0)$ in place of $B\GL1(S^0)$. Recall \Cref{rem:comp-is-harmless-1} and note that the natural map $B\SL1(S^0) \to B\GL1(S^0)$ induces an isomorphism in mod $p$ cohomology, and as 
$BX(p-1)$ is $p$-complete, any map $BX(p-1)\to B\GL1(S^0)$ factors through $BSL_1(S^0)$.
We get a commutative diagram 
\[
\xymatrix{
[BF,BX(p-1)] \rto \dto & \Hom_\cK(H^\ast BX(p-1),H^\ast BF) \dto\\
[BF,B\SL1(S^0)] \rto  & \Hom_\cK(H^\ast B\SL1(S^0),H^\ast BF),
}
\]
with all maps isomorphisms, due to \Cref{thm:coho-mod-nil} and \Cref{lem:Lannes-Zarati}, giving the claim. 

When $p=2$,
we again appeal to \cite[Th\'eor\`eme 0.4]{Lannes}, with $BSO$ in place of $BO$ as well as $B\SL1(S^0)$ in place of $B\GL1(S^0)$. 
Using  \Cref{thm:coho-mod-nil} and \Cref{lem:Lannes-Zarati} we get a commutative diagram
with all maps isomorphisms of sets 
\[
\xymatrix{
[BF,BSO] \rto \dto & \Hom_\cK(H^\ast BSO,H^\ast BF) \dto\\
[BF,B\SL1(S^0)] \rto  & \Hom_\cK(H^\ast B\SL1(S^0),H^\ast BF).
}
\]
To complete the proof, again use \Cref{rem:comp-is-harmless-1} and the observation that
\[
\Hom(F,C_2) \cong \Hom_\cK(H^\ast BC_2,H^\ast BF).\qedhere
\]
\end{proof}

\begin{thm}\label{thm:whod-have-believed-it-1} Let $F$ be an elementary abelian $p$-group and let $X$ be a
$p$-complete $F$-sphere of virtual dimension $k$. 
Then there is a stable vector
bundle $\xi$ over $BF$ of virtual dimension $k$ and a $p$-equivalence of spectra
\[
M\xi \simeq EF_+ \wedge_F X.
\]
Furthermore there is an $F$-equivalence $X \simeq Y$ of $p$-complete $F$-spheres if and only if 
there is an isomorphism of modules over the Steenrod algebra
\[
H^\ast(EF_+ \wedge_F X) \cong H^\ast(EF_+ \wedge_F Y).
\]
Such an isomorphism uniquely determines the $F$-equivalence up to $F$-homotopy.
\end{thm}

\begin{proof} We have phrased this as a result about modules over the Steenrod algebra, but there is more structure
here that can be of use. Let $\xi$ be a spherical fibration over a base space $B$; we assume $\xi$ is oriented if
$p > 2$.  Then the cohomology $H^\ast M\xi$ is a free module of rank $1$ over $H^\ast B$ on the Thom
class $U$. Thus the Steenrod algebra structure on $H^\ast M\xi$ is completely determined by the Cartan formula and the action
of the Steenrod operations on $U$. Note that if $p$ is odd, the Bockstein vanishes on $U$, because $U$ is the reduction
of an integral class. This action is, in turn, completely determined by the Stiefel-Whitney classes,
as in \eqref{eq:def-swi}, or the modified Pontrjagin classes of \eqref{eq:def-qi} and \eqref{eq:pont-and-steen}.

Our result now follows by combining this observation with \Cref{lem:comp-is-harmless},
\Cref{prop:bxp-random},  and \Cref{thm:use-of-T-2} or \Cref{thm:use-of-T-p}, 
as needed.
\end{proof}

The key result of this section is the following application of \Cref{thm:whod-have-believed-it-1}. Let $\cG$ be a compact $p$-adic analytic group. For any group $H$ let
$Z(H)$ denote its center. Note that $Z(\cG)$ acts trivially on both $S^\gg$ and $\ig$, since both actions
are defined by conjugations. 

\begin{thm}\label{thm:whod-have-believed-it} Let $\cG$ be a compact $p$-adic analytic group and let $H$ be a closed subgroup
of $\cG$ such that $H/H \cap Z(\cG)$ is finite. Suppose the $p$-Sylow subgroup $F$ of $H/H \cap Z(\cG)$
is an elementary abelian $p$-group. Then there is an $H$-equivalence $S^\gg \simeq \ig$  if and only if there there
is an isomorphism of modules over the Steenrod algebra
\begin{equation}\label{eq:iso-steen-thom}
H^\ast(EF_+ \wedge_F S^\gg) \cong H^\ast(EF_+ \wedge_F \ig).
\end{equation}
\end{thm}

\begin{proof} Suppose we are given the isomorphism of \eqref{eq:iso-steen-thom}. Then
\Cref{thm:whod-have-believed-it-1} gives an $F$-equivalence $S^\gg \simeq \ig$. We
apply \Cref{lem:red-to-sylow} and \Cref{rem:these-hyp-hold} to get
an equivalence of $H/H \cap Z(\cG)$-spectra. This is automatically an $H$-equivalence.

For the converse, if $S^\gg \to \ig$ is an $H$-equivalence then, since the 
center acts trivially, it is an $H/H \cap Z(\cG)$-equivalence and, by restriction, 
an $F$-equivalence. \Cref{thm:whod-have-believed-it-1} again applies, giving the claim.
\end{proof}

\begin{rem}\label{rem:center-small} In our main examples, the center of $\cG$ is small. For example,
if $\cG = \SS_n$ is the Morava stabilizer group and $H$ is a finite subgroup, then $Z(\SS_n) = \ZZ_p^\times$, so $H \cap Z(\SS_n)$ is a cyclic
group of order at most $p-1$ if $p > 2$. In the more interesting case of $p=2$, $H \cap Z(\SS_n)$ is either trivial or
$\{ \pm 1\}$. Note that $H$ can be a quaternionic group if $p=2$. For more on the finite subgroups of
$\SS_n$, see \cite{Hewett, Bujard}.

Similarly, for $\cG = \Gl_n(\ZZ_p)$, the center is the scalar matrices: diagonal matrices with all entries equal.
Thus the center is again $\ZZ_p^\times$. Of course, any finite group is isomorphic to a subgroup $\Gl_n(\ZZ_p)$
for some $n$. 
\end{rem}


\section{Two algebraic preliminaries}\label{sec:SStech}

In this section we collect two technical remarks about spectral sequences needed for the proof of the main theorem of the next section.
This should perhaps not be read until after skimming the next section to see why on earth we need such things.
 
We begin with a first quadrant cohomology spectral sequences
\[
E_2^{p,q} \Longrightarrow A^{p+q}
\]
with differentials $d_r \colon E_r^{p,q} \to E_r^{p+r,q-r+1}$. The $E_2$-term is bounded below in $p$ and $q$; that is, 
$E_2^{p,q} = 0$ for $p < 0$ and $q < 0$. This means, in particular, the 
$E_r^{p,q} \cong E_\infty^{p,q}$ if $r > \mathrm{max}(p,q+1)$. The motivating example is the homotopy
orbit spectral sequence
\[
H^p(F,H^qX) \Longrightarrow H^{p+q}(EF_+ \wedge_F X)
\]
for some $F$-spectrum $X$ with the property that $H^qX= 0$ for $q < 0$.

Here is the set-up. By the very nature of our project, we will have a diagram of such spectral sequences for positive integers
$j$ bigger than some fixed integer $N$
\begin{equation}\label{eq:ss-diag-abstract}
\xymatrix@C=30pt{
E_2^{p,q}(j)\dto_{g}  \ar@{=>}[r] & A^{p+q}(j) \dto^{g}\\
E_2^{p,q}(j+1)\dto_{f}  \ar@{=>}[r] & A^{p+q}(j+1) \dto^{f}\\
E_2^{p,q} \ar@{=>}[r] & A^{p+q}
}
\end{equation}
We make this diagram specific in \eqref{eq:ss-diag-specific}. 

Let $K_r^{\ast,\ast}(j)$ be the kernel of $g \colon E_r^{\ast,\ast}(j)  \to E_r^{\ast,\ast}(j+1)$.

\begin{lem}\label{lem:ss-seq} Suppose the map induced by $f$
\[
E_2^{\ast,\ast}(j)/K_2^{\ast,\ast}(j) \longr E_2^{\ast,\ast}
\]
is an isomorphism for $j \geq N$. Then for $j > N+m-1$ the map
\[
A^\ast(j)/\mathrm{Ker}\{A^\ast(j) \to A^\ast (j+m+1)\} \to A^\ast
\]
is an isomorphism in degrees $n \leq m$.
\end{lem}

\begin{proof} We break this proof into a number of steps. In the first two steps we show, by induction, that the map
$f\colon E_r^{\ast,\ast}(j) \to E_r^{\ast,\ast}$ induces an isomorphism for $j \geq N+r-2$
\begin{equation}\label{eq:f-iso}
E_r^{\ast,\ast}(j)/K_r^{\ast,\ast}(j) \longr E_r^{\ast,\ast}.
\end{equation}
We have assumed the base case of $r=2$. The third step will complete the argument. 

{\bf Step 1.} Here we show that
\[
f\colon E_{r+1}^{\ast,\ast}(j) \longr E_{r+1}^{\ast,\ast}
\]
is onto for $j \geq N+(r+1)-2 = N+r-1$. It is sufficient to do the case where $j=N+r-1$.

If we choose an $r$-cycle
in $E_{r}^{\ast,\ast}$, then the induction hypothesis implies there is a class $y \in E_{r}^{\ast,\ast}(N+r-2)$
with $f(y) = x$. We show that $g(y) \in E_r^{\ast,\ast}(N+r-1)$ is an $r$-cycle; this will complete this step.

Since the induced map
\[
f:E_{r}^{\ast,\ast}(N+r-2)/K_r^{\ast,\ast}(N+r-2) \to E_r^{\ast,\ast}
\]
is injective and $d_r(f(y))= d_r(x) = 0$, we have that
\[
d_r(y) \in K_r^{\ast,\ast}(N+r-2)
\]
and hence that  $d_r(g(y))=0$ as needed.

{\bf Step 2.} We now show that
\[
f:E_{r+1}^{\ast,\ast}(j)/K^{\ast,\ast}_{r+1}(j)\ \longr E_{r+1}^{\ast,\ast}
\]
is injective for $j \geq N+r-1$. 

It is equivalent to show that if  a class in $ E_{r+1}^{\ast,\ast}(j)$ maps to zero in $E_{r+1}^{\ast,\ast}$, then it maps to zero in
$E_{r+1}^{\ast,\ast}(j+1)$. Suppose we have $y \in E_r^{\ast,\ast}(j)$ and
\[
f(y) = d_r(w) \in E_r^{\ast,\ast}.
\]
By the induction hypothesis we may choose a class $z \in E_r^{\ast,\ast}(j)$ so that $f(z) = w$. Then
\[
f(y-d_r(z)) = f(y) - d_rf(z) = f(y)-d_r(w) = 0 \in E_r^{\ast,\ast}
\]
so, again by the induction hypotheses $y-d_r(z) \in K_r^{\ast,\ast}(j)$. Hence $y-d_r(z)$ maps to zero in
$E_r^{\ast,\ast}(j+1)$, or
\[
d_r(g(z)) = g(y)
\]
as needed.  

{\bf Step 3.} Here we complete the argument with another inductive procedure. 
It is sufficient to show that for $j > N+m-1$ the map
\[
A^m(j)/\mathrm{Ker}\{A^m(j) \to A^\ast (j+m+1)\} \to A^m
\]
is an isomorphism, for in degrees $n < m$, we have $j > N+m-1 > N+n-1$.

Note that in a particular bidegree $(p,q)$, all our spectral sequences satisfy
$E_r^{p,q} = E_\infty^{p,q}$ for $r > \max\{p,q+1\}$. By the first two steps we have that \eqref{eq:f-iso} 
is an isomorphism for $j \geq N +r - 2$, hence
\begin{align}\label{eq:f-iso-infty}
f\colon E_\infty^{p,q}(j)/K_\infty^{p,q}(j) \longr E_\infty^{p,q}
\end{align}
is an isomorphism for
\[
j > N +  \max\{p,q+1\}  - 2. 
\]

Since we have first quadrant cohomological spectral sequences, we have a diagram of filtrations
\[
\xymatrix@C=15pt{
0 = F^{-1}A^m(j) \ar[d] \ar[r]^-\subseteq & F^0A^m(j)  \ar[d] \ar[r]^-\subseteq & \cdots \ar[r]^-\subseteq &F^{m-1}A^m(j)  \ar[d] 
\ar[r]^-\subseteq & F^mA^m(j)=A^m(j) \ar[d] \\
0 = F^{-1}A^m  \ar[r]^-\subseteq & F^0A^m  \ar[r]^-\subseteq
& \cdots \ar[r]^-\subseteq &F^{m-1}A^m  \ar[r]^-\subseteq & F^mA^m=A^m,
}
\] 
and diagrams of short exact sequences
\begin{equation}\label{eq:i-have-to-explain-1}
\xymatrix{
0 \ar[r] &  F^{q-1}A^m(j) \ar[r]\ar[d]  & F^{q}A^m(j) \ar[r]\ar[d] & E_\infty^{{m-q,q}}(j) \ar[r]\ar[d]  & 0\\
0 \ar[r] & F^{q-1}A^m \ar[r] & F^{q}A^m \ar[r] &E_\infty^{{m-q,q}} \ar[r] & 0.
}
\end{equation}
We will use induction on $q$ to show
\begin{align}\label{eq:filtration-q-iso}
f^q: F^qA^m(j)/\ker\{F^qA^m(j) \to F^qA^m(j+q+1)\} \to {F^q}A^m
\end{align}
is an isomorphism for $j > N+m-1$. The case $q=m$ is what we need to prove; the case $q=0$ is the assertion
that we have an isomorphism
\[
\xymatrix{
E_\infty^{m,0}(j)/K_\infty^{m,0}(j) \ar[r]^-\cong & E_2^{m,0},
}
\]
which has already been proved and in fact requires only that $j > N+m-2$. 

Now assume inductively that $f^{q-1}$ is an isomorphism for $j > N+m-1$. 
To complete the induction step we examine the exact sequence of \eqref{eq:i-have-to-explain-1}. The inductive hypothesis
gives that $F^{q-1} A^m(j) \to F^{q-1}A^m$ is onto for $j > N+m-1$, hence by the snake lemma
we have a short exact sequence
\[
0 \to \kappa_1 \to \kappa_2\to \kappa_3 \to 0,
\]
where the $\kappa_i$ are the kernels of the three vertical maps in \eqref{eq:i-have-to-explain-1}.
Note that since also the outer two vertical map in \eqref{eq:i-have-to-explain-1} are onto, the snake lemma implies
that the middle map is onto as well.

It remains to identify $\kappa_2= \ker\{F^qA^m(j) \to F^qA^m \}$ with
\[
\kappa_2'= \ker\{F^qA^m(j) \to F^qA^m (j+q+1) \}.
\]
By construction, $\kappa_2'$ is contained in $\kappa_2$; thus, we just need to show the reverse containment.
Let $x $ be an element of $\kappa_2 \subseteq F^qA^m(j)$. Its image in $E^{m-q,q}_\infty$ is zero, and
since $\kappa_3 \cong \ker\{ E_\infty^{m-q,q}(j) \to E_\infty^{m-q,q}(j+1) \}$, its image in $E_\infty^{m-q,q}(j+1)$ is zero. Hence, from 
the exact sequence
\[
0 \to F^{q-1}A^m(j+1) \to F^{q}A^m(j+1) \to E_\infty^{{m-q,q}} (j+1) \to 0,
\]
it lifts to an element $y \in F^{q-1}A^m(j+1) $, with the property that $y$ maps to zero in $F^{q-1}A^m$; that is,
it is in $\ker\{F^{q-1} A^m(j+1) \to F^{q-1}A^m \}$. By the inductive hypothesis, this kernel is
$\ker\{F^{q-1}A^m(j+1) \to F^{q-1}A^m(j+1+q)  \}$. In particular, the image of $y$ in $F^qA^m(j+q+1)$ is zero;
but this is the same as the image of $x$ in $F^qA^m(j+q+1)$; hence, $x$ is in the kernel $\kappa_2'$.
\end{proof}

Our second algebraic result is about building isomorphisms from arrays of graded abelian groups. 

\begin{rem}\label{rem:chase-diag} Suppose we are given a commutative diagram of graded abelian groups
\[
\xymatrix{
& B_{i+1,j} \ar[r] \ar[d]&  B_{i,j} \ar[r] \ar[d]&  B_{i-1,j}  \ar[d]\\
& B_{i+1,j+1} \ar[r] \ar[d]&  B_{i,j+1} \ar[r] \ar[d]& B_{i-1,j+1} \ar[d]\\
A \ar[r]& A_{i+1} \ar[r] &  A_{i} \ar[r] &  A_{i-1}
}
\]
This is an infinite array, with $i \geq 1$ and $j > i$, but we haven't put in the dots because it makes for a 
very cluttered diagram. In our main examples, we will have
\begin{align*}
A &= H^\ast(EF_+ \wedge_F \ig),\\
A_i &= H^\ast(EF_+ \wedge_F \Sigma_+^\infty B\Gamma_i),\\
B_{i,j} &= H^\ast(EF_+ \wedge_F \Sigma_+^\infty B(\Gamma_i/\Gamma_{j})).
\end{align*}
For integers $i$, $j$, $d$, and $n$, let
\begin{align*}
I(i,d) &= \mathrm{Ker}\{A_i \to A_{i-d}\},\\
J(i,j,m) &= \mathrm{Ker}\{B_{i,j}\to B_{i,j+m}\},\\
\end{align*}

We make the following assumptions, which will have to be justified in our main examples. See
\Cref{lem:isos-are-fin-F-1} and \Cref{lem:isos-are-fin-F-2}. The stipulation that $i > d+1$ arises there,
and $d$ will be the rank of $\cG$. 

\begin{enumerate}

\item For all $i > d+1$, the composition $A \to A_i \to A_i/I(i,d)$ is an isomorphism of graded abelian groups; and,

\item there is an integer $N$, so that for all $i > d+1$ and all $j > N+m-1$ the induced map
$B_{i,j}/J(i,j,m) \to A_i$ is an isomorphism in degrees $n \leq m$.
\end{enumerate}

If we fix $i$ and $j$, let $J_m$ be the kernel
of either of the two ways around the diagram
\[
\xymatrix{
B_{i,j} \ar[r] \ar[d]& B_{i-d,j} \ar[d] \\
B_{i,j+m} \ar[r]& B_{i-d,j+m}. 
}
\]
Note we could write $J(i,j,d,m)$ for $J_m$ but $m$ will be the crucial index. Then we can conclude

\begin{enumerate}

\item[(3)] For all $i > d+1$ and all $j > N+m-1$ the map $\xymatrix{B_{i,j}/J_m \ar[r] & A_i/I(i,d)}$ is an isomorphism
in degrees $n \leq m$.
\end{enumerate}

The proof is a diagram chase. We then can conclude we have maps
\begin{equation}\label{eq:trunc-iso}
\xymatrix{
A \ar[r]^-\cong & A_i/I(i,d) &B_{i,j}/J_m \ar[l] 
}
\end{equation}
If $i > d+1$ and $j > N+m-1$ the first map is an isomorphism and the second map is an isomorphism in degrees $n \leq m$.
\end{rem}


\section{Equivalences of $\cG$-spheres for finite subgroups}\label{sec:squeenrodstairs}

The goal of this section is to show that for suitable closed subgroups $H$ of our group $\cG$ we have an equivalence of
$H$-spectra $\ig \simeq S^\gg$. This is accomplished in \Cref{thm:a-big-one}, as corollary of the
main calculation of this section: we will show in \Cref{thm:main-calc} that for any finite subgroup $F \subseteq \cG/Z(\cG)$, we have an 
isomorphism of modules over the Steenrod algebra
\[
H^\ast (EF_+ \wedge_F \ig) \simeq H^\ast (EF_+ \wedge_F S^\gg). 
\]
Then we appeal to \Cref{thm:whod-have-believed-it} to prove  \Cref{thm:a-big-one}.

\def\cM{{{\mathcal{M}}}}

The strategy is as follows. Fix a finite subgroup $F \subseteq \cG/Z(\cG)$. Since $\cG$ acts by conjugation on $\Gamma_i$ and
$\gg_i$, $F$ acts on all of the spaces $B\Gamma_i$, $B(\Gamma_i/\Gamma_j)$, $B\gg_i$, and so on.

Recall we have compatible transfer maps $\tr: \Sigma_+^\infty B\Gamma_i \to
\Sigma_+^\infty B\Gamma_{i+t}$, for $t>0$ by \eqref{eq:tr-gammas}, and hence maps
\[
r:  \Sigma_+^\infty B\Gamma_i \longr \colim_{\tr} \Sigma_+^\infty B\Gamma_{i+t} = \ig.
\]
We then have a diagram of $F$-spectra
\[
\xymatrix{
\Sigma^\infty_+ B(\Gamma_i/\Gamma_{j}) & \ar[l]_-q \Sigma^\infty_+ B\Gamma_i \ar[r]^-r & \ig.
}
\]
There is a corresponding diagram for $S^{\gg}$
\[
\xymatrix{
\Sigma^\infty_+ B(\gg_i/\gg_{j}) & \ar[l]_-q \Sigma^\infty_+ B\gg_i\ar[r]^-r & S^{\gg}.
}
\]
Finally, there is a group isomorphism $\Gamma_i/\Gamma_j \cong \gg_i/\gg_j$ for $j\leq 2i$, as per \eqref{eq:gg-gamma}. Using the 
techniques of  \Cref{sec:SStech}, we can put together the comparison we need. The key intermediate results
are \Cref{thm:main-calc1} and \Cref{rem:main-calc1-gg}.

In the final applications we will only need $i$ large, so while more generality is possible, all our preliminary lemmas
will set $i > d+1$, where $d \geq 1$ is the rank of $\cG$ and, hence of $\Gamma_i$. This will avoid having to spell out
special cases, especially when $p=2$. As usual, all homology and cohomology in this section is with $\FF_p$-coefficients. 

\begin{notation}\label{not:homotopy-orbits} In this section we will be working heavily with homotopy orbits and
it is convenient to shorten the notation. If $Y$ is a $G$-spectrum for some finite group $G$ we will often write 
\[
Y_{hG} = EG_+ \wedge_G Y.
\]
Hence if $X$ is a $G$-space
\[
[\Sigma^\infty_+ X]_{hG} = EG_+ \wedge_G \Sigma^\infty_+ X.
\]
\end{notation}

\begin{notation}\label{not:all-the-kers} In what follows we will need to have names for the kernels of various restriction
and transfer maps. We will recall the definitions as needed, but we collect them here to put the confusion
in a place where it can be organized. Compare \Cref{rem:chase-diag} as well. 

First, there are the kernels of maps before taking homotopy orbits; they are decorated with subscript $0$:
\begin{align*}
I_0(i) &= \mathrm{ker}\{\tr^\ast:H^\ast B\Gamma_i \longr H^\ast B\Gamma_{i-1}\},\\
J_0(i,j)&=  \mathrm{ker}\{\res^\ast:H^\ast B\Gamma_i/\Gamma_j \longr H^\ast B\Gamma_i/\Gamma_{j+1}\}.
\end{align*}
Next, there are the kernels of maps after taking homotopy orbits; they do not have a subscript, but acquire a new index
because we shift more than one step:
\begin{align*}
I(i,m) &= \mathrm{ker}\{\tr^\ast:H^\ast ([\Sigma_+^\infty B\Gamma_i]_{hF} \longr H^\ast [\Sigma_+^\infty B\Gamma_{i-m}]_{hF}\},\\
J(i,j,m)&=  \mathrm{ker}\{\res^\ast:H^\ast [\Sigma_+^\infty B\Gamma_i/\Gamma_{j}]_{hF} \longr
H^\ast [\Sigma_+^\infty B\Gamma_i/\Gamma_{j+m}]_{hF} \}.
\end{align*}
There is one more kernel, which combines restrictions and transfers. Fixing $i$ and $j$, let $J_m$ be the kernel
of either of the two ways of composing around the commutative diagram
\[
\xymatrix{
H^\ast [\Sigma_+^\infty B\Gamma_i/\Gamma_{j}]_{hF} \ar[d]^{\res^\ast} \ar[r]^-{\tr^\ast}&
H^\ast [\Sigma_+^\infty B\Gamma_{i-d}/\Gamma_{j}]_{hF} \ar[d]^{\res^\ast}\\
H^\ast [\Sigma_+^\infty B\Gamma_{i}/\Gamma_{j+m}]_{hF} \ar[r]_-{\tr^\ast}\ &
H^\ast [\Sigma_+^\infty B\Gamma_{i-d}/\Gamma_{j+m}]_{hF}\ .\\
}
\]
We could write $J(i,j,d,m)$ for $J_m$ but $m$ will be the crucial index. 
Finally, there are the kernels $K_r^{\ast,\ast}(j)$ in \Cref{lem:ss-seq}. These are distinct, but related, and will also be recalled as
needed. 
\end{notation}

We begin with the following algebraic result.  

\begin{lem}\label{lem:isos-are-fin} 
\begin{enumerate}
\item Let $I_0(i)$ be the kernel of the map
$\mathrm{tr}^\ast: H^\ast B\Gamma_{i} \longr H^\ast B\Gamma_{i-1}$. 
For all $i > d+1$ the composition
\[
\xymatrix{
H^\ast \ig \ar[r]^-{r^\ast} & H^\ast B\Gamma_{i} \ar[r] & H^\ast B\Gamma_{i}/I_0(i)
}
\]
is an isomorphism of modules over the Steenrod algebra.
\item Let $J_0(i,j)$ be the kernel of the restriction $H^\ast (B\Gamma_{i}/\Gamma_{j}) \to H^\ast (B\Gamma_{i}/\Gamma_{j+1})$.
Then for all $i >  d+1$ and $j \geq i+1$ the map of modules over the Steenrod algebra
\[
H^\ast (B\Gamma_{i}/\Gamma_{j})/J_0(i,j) \longr H^\ast B\Gamma_{i}
\]
is an isomorphism. 
\end{enumerate}
\end{lem}

\begin{proof} Part (1) follows from part (1) of \Cref{lem:transfer-hom} and part (2) is a consequence
of part (2) of \Cref{lem:basic-coh}. 
\end{proof}

Here is the input for an application of \Cref{lem:ss-seq} and \Cref{rem:chase-diag}. The following diagram 
makes \eqref{eq:ss-diag-abstract} concrete; all the spectral sequences are homotopy orbit spectral sequences. 
\begin{equation}\label{eq:ss-diag-specific}
\xymatrix@C=30pt{
H^p(F,H^qB\Gamma_i/\Gamma_j)\dto_{g}  \ar@{=>}[r] &
H^{p+q}[\Sigma^\infty_+ B\Gamma_i/\Gamma_j]_{hF} \dto^{g}\\
H^p(F,H^qB\Gamma_i/\Gamma_{j+1})\dto_{f}  \ar@{=>}[r] &
H^{p+q}[\Sigma^\infty_+ B\Gamma_i/\Gamma_{j+1}]_{hF} \dto^{f}\\
H^p(F,H^qB\Gamma_i) \ar@{=>}[r] & H^{p+q}[\Sigma^\infty_+ B\Gamma_i]_{hF}
}
\end{equation}

We continue to write $d$ for the rank of $\cG$ and, hence, of $\Gamma_i$. See also \Cref{fig:squeenrodstairs} below.

\begin{lem}\label{lem:isos-are-fin-F-1} Let $I(i,d)$ the be the kernel of the map induced by the transfer
\[
H^\ast ([\Sigma^\infty_+ B\Gamma_{i}]_{hF})
\longr H^\ast ([\Sigma^\infty_+ B\Gamma_{i-d}]_{hF}).
\]
For all $i > d+1$ the composition
\[
H^\ast ([ \ig]_{hF}) \longr H^\ast ([\Sigma^\infty_+ B\Gamma_{i}]_{hF})
\longr H^\ast ([\Sigma^\infty_+ B\Gamma_{i}]_{hF})/I(i,d)
\]
is an isomorphism of modules over the Steenrod algebra.
\end{lem}

\begin{proof} By part (1) of \Cref{lem:isos-are-fin} we have that the maps
\[
H^\ast(\ig) \longr \mathrm{Im}\{\tr^\ast \colon H^\ast B\Gamma_{i+1} \to H^\ast B\Gamma_{i}\} \longr H^\ast B\Gamma_i/I_0(i)
\]
are isomorphisms. Furthermore $H^\ast(\ig) \cong \FF_p$ is concentrated in degree $d$, the map $H^d(\ig) \to H^d B\Gamma_{i}$
is an isomorphism, and $H^n B\Gamma_{i} = 0$ if $n > d$ by \Cref{thm:coh-gammai}.

The spectral sequence
\[
E_2^{p,q}(\ig) = H^p(F,H^q(\ig)) \Longrightarrow H^{p+q}([\ig]_{hF})
\]
has the property that $E_2^{p,q} = 0$ unless $q=d$. Thus it collapses at $E_2$. 

The map of spectral sequences 
\[
\xymatrix{
E_2^{p,q}(\ig) = H^p(F,H^q(\ig)) \ar[d] \ar@{=>}[r] &H^{p+q}([\ig]_{hF})\ar[d] \\
E_2^{p,q}(\Sigma^\infty_+B\Gamma_i) = H^p(F,H^q\Sigma^\infty_+B\Gamma_i) \ar@{=>}[r] &
H^{p+q}([\Sigma^\infty_+B\Gamma_i]_{hF})
}
\]
induces isomorphisms
\[
E_\infty^{p,d}(\ig) \cong E_\infty^{p,d}(\Sigma^\infty_+B\Gamma_i)
\]
as $E_2^{p,q}(\Sigma^\infty_+B\Gamma_i)= 0$ for $q > d$. In addition
if $x \in H^\ast ([ \Sigma^\infty_+B\Gamma_i]_{hF})$ is detected by $a \in E_\infty^{p,q}(\Sigma^\infty_+B\Gamma_i)$
with $q < d$, then
\[
\tr^\ast (a) = 0 \in E_\infty^{p,q}(\Sigma^\infty_+B\Gamma_{i-1});
\]
hence,
\[
\tr^\ast (x) \in H^\ast ([ \Sigma^\infty_+B\Gamma_{i-1}]_{hF})
\]
is detected by a class $E_\infty^{p+s,q-s}(B\Gamma_{i-1})$ with $s > 0$. Thus if we apply the transfer $d$ times
the class $x$ will be sent to zero, as needed. 
\end{proof}

\begin{figure}
\center
\includegraphics[width=\textwidth]{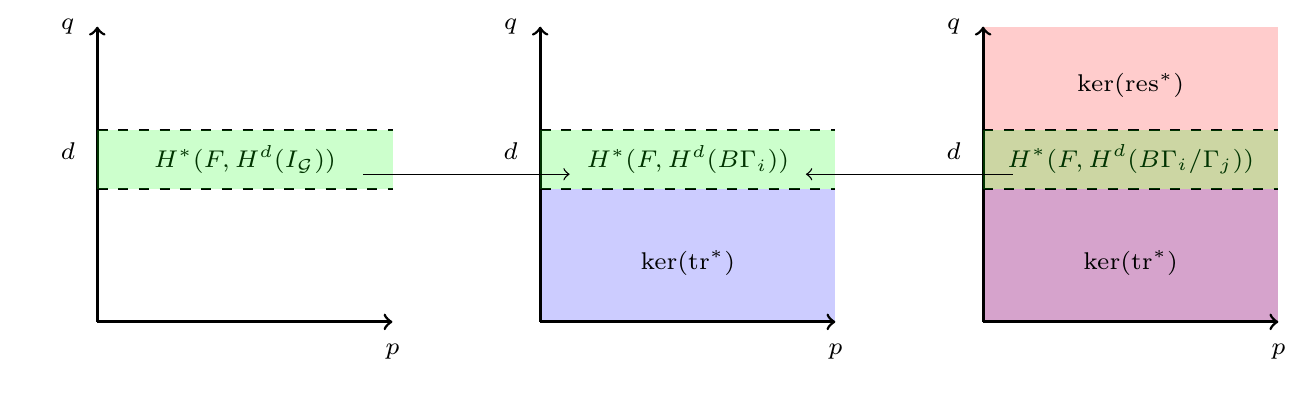}
\captionsetup{width=\textwidth}
\caption{}
\label{fig:squeenrodstairs}
\end{figure}

\begin{lem}\label{lem:isos-are-fin-F-2}  Let $J(i,j,m)$ be the kernel of the restriction map
\[
H^\ast ([\Sigma^\infty_+ B\Gamma_{i}/\Gamma_{j}]_{hF}) \to
H^\ast ([\Sigma^\infty_+ B\Gamma_{i}/\Gamma_{j+m}]_{hF}).
\]
Let $i > d+1$ and let $j > i+m$. Then the homomorphism of modules over the Steenrod algebra
\[
H^\ast ([\Sigma^\infty_+ B\Gamma_{i}/\Gamma_{i+j}]_{hF})/J(i,j,m) \longr
H^\ast ([\Sigma^\infty_+ B\Gamma_{i}]_{hF})
\]
is an isomorphism in degrees $n \leq m$.
\end{lem}

\begin{proof} This is a direct application of \Cref{lem:ss-seq}, with $N= i+1$, and part (2) of \Cref{lem:isos-are-fin}.  

To use \Cref{lem:ss-seq} we must justify the hypotheses of that result. We let
\begin{align*}
E_2^{*,*}(j) &= H^*(F, H^*(B\Gamma_i/\Gamma_j)) \Longrightarrow H^{*}
([\Sigma^\infty_+ B\Gamma_{i}/\Gamma_{j}]_{hF}) = A^{*}(j),\\
E_2^{*,*} &= H^*(F, H^*(B\Gamma_i) )\Longrightarrow H^{*} ([\Sigma^\infty_+ B\Gamma_{i}]_{hF}) = A^{*}, \text{ and} \\
K_2^{*,*}(j) &= \mathrm{Ker}\{E_2^{*,*}(j) \longr E_2^{*,*}(j+1). \}
\end{align*}
We need to check that we have an isomorphism
\[
E_2^{\ast,\ast}(j)/K_2^{\ast,\ast}(j) \longr E_2^{\ast,\ast}.
\]

Suppose $j \geq i+1=N$. From \Cref{lem:isos-are-fin}, we have exact sequences

\[
\xymatrix{
0 \ar[r] & J_0(i,j) \ar[r] \ar[d]^-{=0}&  H^\ast (B\Gamma_{i}/\Gamma_{j}) \ar[d]
\ar[r] &H^\ast B\Gamma_{i} \ar[r] \ar[d]^-{\cong} \ar@/_1pc/@{.>}[l]& 0\\
0 \ar[r] & J_0(i,j+1) \ar[r] &  H^\ast (B\Gamma_{i}/\Gamma_{j+1}) \ar[r]
&H^\ast B\Gamma_{i} \ar[r]  \ar@/_1pc/@{.>}[l] & 0 .
}
\]
By \Cref{thm:coh-gammai} and \Cref{lem:basic-coh} we have algebra splittings, as indicated, and these splitting respect
the $F$-action. Indeed, $H^*(B\Gamma_i) \cong \Lambda(V_i)$, where 
\[V_i \cong H^1(B\Gamma_i) \cong H^1(B\Gamma_{i}/\Gamma_{j}).\] 
The action of $F$ preserves $V_i$ for degree reasons and thus $\Lambda(V_i)$ since $F$ acts by algebra 
homomorphisms. These short exact sequences would normally give long exact sequences in group cohomology, but
the splittings reduce these to the following short exact sequences:
\[ 
\xymatrix{
0 \to H^*(F,J_0(i,j)) \ar[r] \ar[d]_-{0} & H^*(F,H^\ast (B\Gamma_{i}/\Gamma_{j})) \ar[r] \ar[d] & 
H^*(F, H^\ast B\Gamma_{i}) \ar[d]^-{\cong}\to 0 \\ 
 0 \to H^*(F,J_0(i,j+1)) \ar[r] & H^*(F,H^\ast (B\Gamma_{i}/\Gamma_{j+1}))  \ar[r] &  H^*(F, H^\ast B\Gamma_{i}) \to 0
 }
 \]
The left most vertical map is zero as the map
\[
J_0(i,j) =  \mathrm{ker}\{\res^\ast:H^\ast B\Gamma_i/\Gamma_j \longr H^\ast B\Gamma_i/\Gamma_{j+1}\}\longr
H^\ast B\Gamma_i/\Gamma_{j+1}
\]
 is zero.
An easy diagram chase implies that $K_2^{\ast,\ast}(j)$ is the kernel of
\[
H^*(F,H^\ast (B\Gamma_{i}/\Gamma_{j})) \to H^*(F, H^\ast B\Gamma_{i}).
\]
The splitting above implies that this map is surjective.
\end{proof}

Applying \Cref{lem:isos-are-fin-F-1}, \Cref{lem:isos-are-fin-F-2}, and the
isomorphism of \eqref{eq:trunc-iso} we get the following result. Fix $i$ and $j$ and let $J_m$ be the kernel
of either of the two ways around the diagram
\[
\xymatrix{
H^\ast [\Sigma_+^\infty B\Gamma_i/\Gamma_{j}]_{hF} \ar[d]^{\res^\ast} \ar[r]^-{\tr^\ast}&
H^\ast [\Sigma_+^\infty B\Gamma_{i-d}/\Gamma_{j}]_{hF} \ar[d]^{\res^\ast}\\
H^\ast [\Sigma_+^\infty B\Gamma_{i}/\Gamma_{j+m}]_{hF} \ar[r]_-{\tr^\ast}\ &
H^\ast [\Sigma_+^\infty B\Gamma_{i-d}/\Gamma_{j+m}]_{hF}\ .\\
}
\]

\begin{prop}\label{thm:main-calc1} Let $i > d+1$ and $j > i+m$.
Let $I(i,d)$ the be the kernel of the map induced by the transfer
\[
H^\ast ([\Sigma^\infty_+ B\Gamma_{i}]_{hF})
\longr H^\ast ([\Sigma^\infty_+ B\Gamma_{i-d}]_{hF}).
\]
Then the maps
\[
\xymatrix{
\Sigma^\infty_+ B(\Gamma_i/\Gamma_{i+j}) & \ar[l]_-q \Sigma^\infty_+ B(\Gamma_i) \ar[r]^-r & \ig
}
\]
define homomorphisms of modules over the Steenrod algebra
\[
\xymatrix{
H^\ast  ([\ig]_{hF}) \ar[r]^-\cong &
H^\ast ([\Sigma^\infty_+ B\Gamma_{i}]_{hF})/I(i,d)&
\ar[l] H^\ast ([\Sigma^\infty_+ B\Gamma_{i}/\Gamma_{j}]_{hF})/J_m.
}
\]
The first map is an isomorphism and the second map is an isomorphism in cohomological degreee
$n$ with $n \leq m$.
\end{prop}

The same argument which proved \Cref{thm:main-calc1} can be immediately adapted to prove the following. 

\begin{prop}\label{rem:main-calc1-gg} Let $i > d+1$ and $j > i+m$.
Let $I^\mathfrak{a}(i,d)$ the be the kernel of the map induced by the transfer
\[
H^\ast ([\Sigma^\infty_+ B\gg_{i}]_{hF})
\longr H^\ast ([\Sigma^\infty_+ B\gg_{i-d}]_{hF}).
\]
Let $J^\mathfrak{a}_m$ denote the kernel of the map
\[
\res^\ast\tr^\ast = \tr^\ast\res^\ast: H^\ast ([\Sigma^\infty_+ B\gg_{i}/\gg_{j}]_{hF}) \to
H^\ast ([\Sigma^\infty_+ B\gg_{i-d}/\gg_{j+m}]_{hF}).
\]
Then the maps
\[
\xymatrix{
\Sigma^\infty_+ B(\gg_i/\gg_{i+j}) & \ar[l]_-q \Sigma^\infty_+ B(\gg_i) \ar[r]^-r & S^\gg
}
\]
define homomorphisms of modules over the Steenrod algebra
\[
\xymatrix{
H^\ast  ([S^\gg]_{hF}) \ar[r]^-\cong &
H^\ast ([\Sigma^\infty_+ B\gg_{i}]_{hF})/I^\mathfrak{a}(i,d)&
\ar[l] H^\ast ([\Sigma^\infty_+ B\gg_{i}/\gg_{j}]_{hF})/J^\mathfrak{a}_m.
}
\]
The first map is an isomorphism and the second map is an isomorphism in cohomological degreee
$n$ with $n \leq m$.
\end{prop}

We now come to our main calculation. Consider the following diagram of modules over the Steenrod algebra 
with $i < j \leq 2i$ and $j > i+m$.
\begin{equation}\label{eq:big-diagram}
\xymatrix{
H^\ast ([B\Gamma_i/\Gamma_{j}]_{hF})/J_m \ar[r]^-\cong \ar[d]&
H^\ast ([B\gg_i/\gg_j]_{hF})/J^\mathfrak{a}_m\ar[d]\\
H^\ast ([B\Gamma_i]_{hF})/I(i,j) &
H^\ast ([B\gg_i]_{hF})/I^\mathfrak{a}(i,j)\\
H^\ast ([\ig]_{hF}) \ar[u]_\cong&
H^\ast ([S^\gg]_{hF}) \ar[u]_\cong
}
\end{equation}
The horizontal isomorphism comes from the isomorphism of groups
\[
\Gamma_{i}/\Gamma_{j} \cong \gg_i/\gg_j
\]
discussed in \eqref{eq:gg-gamma}. The upwards vertical maps are isomorphisms by \Cref{thm:main-calc1} and
\Cref{rem:main-calc1-gg}. The same results show that the downward  vertical maps are isomorphisms in degrees $n$
with $n \leq m$.

\begin{thm}\label{thm:main-calc} Let $F \subseteq \cG/Z(\cG)$ be a finite subgroup. 
Then the maps of \eqref{eq:big-diagram} define an isomorphism of modules
over the Steenrod algebra
\[
 H^\ast  (EF_+\wedge_F \ig) \cong H^\ast (EF_+ \wedge_F S^\gg).
\]
\end{thm}

\begin{proof} We look at \eqref{eq:big-diagram}. In any given range of degrees up to an integer $m$, we may choose $i$ and $j$
so that $i < j \leq 2i$ and $j > i+m$.
\end{proof}

Combining \Cref{thm:whod-have-believed-it} and \Cref{thm:main-calc} immediately implies our key result. 

\begin{thm}\label{thm:a-big-one} Let $\cG$ be compact $p$-adic analytic group and let $H$ be a closed subgroup
of $\cG$ such that $H/H \cap Z(\cG)$ is finite. Suppose the $p$-Sylow subgroup of $H/H \cap Z(\cG)$
is an elementary abelian $p$-group. Then there is an 
$H$-equivariant equivalence
\[
\ig \simeq S^\gg.
\]
\end{thm}


\section{Analyzing the linear action}\label{sec:anal-lin}

In this section we write down a general result that allows us to use linear algebra to analyze the equivariant homotopy
type of $S^\gg$ for finite subgroups. The main result is \Cref{prop:recog-princ}, but before getting there
we need an intermediate result about the transfer for finite covering spaces of equivariant manifolds, namely \Cref{prop:eq-collapse-stable}. This
can be proved by putting together ideas from standard sources such as \cite{Bredon} and \cite{LMS}, but it is easy enough
and clearer to be completely explicit. 

\subsection{Equivariant geometric transfer}
Fix a finite group $F$ and let
\[
q: M \longr N
\]
be a $F$-equivariant differentiable finite-sheeted cover of a closed $F$-manifold $N$. To be clear:
\begin{enumerate}

\item $M$ and $N$ are closed $C^\infty$-manifolds of dimension $d$ with a differentiable $F$-action;

\item $q$ is an $F$-equivariant differentiable map and a finite covering map.
\end{enumerate}
If $M$ is a differentiable manifold, let $TM$ denote the tangent bundle and $T_mM$ the fiber of $TM$ at
$m \in $M.

\begin{rem}[{\bf Equivariant Geometric Transfers}]\label{rem:eq-transfer} Let $\cU=\{U_i\}$, $i \in A$,
be a finite open cover of $M$ by open subsets $U_i \subseteq M$ with the following properties:
\begin{enumerate}

\item for all $i$ the restriction of the covering map to $q:U_i \to N$ is an open embedding defining a
diffeomorphism onto its image; and,

\item for all $g \in F$ and all $i \in A$, there is a $k \in A$ so the action by $g$ on $M$ restricts to a diffeomorphism
$g:U_i \to U_k$.
\end{enumerate}
Note that part (2) defines an action of $F$ on $A$, and hence
a permutation representation $W = \RR^A$.

Next, let $\phi_i$ be a partition of unity subordinate to the cover $\cU$ with the following equivariance property: if $g \in H$ and
$g:U_i \to U_k$ is as in Part (2), then 
\[
\phi_k(gx) = \phi_i(x).
\]
The existence of such partitions of unity can be found in \S III.6 of \cite{Bredon}. The map 
\begin{align*}
j:M &\longr W \times N\\
x &\mapsto  ((\phi_i(x)),q(x))
\end{align*}
is an $F$-equivariant differentiable embedding and we have a commutative diagram of $F$-maps
\[
\xymatrix{
M \ar[r]^-j \ar[dr]_q & W \times N \ar[d]^{p_2}\\
&N.
}
\]
We have an isomorphism of the $F$-equivariant tangent bundles $q^\ast T_N \cong T_M$; this a property of covering
spaces. For any base $F$-space $Y$ let $\theta_W$ be the trivial $F$-bundle of $Y$ with total space $Y \times W$
and diagonal $F$-action. Then we have an isomorphism of $F$-bundles
$T_{W \times N} \cong \theta_W \oplus p_2^\ast T_N$. Thus we can conclude that
\[
j^\ast T_{W \times N} \cong \theta_W \oplus T_M.
\]
It follows that the normal bundle of $M$ in $W \times N$ is isomorphic to the trivial bundle $\theta_W$ over
$M$. Choose an equviariant tubular neighborhood $\nu$ of $M$ in $W \times N \subseteq S^W \times N $. 
(See Theorem IV.2.2 of \cite{Bredon}.) Then one model for the transfer is the Thom collapse map
\begin{align}\label{eq:model-trans}
S^W \wedge N_+ \cong &(S^W \times N)/(\{\infty\} \times N)\\
&\longr (S^W \times N)/((S^W \times N)-\nu) \cong \overline{\nu}/\partial{\nu} = S^W \wedge M_+.\nonumber
\end{align}
This last isomorphism uses that the normal bundle is trivial, so there is an equivariant
diffeomorphism
\begin{equation}\label{eq:triv-tube-1}
W \times M \cong \nu.
\end{equation}
The collapse map of \eqref{eq:model-trans} is an unstable model for the transfer,
so we write $\tr_W: S^W \wedge N_+ \to S^W \wedge M_+$ for this map.
After further suspension we can cancel the representation sphere $S^W$
and get the stable equivariant transfer map
\[
\tr\colon \Sigma^\infty_+ N \longr \Sigma^\infty_+ M.
\]
This is independent of the choices. 
\end{rem}

The next result is a special case of the equivariant tubular neighborhood theorem. See Theorem VI.2.2 of
\cite{Bredon}.

\begin{lem}\label{lem:eq-top} Let $M$ be a smooth manifold with a differentiable action by a finite group $F$.
Let $m \in M$ be a fixed point for the action. Then there exists an $F$-invariant open neighborhood $U$ of $m$
and an $F$-equivariant diffeomorphism $f\colon U \cong T_mM$ taking $m$ to $0$.

Furthermore, the closure $\overline{U}$ of $U$ is homeomorphic to a closed ball and the map $f$ extends to an
$F$-equivariant collapse map
\begin{equation}\label{eq:collapse}
f_U: M \to M/(M-U) \cong \overline{U}/\partial\overline{U} \cong S^{T_mM}.
\end{equation}
\end{lem}

If $q:M \to N$ is our
covering map, then $q(m) \in N$ is also an $F$-fixed point. We may suppose further 
that $q$ maps $U$ and $\overline{U}$ diffeomorphically onto the images $V$ and $\overline{V}$ in $N$. 
Note that $q:M \to N$ defines an isomorphism of representations $T_mM \cong T_{q(m)}N$.

\begin{prop}\label{prop:eq-collapse-stable} There is an $F$-equivariant commutative diagram of spectra
\begin{equation*}
\xymatrix{
\Sigma_+^\infty N \ar[r]^\tr \ar[d]_{\Sigma_+^\infty\wedge f_V}& \Sigma_+^\infty M\ar[d]^{\Sigma_+^\infty\wedge f_U}\\
S^{T_{q(m)}N} \ar[r]_\simeq & S^{T_mM}
}
\end{equation*}
where $f_U$ and $f_V$ are the equivariant collapse maps to the top cell around the fixed points. 
\end{prop}

\begin{proof} We claim there is a commutative diagram of $F$-spaces
\begin{equation}\label{eq:eq-collapse-unstable-1}
\xymatrix{
S^W \wedge N_+ \ar[r]^{\tr_W} \ar[d]_{S^W \wedge f_V}& S^W \wedge M_+\ar[d]^{S^W \wedge f_U}\\
S^{W+T_{q(m)}N} \ar[r]_\simeq & S^{W+T_mM}.
}
\end{equation}

To see this, let $\nu_U \subseteq \nu$ be the image of $W \times U$ under the equivariant diffeomorphism $W \times M \cong \nu$
of \eqref{eq:triv-tube-1}. Then a model for the square \eqref{eq:eq-collapse-unstable-1} is
\[
\xymatrix{
\frac{S^W \times N}{\{\infty\} \times N} \ar[r] \ar[d] & \frac{S^W \times N}{(S^W \times N)-\nu} \ar[d]\\
\frac{S^W \times N}{\{\infty\} \times N \cup S^W \times (N-V)} \ar[r] & \frac{S^W \times N}{(S^W \times N)-\nu_U}.
}
\]

We can now take the diagram of \eqref{eq:eq-collapse-unstable-1}, stabilize, and cancel the representation sphere
$S^W$ to obtain the result.
\end{proof}

\subsection{Linear spheres} Let $\cG$ be a compact $p$-adic analytic group of rank $d$ 
and let $\gg$ be the adjoint representation.

\begin{prop}\label{prop:recog-princ} Let $F \subseteq \cG$ be a finite subgroup. Suppose there is
a finitely generated free abelian group $L \subseteq \gg$ with the properties that 
\begin{enumerate}

\item $L$ is stable under the adjoint action of $F$ on $\gg$, and

\item $L/pL \cong \gg/p\gg$.
\end{enumerate}
Let $V = \RR \otimes L$ and let $S^V$ be the one-point compactification of $V$. Then there is
an $F$-equivariant map $S^V \to S^\gg$ which becomes a weak equivalence after completion at $p$. 
\end{prop}

The proof is below, after \Cref{prop:int-transfers}. The argument we have in mind proceeds geometrically,
using that $V/L= (\RR\otimes L)/L$ is an $F$-parallelizable 
manifold, as shown in \Cref{lem:lat-trival}. 

As before, if $Y$ is any $F$-space and $W$ is a real representation of $F$, then $\theta_W$ denotes the trivial bundle with total
space $Y \times W$. If $M$ has an action by a finite group $F$ and $m \in M$ is a fixed point, then $T_mM$ is real representation
of $F$. Note also that $0 = 0 + L \in V/L$ is fixed under the action of $F$. 

\begin{lem}\label{lem:lat-trival} Let $F$ be a finite group acting on a finitely generated free abelian group $L$. Let
$V = \RR \otimes L$. Then there is an equivariant isomorphism of bundles over $V/L$
\[
T(V/L) \cong \theta_V.
\]
In particular there is an isomorphism of $F$-representations $T_0(V/L) \cong V$. 
\end{lem}

\begin{proof} The standard linear trivialization $\theta_V \cong TV$ over $V$ descends to the needed trivialization over $V/L$. 
Specifically, the trivialization over $V$ is the map
\[
V \times V\longr TV
\]
sending $(v,w)$ to $(v,\gamma'(0))$ where $\gamma(t) = v + tw$.
\end{proof}

\Cref{lem:lat-trival} and \Cref{lem:eq-top} immediately imply the following.

\begin{prop}\label{prop:collapse-eq} Let $F$ be a finite group acting on a finitely generated free abelian group $L$.
Let $V = \RR \otimes L$. Then there is a choice of $F$-equivariant neighborhood $U$
of $0 + L \in V/L$ that gives an equivariant collapse map
\[
f_U: V/L \to S^V.
\]
This map sends $0 + L$ to $0 \in V \subseteq S^V$. 
\end{prop}

The following result is an immediate consequence of \Cref{prop:eq-collapse-stable} and \Cref{prop:collapse-eq}. 

\begin{prop}\label{prop:int-transfers} Let $F$ be a finite group acting on a finitely generated free abelian group $L$.
Let $V = \RR \otimes L$. Then there is a stable $F$-equivariant equivalence
\[
\hocolim \Sigma^\infty_+ (V/p^iL) \simeq S^V
\]
where the colimit is over the transfers for the covering map $V/p^{i+1}L \to V/p^iL$. 
\end{prop} 

We can now prove the main result of this section.

\begin{proof}[{\bf Proof of  \Cref{prop:recog-princ}}]
Let $\Gamma_i \subseteq \cG$
be our preferred set of open uniformly powerful subgroups of the compact $p$-adic analytic group $\cG$. Recall from \Cref{def:class-pro-finite} that we have defined
\[
B(p^i\gg) = \holim B(p^i\gg/p^{i+j}\gg)
\]
and we proved in \Cref{prop:it-doesn't-matter} that there is an equivalence after $p$-completion
\begin{equation}\label{eq:here-we-go-again}
\Sigma_+^\infty B(p^i\gg) = \holim \Sigma_+^\infty B(p^i\gg/p^{i+j}\gg).
\end{equation}
Finally, $S^\gg = \hocolim \Sigma_+^\infty B(p^i\gg)$, again after $p$-completion. 

Our assumption that $L/pL \cong \gg/p\gg$ allows us to conclude that for all $i$ the map
\[
B(p^iL) \longr B(p^i\gg)
\]
induces $p$-completion on the fundamental group. The map $V/p^iL \to B(p^iL)$ classifying the principal $p^iL$ bundle
$V \to V/p^iL$ is a weak equivalence and we also have that the composition
\[
V/p^iL \longr B(p^iL) \longr B(p^i\gg)
\]
induces $p$-completion on the fundamental group. Since both spaces have no non-zero higher homotopy groups
it is then an isomorphism on $H_\ast(-,\FF_p)$. Then
\[
\Sigma_+^\infty V/p^iL \longr \Sigma_+^\infty B(p^i\gg)
\]
is also an isomorphism on $H_\ast(-,\FF_p)$. Using \eqref{eq:sg-defined} and \Cref{prop:int-transfers} 
we can now conclude that
\[
S^V \simeq \hocolim \Sigma_+^\infty V/p^iL \longr \hocolim \Sigma_+^\infty B(p^i\gg) = S^\gg
\]
is an isomorphism on $H_\ast(-,\FF_p)$. This implies the result.
\end{proof} 

In our examples the natural choice of a finitely generated free abelian group $L \subseteq \gg$ may not have property that
$L/pL \cong \gg/p\gg$. The following result will let us relax this hypothesis.

\begin{lem}\label{lem:sub-lattice} Let $M$ be a free $\ZZ_p$-module of finite rank and let $L_0 \subseteq M$ be 
a finitely generated free abelian group so that $\QQ_p \otimes L_0 \cong \QQ_p \otimes_{\ZZ_p} M$. Define
\[
L = \{ x \in M\ |\ p^kx \in L_0\ \hbox{\rm{for some} $k \geq 0$}\} \subseteq M.
\]
Then $L$ is a finitely generated free abelian group, $\QQ_p \otimes L \cong \QQ_p \otimes_{\ZZ_p} M$, $L_0 \subseteq L$
has finite index, and
\[
L/pL \cong M/pM.
\]
Furthermore, if $f:M \to M$ is any continuous $\ZZ_p$-linear isomorphism such that $f(L_0)=L_0$, then $f(L) = L$. 
\end{lem} 

\begin{proof} Filter $M$ by powers of $p$; that is, set $F_sM = p^sM$. Let $E_\ast M$ be the associated graded module. Then
$E_\ast M = \FF_p[x] \otimes M/pM$ where $M/pM$ has filtration $0$ and $x$ has filtration $1$ and is the residue
class of $p$ in $E_\ast \ZZ_p$. Since $M/pM$ is finite, $E_\ast M$ is a finitely generated $\FF_p[x]$-module.

We have induced filtrations on $L_0$ and $L$; for example, $F_sL = F_sM \cap L$. The hypothesis
that $\QQ_p \otimes L_0 \cong \QQ_p \otimes_{\ZZ_p} M$ implies there is 
an integer $n$ so that we have inclusions of graded $\FF_p[x]$-modules 
\[
x^nE_\ast M \subseteq E_\ast L_0 \subseteq E_\ast L \subseteq E_\ast M. 
\]
We first show the last inclusion is an equality. 

Suppose $a \in E_sM$ is represented by $\alpha \in F_sM$. Then $x^na \in E_\ast L_0$ so there is a $\beta \in F_{s+n}L_0$
and $z \in p^{s+n+1}M = F_{s+n+1}M$ so that 
\[
p^n \alpha = \beta + z.
\]
Since $p:F_sM \to F_{s+1}M$ is an isomorphism, we may choose $y \in F_{s+1}M$ so that $p^{n}y = z.$ Set
$\gamma = \alpha - y$. Then $\gamma \equiv \alpha$ modulo $F_{s+1}M$ and hence $\gamma$ also represents $a$. 
Furthermore
\[
p^n \gamma =  \beta \in L_0
\]
Thus $\gamma \in L$ as needed. 

We can now prove the result. Since $p^nL \subseteq L_0$ and $p^n: L \to p^nL$ is an isomorphism,
we have that $L$ is finitely generated and free. The inclusions $p^nL \subseteq L_0 \subseteq L$ imply
$L_0$ is of finite index in $L$, that $L_0$ and $L$ have the same rank, and that
\[
\QQ_p \otimes L_0 \cong \QQ_p \otimes L \cong \QQ_p \otimes_{\ZZ_p} M.
\]
The equality $E_\ast L = E_\ast M$ shows $L/pL = M/pM$.  

Let $f:M \to M$ be any isomorphism, and let $x \in L$ so that $p^kx \in L_0$ for some $k$. Since $L_0$ is invariant under
$f$, we have $p^kf(x) \in L_0$, so $f(x) \in L$. That every $x\in L$ is in $f(L)$ is an easy exercise. 
\end{proof}

Using this we have the following useful variant of \Cref{prop:recog-princ}. 

\begin{prop}\label{prop:recog-princ-bis} Let $F \subseteq \cG$ be a finite subgroup. Suppose there is
a finitely generated free abelian group $L_0 \subseteq \gg$ with the properties that 
\begin{enumerate}

\item $L_0$ is stable under the adjoint action of $F$ on $\gg$, and

\item $\QQ_p \otimes L_0 \cong \QQ_p \otimes_{\ZZ_p} \gg$.
\end{enumerate}
Let $V = \RR \otimes L_0$ and let $S^V$ be the one-point compactification of $V$. Then there is
an $F$-equivariant map $S^V \to S^\gg$ which becomes a weak equivalence after completion at $p$. 
\end{prop}

\begin{proof} Use \Cref{lem:sub-lattice} to produce a finitely generated free abelian group $L$ so that
$L_0 \subseteq L \subseteq \gg$,  $L$ is $F$-invariant, and $L/pL \cong \gg/p\gg$. Since
$V = \RR \otimes L_0 \cong  \RR \otimes L$, \Cref{prop:recog-princ} now applies.
\end{proof}

 
\section{Lubin-Tate theory and its fixed points}\label{sec:e-theory}

We fix $p$ and $n \geq 1$. If $F$ is a formal group law of height $n$ over a finite algebraic extension $k$ of $\FF_p$, 
let $\LTE = E(k,F)$ be the Lubin-Tate spectrum -- aka the Morava $E$-theory -- associated to the pair $(k,F)$. 
The spectrum $\LTE$ is an $E_\infty$-ring spectrum with an action, through $E_\infty$-ring maps,
of the group $\GG = \Aut(k,F)$ of the automorphisms of pair $(k,\Gamma)$ \cite{rezk_hm, gh_modulien}.
Namely, the elements of $\GG$ are pairs $( f, \varphi)$ where $\varphi \in \Gal(k/\F_p)
$ and $f \colon \varphi^*F \xrightarrow{\cong} F$ is an isomorphism; an explicit description of $\GG$ for the Honda formal group law is given in \Cref{exmp:exam2}.
 A general theory is possible here, but we will focus on two examples.
For both, the pair $(k,F)$ will have the property that for any algebraic extension $k \subseteq k'$ the inclusion
$\Aut(k,F) \subseteq \Aut(k',F)$ is an isomorphism. We will say that the pair $(k,F)$ has all automorphisms at $k$. 

\begin{exam}\label{ex:which-fgls1} The classical example is the Honda formal group law $F$ of height $n$ over $\FF_{p^n}$.
In this case $F$ is the unique $p$-typical formal group law defined over $\FF_p$ with $p$-series
$x^{p^n}$. Not all automorphisms of $F$ are defined over $\FF_p$; for this we need to pass to $\FF_{p^n}$.
\end{exam}
\begin{exam}\label{ex:which-fgls2}
The other case of interest is specific to $n=2$, and we only concentrate on $p=2$ or $3$. Then $F$ is the formal group law of a supersingular
elliptic curve defined over $\FF_p$. We can and will choose the curves so that once we
pass to $\FF_{p^2}$ we get all automorphisms.
\end{exam}

{\bf From now on we assume we are working with one of these two examples.} In both cases $k=\FF_{p^n}$,
where $n$ is the height of the formal group.

\begin{rem}[{\bf Structure of the Morava stabilizer group}]\label{rem:struct-morava-stab} Because the formal
group law $F$ is defined over $\FF_p$, there is a split surjective homomorphism
\[
\GG = \Aut(\FF_{p^n},F) \longr \Gal(\FF_{p^n}/\FF_p)
\]
with kernel the group $\SS := \Aut(F/\FF_{p^n})$, the group of automorphisms of $F$ over $\FF_{p^n}$. Both $\GG$
and its subgroup $\SS \subseteq \GG$ are compact $p$-adic analytic groups. We write
\begin{equation}\label{eq:subgrs-of-sn}
\cdots \subseteq \Gamma_{i+1} \subseteq \Gamma_i \subseteq \cdots \subseteq \Gamma_1 \subseteq \SS
\end{equation}
for the filtration of \Cref{exmp:exam2}.
\end{rem}

There is a non-canonical isomorphism 
\[
W(\FF_{p^n})[[u_1,\ldots,u_{n-1}]][u^{\pm 1}]\cong \LTE_\ast 
\]
where the power series ring is in degree $0$ and $u$ is in degree $-2$. The power series ring $\LTE_0$ is a complete
local Noetherian ring with maximal ideal $\mm = (p,u_1,\ldots,u_{n-1})$ and residue field $\FF_{p^n}$. Define
$\LTK=K(n)$ to be the version of Morava $K$-theory with $\LTK_\ast = \FF_{p^n}[u^{\pm 1}]$ and formal group law $F$.
There is map of complex oriented ring spectra $\LTE \to \LTK$ which, on coefficients,
is the quotient by $\mm$. The category of $\LTK$-local spectra is the standard $K(n)$-local category. 

The $\LTK$-local category is a closed symmetric monoidal category with smash product $L_{\LTK}(X \wedge Y)$.
As is commonly done, we define
\[
\LTE_\ast X \cong \pi_\ast L_{\LTK}(\LTE \wedge X). 
\]
This is not quite a homology theory, so the notation is slightly abusive. It has very nice algebraic properties;
see for example \Cref{rem:e-of-fixed points}. The action of $\GG$ on $\LTE$ gives a continuous action of $\GG$
on $\LTE_\ast X$ making $\LTE_\ast X$ a \emph{Morava module}
(see \cite[Section 1.3]{BobkovaGoerss} for a definition).

\begin{rem}\label{rem:towers-of-n} Recall that $L_n$-localization is localization with respect
to the homology theory $K(0)\vee\ldots\vee K(n)$ or, equivalently, with respect to $\LTE$. Here $K(0)$ is rational homology.

Using the periodicity results of Hopkins and Smith \cite{nilpotence2},
Hovey and Strickland produce a sequence of ideals $J(i) \subseteq \mm \subseteq \LTE_0$ and finite type $n$
spectra $\MJ{i}$ with the following properties:
\begin{enumerate}

\item $J(i+1) \subseteq J(i)$ and $\bigcap_i J(i) = 0$;

\item $\LTE_0/J(i)$ is finite;

\item $\LTE_\ast(\MJ{i}) \cong \LTE_\ast/J(i)$ and there are spectrum maps $q\colon \MJ{i+1} \to \MJ{i}$ realizing the quotient
$\LTE_\ast/J(i+1) \to \LTE_\ast/J(i)$;

\item There are maps $\eta=\eta_i\colon S^0 \to \MJ{i}$ inducing the quotient map $\LTE_0 \to \LTE_0/J(i)$ and
$q\eta_{i+1} = \eta_i\colon
S^0 \to \MJ{i}$; 

\item If $X$ is a finite type $n$ spectrum, then the map $X \to \holim_i(X \smsh \MJ{i})$ induced by the maps
$\eta$ is an equivalence.

\item  If $X$ is any $L_n$-local spectrum, then $L_{\LTK}X \simeq \holim_i X \smsh
\MJ{i}$. In particular we have $ \LTE \simeq \holim_i \LTE \smsh \MJ{i}$.

\end{enumerate} 
Most of this is proved in \cite[Section 4]{666}, while (6) is proved in \cite[Proposition 7.10]{666}.
In the same source, Hovey and Strickland also prove that items (1)-(5) characterize the tower $\{\MJ{i}\}$ up to equivalence
in the pro-category of towers under $S^0$; see \cite[Proposition 4.22]{666}.
Note that the topology on $\LTE_0$ defined by the sequence $\{ J(i)\}$ is same as the $\mm$-adic topology and
that $\G$ acts on $\LTE_\ast/J(i)$ through a finite quotient of $\GG$. 
\end{rem}

We can now prove the following useful recognition lemma. 

\begin{lem}\label{lem:detecting-kn-equiv} Let $f:X \to Y$ be a map of $L_n$-local spectra. Then $L_\LTK f:L_\LTK X \to 
L_\LTK Y$ is an equivalence if and only if there is a finite type $n$ complex $T$ so that 
\[
f \wedge T: X \wedge T \longr Y \wedge T
\]
is an equivalence. 
\end{lem}

\begin{proof} The collection of type $n$ complexes forms a thick subcategory $\cC_n$ of finite spectra, so if $f \wedge T$ is
an equivalence for any $T \in \cC_n$, it is an equivalence for all $T \in \cC_n$. Thus one implication follows
from item (6) of \Cref{rem:towers-of-n}. The other implication follows from the fact that if $X$ is $L_n$-local and
 $T$ is of type $n$, the map $X \wedge T \to L_\LTK (X \wedge T)$ is an equivalence.
\end{proof} 

\begin{rem}[{\bf Fixed point spectra and transfers}]\label{rem:fixed-and-tr} According to the theory of Devinatz and Hopkins \cite{DH}, for any closed subgroup $K  \subseteq \GG$ there is a continuous homotopy fixed point spectrum
$\LTE^{hK}$. These are again $E_\infty$-ring spectra, and there is a fixed point spectral sequence
\begin{equation}\label{eq:fixed-pnt-ss}
H^s(K,\LTE_t) \Longrightarrow \pi_{t-s}\LTE^{hK}.
\end{equation}
The $E_2$-term is continuous group cohomology. We have $\LTE^{h\GG} \simeq L_{\LTK}S^0$ and 
the spectral sequence of \eqref{eq:fixed-pnt-ss} is
the $\LTE$-based Adams-Novikov spectral sequence in the $\LTK$-local category.

The spectral sequence of \eqref{eq:fixed-pnt-ss} has a very rigid convergence properties. There is an integer
$N$, independent of $K$, so that $E_\infty^{s,\ast} = 0$ for $s \geq N$. If $K \subseteq \Gamma_1$
(or $K \subseteq \Gamma_2$ if $p > 2$), then 
\[
H^s(K,\LTE_t) = 0 
\]
for $s > n^2$, as $K$ is then a Poincar\'e duality group of dimension $\leq n^2$. 

These fixed point spectra of $\LTE$ also  have the property that if $K_1 \subseteq K_2$ is normal and of finite index, 
then $\LTE^{hK_1}$  has an action by $K_2/K_1$ and there is a weak equivalence
\begin{equation}\label{eq:double-fixed}
\LTE^{hK_2} \simeq (\LTE^{hK_1})^{hK_2/K_1}.
\end{equation}
This is proved as Theorem 4 of \cite{DH}.
In particular, we have a transfer map $\tr: \LTE^{hK_1} \to \LTE^{hK_2}$. If $K \subseteq \GG$ is open (and hence closed), 
then $K$ is of finite index in $\GG$ and we have a transfer map
\[
\tr:\LTE^{hK} \to L_{\LTK}S^0 \simeq \LTE^{h\GG}.
\]
\end{rem}

\begin{rem}\label{rem:e-of-fixed points} Let $H \subseteq \GG$ be a closed subgroup. 
A crucial property of the fixed point spectrum $\LTE^{hH}$, from \cite[Theorem 2]{DH}, is that there are isomorphisms
\begin{align}\label{eq:e-of-efixed}
\LTK_\ast \LTE^{hH} &\cong \pi_\ast L_{\LTK}(\LTK \wedge \LTE^{hH}) \cong \map(\GG/H,\LTK_\ast)\\
\LTE_\ast \LTE^{hH} &\cong \pi_\ast L_{\LTK}(\LTE \wedge \LTE^{hH}) \cong \map(\GG/H,\LTE_\ast)\nonumber
\end{align} 
where $\map$ denotes continuous set maps. The action of $\GG$ on the left factor of $\LTE$ defines the Morava
module structure on $\LTE_\ast\LTE^{hH}$. The isomorphism of \eqref{eq:e-of-efixed} becomes an isomorphism of
Morava modules if we give the module of continuous maps the conjugation action
\[
(g\varphi)(x) = g\varphi(g^{-1}x).
\]
If $H$ is normal, then $\GG$ acts on $\LTE^{hH}$. The isomorphism of \eqref{eq:e-of-efixed} becomes equivariant if  we define an action on the module of continuous maps by
\[
(g\star \varphi)(x) = \varphi(xg).
\]
\end{rem}

\begin{rem}[{\bf Continuous actions and fixed points}]\label{rem:defn-cont-action} For one of our main
results we will need some details about how the continuous fixed points of \cite{DH} are constructed.  
See \Cref{lem:baby-step-2}.

Let $\cG = \lim_i G_i$ be a profinite group. We are thinking of $\cG = \GG$, $\SS$, or $\Gamma_i \subseteq \SS$.
Let $X$ be a spectrum presented as a homotopy inverse limit
\[
X \simeq \holim X_j
\]
with $\pi_t X_j$ {\bf finite} for all  $j$ and $t$. We are thinking of the $\LTE \simeq \holim_j \LTE \smsh \MJ{j }$, as in
point (6) of \Cref{rem:towers-of-n}. 
Define the spectrum of continuous map by
\begin{equation}\label{eq:cont-funk}
\Func(\Sigma_+^\infty \cG,X) = \mathop{\holim}_j \mathop{\hocolim}_i F(\Sigma_+^\infty G_i,X_j).
\end{equation}
This definition is built so that 
\begin{equation}\label{eq:cont-funk-cons}
\pi_\ast \Func(\Sigma_+^\infty \cG,X) \cong \map(\cG,\pi_\ast X),
\end{equation}
where the target is the group of continuous maps. We can extend this definition of the continuous function spectrum
to
\[
\Func(\Sigma_+^\infty \cG^{1+s},X) = \mathop{\holim}_j \mathop{\hocolim}_i F(\Sigma_+^\infty \cG^{1+s}_i,X_j),
\]
where $\cG^m$ is the $m$-fold Cartesian product of $\cG$ with the product topology.
Then the data of a continuous action of $\cG$ on $X$ is a map $X \to \Func(\Sigma_+^\infty \cG,X)$ which extends 
(in the obvious way) to a map of cosimplicial spectra
\[
X \longr \Func(\Sigma_+^\infty \cG^{1+\bullet},X).
\]
One of the main theorems of \cite{DH} is that the action of $\GG$ in $\LTE$ can be refined to a continuous action.

If $\Gamma \subseteq \cG$ is open,  then we define
\[
X^{h\Gamma} = \holim_\Delta  \Func(\Sigma_+^\infty \cG^{1+\bullet},X)^\Gamma =
\holim_\Delta \Func(\Sigma_+^\infty (\cG/\Gamma \times \cG^{\bullet}),X).
\]
The Bousfield-Kan Spectral Sequence of this cosimplicial space then gives the homotopy fixed point
spectral sequence
\[
H^s(\Gamma,\pi_tX) \Longrightarrow \pi_{t-s}X^{h\Gamma},
\]
where the $E_2$-term is continuous cohomology.
\end{rem}

This notion of homotopy fixed points commutes with various types of inverse limits; the following will
be sufficient for our purposes. If $X$ is a spectrum, let $P_nX$ be its $n$th Postnikov section. 

\begin{lem}\label{eq:fixed-inv-commute} Let $\cG$ be a profinite group and let $\Gamma \subseteq \cG$
be an open subgroup. Suppose $X$ is presented as a homotopy inverse limit
\[
X \simeq \holim X_j
\]
with $\pi_t X_j$  finite for all $j$ and $t$ and that, with this presentation, $X$ has a continuous $\cG$-action.
Then $P_nX \simeq \holim P_nX_j$ and the natural map
\[
X^{h\Gamma} \longr \holim_n (P_nX)^{h\Gamma}
\]
is an equivalence.
\end{lem}

\begin{proof} Note that by the definition \eqref{eq:cont-funk} and \eqref{eq:cont-funk-cons}
\[
\pi_t \Func(\Sigma_+^\infty \cG^\bullet,X) \longr \pi_t \Func(\Sigma_+^\infty \cG^\bullet,P_nX) 
\]
is an isomorphism for $t \leq n$. Thus the natural map
\[
\Func(\Sigma_+^\infty \cG^{1+\bullet},X) \longr \holim_n \Func(\Sigma_+^\infty \cG^{1+\bullet},P_nX) 
\]
is an equivalence of cosimplicial spectra. The result follows.
\end{proof}


\section{Dualizing the Lubin-Tate spectrum}\label{sec:dualofe}

In this section we prove one of our main theorems, identifying the equivariant homotopy type of the Spanier-Whitehead dual
of $\LTE$. See \Cref{thm:dual-en} and \Cref{cor:what-we-use}.

Let $D(\hbox{-}) = F(\hbox{-},L_{\LTK}S^0)$  denote Spanier-Whitehead duality in the $\LTK$-local category. By definition,
a $\LTK$-local spectrum $X$ is dualizable if the natural map
\[
L_{\LTK}(DX \wedge Y) \to F(X,Y)
\]
is an equivalence for all $\LTK$-local spectra $Y$.

\begin{lem}\label{lem:lem-colim-fixed} (1) Let $\Gamma \subseteq \GG$ be an open subgroup. Then $\LTE^{h\Gamma}$
is dualizable in the $\LTK$-local category. Furthermore, $\LTE_\ast D\LTE^{h\Gamma} $ is zero in odd degrees and
there is a $\GG$-equivariant isomorphism
\[
\LTE_0 D\LTE^{h\Gamma} \cong \LTE^0\LTE^{h\Gamma} \cong \LTE_0[\GG/\Gamma]
\]
where $\GG$ acts on $\LTE_0[\GG/\Gamma]$ by $g(\sum a_ix_i\Gamma) = \sum g(a_i)gx_i\Gamma$.
\bigskip

(2) Let $\Gamma_i \subseteq \SS \subseteq \GG$ be the open subgroups of \eqref{eq:subgrs-of-sn}. 
The natural maps $\LTE^{h\Gamma_i} \to \LTE$ induce a $\GG$-equivariant equivalence in the $\LTK$-local 
category
\[
\xymatrix{
\mathop{\hocolim}\LTE^{h\Gamma_i} \ar[r]^-\simeq & \LTE\ .
}
\]
\end{lem}

\begin{proof} Since $\Gamma$ is open, it is of finite index in $\GG$; therefore, \eqref{eq:e-of-efixed} implies that
$\LTE_\ast \LTE^{h\Gamma}$ is finitely generated as an $\LTE_\ast$-module. Then Theorem 8.6 of \cite{666} implies
that $\LTE^{h\Gamma}$ is dualizable. To finish the proof of part (1) we use the standard Spanier-Whitehead duality
isomorphism, the Universal Coefficient Spectral Sequence
\[
\LTE_0 D\LTE^{h\Gamma} \cong \LTE^0\LTE^{h\Gamma} \cong \Hom_{\LTE_0}(\LTE_0 \LTE^{h\Gamma},\LTE_0)
\]
and \eqref{eq:e-of-efixed}. Note that to make this isomorphism $\GG$-equivariant we must act by conjugation on the 
target; that is, $(g\psi)(a) = g\psi(g^{-1}a)$. 

For (2) we need to check that
\[
\hocolim \LTK_\ast \LTE^{h\Gamma_i} \cong \LTK_\ast \LTE^{h\Gamma_j}.
\]
By \eqref{eq:e-of-efixed}, this is equivalent to showing the map
\[
\hocolim \map(\GG/\Gamma_i,\LTK_\ast) \to \map(\GG,\LTK_\ast)
\]
is an isomorphism, where $\map$ continues to denote the continuous set maps. Since the topology on $\LTK_\ast$ is discrete, this is clear. 
\end{proof}
\bigskip


In order to assemble the Spanier-Whitehead duals of $\LTE^{h\Gamma_i}$
for various $i$ we need
a version of Frobenius reciprocity. Let $R$ be an $E_\infty$-ring spectrum and suppose $G$ is a finite group that acts on $R$
through $E_\infty$-ring maps. Then the spectrum $R^h:= F(EG_+, R)$ is an $E_\infty$-ring in genuine $G$-equivariant spectra \cite{HillMeier17}. \footnote{In fact, the same source explains that $R^h$ has more equivariant multiplicative structure, but we don't need that here.} The notation is chosen so that for any subgroup $H\subseteq G$, the categorical fixed points $(R^h)^H$ agree with the homotopy fixed points of the original $H$-action on $R$.
Suppose that $K \subseteq H \subseteq G$ are subgroups. Then we have the following maps
\begin{enumerate}
\renewcommand{\theenumi}{\alph{enumi}}
\item the inclusion (restriction) map $r:R^{hH} \to R^{hK}$; 

\item the transfer (induction) map $\tr:R^{hK} \to R^{hH}$; and,

\item conjugation maps $c_g:R^{hH} \to R^{h(g^{-1}Hg)}$,
\end{enumerate}
These maps are in fact defined for any genuine $G$-spectrum, and are at the core of the identification of genuine $G$-spectra with spectral Mackey functors \cite{GuillouMayII, Barwick}. 
Because it will be important for the next Lemma, let us recall the provenance of these maps. 
The restriction is induced by the map of $G$-sets $p:G/K \to G/H$; then $r = F(p_+,R^h)^G$. The conjugation maps are similarly induced by maps of $G$-sets, namely the conjugation by $g\in G$ maps $G/g^{-1}Hg \to G/H$. The transfer is not induced by a map of $G$-sets, but rather by the Spanier-Whitehead dual of $p_+$ in $G$-spectra \cite{AlaskaNotes}. 

\begin{lem}\label{lem:brand-new-frob-11} The inclusion, transfer, and conjugation maps have following properties.
\begin{enumerate}

\item The inclusion $r:R^{hH} \to R^{hK}$ and conjugation $c_g:R^{hH} \to R^{h(g^{-1}Hg)}$ are maps of $E_\infty$-ring spectra.

\item We have Frobenius Reciprocity commutative diagrams such as
\[
\xymatrix@R=15pt{
&R^{hK} \wedge R^{hK} \ar[r]^-m & R^{hK} \ar[dd]^{\tr}\\
R^{hH} \wedge R^{hK} \ar[ur]^{r \wedge 1} \ar[dr]_{1 \wedge \tr}\\
&R^{hH} \wedge R^{hH} \ar[r]_-m & R^{hH}.
}
\]
\end{enumerate}
\end{lem}

\begin{proof} We use ideas from \cite{ElmendMay}. 
More generally, this result holds for any $E_\infty$-ring $A$ in genuine $G$-spectra, not just those of form $A=R^h$.

For such an $A$, we have that the restriction and conjugations are maps of $E_\infty$-rings, since they are obtained by taking $G$-fixed points (a lax monoidal functor) of the $E_\infty$-ring maps 
\[F(G/H_+, R) \to F(G/K_+,R) \ \ \text{ and } \ \ F(G/H_+,R) \to F(G/g^{-1}Hg_+,R).\]

For the transfer, it suffices to restrict to $H$, so that $\tr$ will be the $H$-fixed points of the $A$-module map $F(H/K_+,A) \to A$, obtained by mapping into $A$ the Spanier-Whitehead dual in $H$-spectra of the map of $H$-sets $H/K_+ \to H/H_+$. Upon taking $H$-fixed points, we get that $A^K \to A^H$ is an $A^H$-module map, where $A^K$ has the $A^H$-algebra structure given by the restriction map. Frobenius Reciprocity is simply the diagramatic form of the statement that $\tr$ is an $A^H$-module map.
\end{proof}

Much more sophisticated results are possible, using the language of spectral Mackey functors
and spectral Green functors. See \cite{Barwick} and \cite{BarGlaSha}. 

From \Cref{lem:brand-new-frob-11} we have pairings
\[
\xymatrix{
R^{hH} \wedge R^{hH} \ar[r]^-m & R^{hH} \ar[r]^{\tr} & R^{hG}
}
\]
and get maps
\[
R^{hH} \to F(R^{hH},R^{hG})
\]
that switch restriction and transfer; for example, the following diagram commutes
\begin{equation}\label{eq:frob-infty} 
\xymatrix{
R^{hK} \ar[r] \ar[d]_\tr & F(R^{hK},R^{hG}) \ar[d]^{F(r,1)}\\
R^{hH} \ar[r] &F(R^{hH},R^{hG}).
}
\end{equation}


\begin{lem}\label{lem:dual-e-for} Let $r_j:\LTE^{h\Gamma_{j}} \to \LTE^{h\Gamma_{j+1}}$ be the restriction maps.
The maps $f_j$ fit into a $\GG$-equivariant commutative diagram
\[
\xymatrix{
\LTE^{h\Gamma_{j+1}} \dto_{\tr} \rto^-{f_{j+1}} & F(\LTE^{h\Gamma_{j+1}},L_{\LTK}S^0) \dto^{F(r_j,1)}\\
\LTE^{h\Gamma_{j}} \rto_-{f_j} &F(\LTE^{h\Gamma_{j}},L_{\LTK}S^0).
}
\]
Furthermore, we have a $\LTK$-local $\GG$-equivariant equivalence
\[
\mathop{\holim}_\tr \LTE^{h\Gamma_{j}} \simeq D\LTE.
\]
\end{lem}

\begin{proof} The commutative square can be obtained from \eqref{eq:frob-infty} by taking $R = E^{h\Gamma_{i+1}}$, 
$G = \GG/\Gamma_{i+1}$, $H = \Gamma_i/\Gamma_{I+1}$ and $K = \{e\} = \Gamma_{i+1}/\Gamma_{i+1}$.
Note that $R^{hG} = L_\LTK S^0$.

The last statement follows from Part (2) of \Cref{lem:lem-colim-fixed}. 
\end{proof}

\begin{rem}\label{rem:strickland-nsquared} It is an observation due to the third author (see \cite{StrickGH}), that 
there is an underlying (that is, non-equivariant)  equivalence of spectra $D\LTE \simeq \Sigma^{-n^2}\LTE$.
This included a calculation the $\GG$-action on $\pi_\ast D\LTE$. 
This older result is a corollary of our theorem \Cref{thm:dual-en}, but we will end up giving a proof of this fact as we go;
the argument is essentially the same as in \cite{StrickGH}. In summary, we will show certain spectral sequences
that appear in the proof of  \Cref{lem:baby-step-1} have 
only one non-zero line and the shift of $n^2$ is due to the degree of that line. 
\end{rem} 

We now add an algebraic result which will be used in several arguments to follow. There is nothing special
about the group $\GG$; the argument works equally well for any compact $p$-adic analytic group $\cG$.
We fix a nested sequence of open subgroups $\Gamma_i \subseteq \cG$, as in \Cref{rem:in-practice-gam}.
An abelian group $A$ is finite $p$-torsion if $A$ is finite and there is an integer $n$ so that $p^nA=0$.

\begin{lem}\label{lem:baby-step-0} Let $\cG$ be a compact $p$-adic analytic group of dimension 
$d$ and $A$ be a finite $p$-torsion discrete $\cG$-module. 
\smallskip

(1) There is an integer $N$ so that for all $i \geq N$ the subgroup $\Gamma_i$ acts trivially on $A$. 

(2) Taking the limit over transfers gives an isomorphism
\begin{equation*}\label{eq:lim-trans-gen}
\lim_j H^s(\Gamma_j,A) \cong
\begin{cases}A, & s=d;\\
0, & s \ne d.
\end{cases}
\end{equation*}
If $i \geq N$, then the natural map $\lim_j H^{d}(\Gamma_j,A) \to  H^{d}(\Gamma_i,A) \cong A$ 
is an isomorphism. 

(3) Taking the colimits over restriction gives an isomorphism
\begin{equation*}\label{eq:colim-dual-1}
\mathop{\colim}_i H^s(\Gamma_i,A) \cong
\begin{cases}A & s=0;\\
0 & s > 0.
\end{cases}
\end{equation*}
If $j \geq N$, then natural map $A \cong H^{d}(\Gamma_j,A) \to \colim_i H^{d}(\Gamma_i,A)$ 
is an isomorphism. 
\end{lem}

\begin{proof} To prove (1) we use that if $A$ is discrete the orbit of any element $a \in A$ is finite. Let $H_a \subseteq \GG$
be the isotropy subgroups of $a$, then $H_a$ is open, so $\cap_a H_a$ is also open, since this is a finite intersection.
Choose $N$ so that $\Gamma_N \subseteq \cap_a H_a$.

We next prove part (2). Since the limit depends only on large $j$, we make take $j \geq N$ and assume the
action is trivial.  Furthermore, since $H^k(\Gamma_i,\FF_p)$ is finite for all $k$, by \Cref{thm:coh-gammai}, the functor
\[
A \longmapsto \lim_j H^s(\Gamma_j,A)
\]
sends short exact sequences in $A$ to long exact sequences.  If $A = \ZZ/p$ the result can be found in \Cref{lem:transfer-hom};
more precisely, that result is for homology, but in this case cohomology is dual to homology.
The result is then immediate if $pA = 0$. Now use induction on $k$ and the long exact sequence in cohomology
obtained from the short exact sequence
\[
\xymatrix{
0 \ar[r] & A/pA \ar[r]^-{p^{k-1}} &  A \ar[r] & A/p^{k-1}A \ar[r] & 0 \ .
}
\]

Part (3) is proved the same way, now using part (2) of  \Cref{thm:coh-gammai}
\end{proof}

The next step is to decompose the homotopy inverse limit of \Cref{lem:dual-e-for}
even further, using part (2) of \Cref{lem:lem-colim-fixed}.

\begin{lem}\label{lem:baby-step-5} The $\GG$-equivariant $\LTK$-local equivalence
$\hocolim_i \LTE^{h\Gamma_i} \simeq \LTE$
induces a $\GG$-equivariant $\LTK$-local equivalence
\[
\xymatrix{
\hocolim_i [(\LTE^{h\Gamma_i})^{h\Gamma_j}]  \ar[r]^-\simeq & [\hocolim_i \LTE^{h\Gamma_i}]^{h\Gamma_j}
\simeq \LTE^{h\Gamma_j}
}
\]
and hence a $\GG$-equivariant $\LTK$-local equivalence
\[
\holim_j\hocolim_i [(\LTE^{h\Gamma_i})^{h\Gamma_j}] \simeq \holim_j \LTE^{h\Gamma_j}
\]
\end{lem}

\begin{proof} We apply \Cref{lem:detecting-kn-equiv}: to prove it is a $\LTK$-local equivalence, it is
sufficient to check we have an equivalence after smashing with some finite type $n$ complex $T$. Since $T$ is finite,
$(-) \wedge T$ commutes with homotopy fixed points and the map becomes
\begin{equation}\label{eq:switch-T-1}
\hocolim_i [((\LTE\wedge T)^{h\Gamma_i})^{h\Gamma_j}]  \longr  (\LTE\wedge T)^{h\Gamma_j}
\end{equation}
We may assume that we have chose $T$ so that for all $t$ the $\GG$-module $\LTE_t T$ is discrete, finite, and
annihilated by $\mm^k$ for some $k$. Such $T$ appeared in \Cref{rem:towers-of-n}. Write
$X_i = (\LTE\wedge T)^{h\Gamma_i}$.

In the next diagram we use the following basic fact: If $G = \lim G_s$ is profinite and $M$ is finite and discrete, then
$\colim H^\ast(G_s,M) = H^\ast (G,M)$.

We now have a diagram of spectral sequences
\[
\xymatrix{
\colim_iH^s(\Gamma_j,\pi_tX_i) \ar[r]\ar@{=>}[d] & H^s(\Gamma_j,\LTE_tT) \ar@{=>}[d]\\
\colim_i \pi_{t-s}(X_i^{h\Gamma_j}) \ar[r]  &\pi_{t-s}(\LTE\wedge T)^{h\Gamma_j}.
}
\]
Since $\colim \pi_\ast X_i = \LTE_\ast T$ and $\pi_tX_i$ is finite and discrete for each $t$, the map across the 
top is an isomorphism. Since each of the spectral sequences 
\[
H^s(\Gamma_j,\pi_tX_i) \Longrightarrow \pi_{t-s}(X_i^{h\Gamma_j})
\]
strongly converge and $\colim$ is an exact functor, the spectral sequence on the left strongly converges. Since the spectral 
sequence on the right strongly converges as well, we are done. 
\end{proof}

The next step is to switch  the limit and colimit.

\begin{lem}\label{lem:baby-step-1} The natural map
\begin{equation}\label{eq:switch}
\colim_i\holim_j (\LTE^{h\Gamma_i})^{h\Gamma_j} \longr \holim_j\colim_i (\LTE^{h\Gamma_i})^{h\Gamma_j}
\simeq D\LTE
\end{equation}
is a $\GG$-equivariant $\LTK$-local equivalence.
\end{lem}

\begin{proof} As a key to the upcoming argument, we make the following convention. If we use the index
$i$ we will be taking a colimit along maps induced by restrictions; thus for example
\[
\xymatrix{
\cdots \ar[r] & \LTE^{h\Gamma_{i-1}} \ar[r]^-r & \LTE^{h\Gamma_i} \ar[r]^-r &  \LTE^{h\Gamma_{i+1}} \ar[r] &\cdots
}
\]
If we use the index $j$ we will be taking the 
limit along maps induced by transfer; thus for example,
\[
\xymatrix{
\cdots \ar[r] & \LTE^{h\Gamma_{j+1}} \ar[r]^-{\tr} & \LTE^{h\Gamma_j} \ar[r]^-{\tr} &  \LTE^{h\Gamma_{j-1}} \ar[r] &\cdots
}
\]

The spectra $\LTE^{h\Gamma_i}$ are $\LTK$-local, hence $L_n$-local. Since $L_n$-localization
is smashing, the category of $L_n$-local spectra is closed under homotopy colimits and homotopy limits; hence
the map of \eqref{eq:switch} is a map of $L_n$-local spectra. We now apply \Cref{lem:detecting-kn-equiv}:
to prove it is a $\LTK$-local equivalence, it is sufficient to check it is an equivalence after smashing with some finite
type $n$ complex $T$. Since $T$ is finite the map may be rewritten
\begin{equation}\label{eq:switch-T}
\colim_i\holim_j ((\LTE\wedge T)^{h\Gamma_i})^{h\Gamma_j} \longr
\holim_j\colim_i ((\LTE\wedge T)^{h\Gamma_i})^{h\Gamma_j}\ .
\end{equation}
As before we assume that for all $t$, the $\GG$-module $\LTE_t T$ is discrete, finite and
annihilated by $\mm^k$ for some $k$; see \Cref{rem:towers-of-n}. In particular,
$\LTE_t T$ satisfies the hypotheses of \Cref{lem:baby-step-0}.

By \Cref{lem:detecting-kn-equiv}, the fixed point spectral sequences
\[
H^s(\Gamma_i,\LTE_t T) \Longrightarrow \pi_{t-s}(\LTE\wedge T)^{h\Gamma_i} 
\]
have a horizontal vanishing line at $s=n^2$ at $E_2$ for all $i \geq 2$. For all $t$, the cohomology groups
are finite and $p$-torsion; hence, $\pi_t(\LTE\wedge T)^{h\Gamma_i} $ is finite and of bounded $p$-power order for
all $t$. A similar 
argument shows $\pi_t((\LTE\wedge T)^{h\Gamma_i})^{h\Gamma_j}$ is  also finite and of bounded $p$-power order for all $t$. 
Thus, to prove the result, it's enough to prove that 
\begin{equation*}
\mathop{\colim}_i\lim_j \pi_*((\LTE\wedge T)^{h\Gamma_i})^{h\Gamma_j} \longr
\lim_j\mathop{\colim}_i \pi_*((\LTE\wedge T)^{h\Gamma_i})^{h\Gamma_j}
\end{equation*}
is an isomorphism as all $\lim{}^1$-terms vanish.

By assembling the $\Gamma_j$ fixed point spectral sequences we get a diagram of spectral sequences
\[
\xymatrix{
\mathop{\colim}_i\lim_j H^s(\Gamma_j,\pi_t(\LTE\wedge T)^{h\Gamma_i}) \ar[d] \ar@{=>}[r] &
\mathop{\colim}_i\lim_j \pi_{t-s}((\LTE\wedge T)^{h\Gamma_i})^{h\Gamma_j} \ar[d]\\
\lim_j\mathop{\colim}_i H^s(\Gamma_j,\pi_t(\LTE\wedge T)^{h\Gamma_i}) \ar@{=>}[r] &
\lim_j\mathop{\colim}_i \pi_{t-s}((\LTE\wedge T)^{h\Gamma_i})^{h\Gamma_j}\ .
}
\]
There is an issue here: while $\colim$ is exact on directed systems of abelian groups, $\lim$ is not, so applying 
$\lim$ to a system of spectral sequences does not necessarily yield a spectral sequence. Here again we use that
the terms $H^s(\Gamma_j,\pi_t(\LTE\wedge T)^{h\Gamma_i})$ are all finite, so that there will be no $\lim^1$-terms
and we do get a diagram of strongly convergent spectral sequences. 

It is now sufficient to show that the right vertical map on $E_2$-terms is an isomorphism. We will show much more: it will turn out 
that both the source and target of the map on $E_2$-terms are zero unless $s=n^2$; this directs our attention to $s=n^2$.

We now use \Cref{lem:baby-step-0}.
Choose $j$ large enough that the action of $\Gamma_j$ on $\pi_t(\LTE\wedge T)$ is trivial. Then we have a diagram
\[
\xymatrix{
\colim_i\lim_j H^{n^2}(\Gamma_j,\pi_t(\LTE\wedge T)^{h\Gamma_i}) \ar[r] \ar[d] & 
\colim_iH^{n^2}(\Gamma_j,\pi_t(\LTE\wedge T)^{h\Gamma_i})\ar[dd]\\
\lim_j\colim_i H^{n^2}(\Gamma_j,\pi_t(\LTE\wedge T)^{h\Gamma_i})\ar[d]\\
\lim_jH^{n^2}(\Gamma_j,\pi_t(\LTE\wedge T))\ar[r]& 
H^{n^2}(\Gamma_j,\pi_t(\LTE\wedge T)).
}
\]
By \Cref{lem:baby-step-0}, we know $H^{n^2}(\Gamma_j,\pi_t(\LTE\wedge T)) \cong \pi_t(\LTE\wedge T)$ 
and that the natural map $\lim_jH^{n^2}(\Gamma_j,\pi_t(\LTE\wedge T))\to
H^{n^2}(\Gamma_j,\pi_t(\LTE\wedge T))$ is an isomorphism. Thus we have a commutative triangle
\[
\xymatrix@C=5pt{
\colim_i\lim_j H^{n^2}(\Gamma_j,\pi_t(\LTE\wedge T)^{h\Gamma_i}) \ar[rr] \ar[dr]_f &&
\lim_j\colim_i H^{n^2}(\Gamma_j,\pi_t(\LTE\wedge T)^{h\Gamma_i})\ar[dl]^g\\
& \pi_t(\LTE\wedge T)
}
\]
We show both the maps $f$ and $g$ are isomorphisms. 

We begin with the map $f$. Using part (2) of \Cref{lem:baby-step-0} we have
\[
\lim_j H^s(\Gamma_j,\pi_t(\LTE\wedge T)^{h\Gamma_i}) \cong
\begin{cases}
\pi_t(\LTE\wedge T)^{h\Gamma_i}, &s = n^2;\\
0,  &s \ne n^2.
\end{cases} 
\]
Then, using part (2) of \Cref{lem:lem-colim-fixed}, we have (as needed) that
\[
\mathop{\colim}_i\lim_j H^s(\Gamma_j,\pi_t(\LTE\wedge T)^{h\Gamma_i}) \cong 
\begin{cases}
\pi_t(\LTE\wedge T), &s = n^2;\\
0, &s \ne n^2.
\end{cases} 
\]

We complete the argument by analyzing the map $g$. We claim the maps
\[
(\LTE\wedge T)^{h\Gamma_i} \simeq \LTE^{h\Gamma_i}\wedge T \to \LTE\wedge T
\]
induce an isomorphism
\[
\colim_i H^s (\Gamma_j,\pi_t(\LTE\wedge T)^{h\Gamma_i})) \cong H^s(\Gamma_j,\pi_t(\LTE \wedge T))
\]
and hence, by \Cref{lem:baby-step-0},  isomorphisms
\[
\lim_j \mathop{\colim}_i H^s (\Gamma_j,\pi_t(\LTE\wedge T)^{h\Gamma_i})) \cong 
\begin{cases}
\pi_t(\LTE \wedge T), & s = n^2;\\
0, & s \ne n^2.
\end{cases}
\]

If $M$ is any finite and discrete $\GG$-module, then $\Gamma_k$ acts trivially on $M$ for large
$k$. It follows that for all $j$ we have 
\[
\colim_k H^\ast(\Gamma_j/\Gamma_{j+k},M) \cong
H^\ast (\Gamma_j,M).
\]
This applies to $M = \pi_t(\LTE\wedge T)^{h\Gamma_i}$ and $M= \pi_t(\LTE \wedge T)$. Note also
that by part (2) of \Cref{lem:lem-colim-fixed} we have $\hocolim_i \LTE^{h\Gamma_i}\wedge T \simeq \LTE\wedge T$.
We now have
\begin{align*}
\colim_i H^s(\Gamma_j,\pi_t(\LTE\wedge T)^{h\Gamma_i}) 
&\cong \colim_i \colim_k H^s(\Gamma_j/\Gamma_{j+k},\pi_t(\LTE\wedge T)^{h\Gamma_i})\\
&\cong \colim_k \colim_i H^s(\Gamma_j/\Gamma_{j+k},\pi_t(\LTE\wedge T)^{h\Gamma_i})\\
& \cong \colim_k H^s(\Gamma_j/\Gamma_{j+k},\colim_i \pi_t(\LTE\wedge T)^{h\Gamma_i})\\
& \cong \colim_k H^s(\Gamma_j/\Gamma_{j+k},\pi_t (\LTE\wedge T))\\
&\cong H^s(\Gamma_j,\pi_t (\LTE\wedge T)).
\end{align*}
The second to last isomorphism uses part (2) of \Cref{lem:lem-colim-fixed}.
\end{proof} 

In light of \Cref{lem:baby-step-1}, the project now is to compute $\hocolim_i\holim_j 
(\LTE^{h\Gamma_i})^{h\Gamma_j}$; see  \Cref{lem:baby-step-4}. The next result is the first step.

\begin{lem}\label{lem:baby-step-2} Let $j  > i$ so that the action $\Gamma_j$
on $\LTE^{h\Gamma_i}$ is trivial. Then  there is a $\GG$-equivariant $\LTK$-local equivalence
\[
(\LTE^{h\Gamma_i})^{h\Gamma_j} \xrightarrow{\simeq} F(\Sigma^\infty_+ B\Gamma_j,\LTE^{h\Gamma_i}).
\]
where $\GG$ acts on $\Sigma^\infty_+ B\Gamma_i$ through conjugation and diagonally on the function spectrum.
\end{lem}

\begin{proof} We use the following basic fact about fixed point spectra. Let $G$ be a finite group, $K$ a 
normal subgroup of $G$, and  $X$ a  $G$ spectrum on which $K$ acts trivially. Then there is a natural equivalence of $G$-spectra
\[
\xymatrix{
X^{hK} \ar[r]^-\simeq & F(\Sigma^\infty_+ BK,X)
}
\]
where $G$ acts on $\Sigma^\infty_+ BK$ through conjugation and diagonally on the function spectrum. The
difficulty here is that $\GG$ is not finite.

Since we seek a $\LTK$-local equivalence we need only prove
this after smashing with a type $n$-complex $T$. See \Cref{lem:detecting-kn-equiv}. Thus we assume
we have a spectrum $X$ with a trivial $\Gamma = \Gamma_j$ action and with the property that $\pi_\ast X$ 
finite, discrete, and annihilated by some power of $p$. We will
show there is a natural $\GG$-equivariant equivalence
\[
X^{h\Gamma_j} \simeq F(\Sigma^\infty_+ B\Gamma_j,X).
\]
In practice, we will take $X = \LTE^{h\Gamma_i} \wedge T \simeq (\LTE \wedge T)^{h\Gamma_i}$.

Let $\Gamma = \Gamma_j$. Since $\Gamma$ acts trivially of $X$, we have a trivial action of the
finite group $\Gamma/\Gamma_k$ of $X$ for all $k \geq j$ and the maps $X^{h\Gamma/\Gamma_k} \to X^{h\Gamma}$
induced by the quotient map on groups give us a natural equivariant diagram
\[
\xymatrix{
\hocolim X^{h\Gamma/\Gamma_k} \ar[r]^-\simeq \ar[d]_f & \hocolim F(\Sigma_+^\infty B(\Gamma/\Gamma_k),X)\ar[d]^g\\
X^{h\Gamma}& F(\Sigma_+^\infty B\Gamma,X).
}
\]

Our strategy is now as follows. We will show that the maps $f$ and $g$ are equivalences when $X$ is bounded above in homotopy; 
that is, there is an integer $n$ so that $\pi_t X=0$ for $t > n$. For general $X$, this will then provide us a natural
equivalence
\[
(P_nX)^{h\Gamma} \simeq F(\Sigma^\infty_+ B\Gamma, P_nX).
\]
for all $n$. We now then can use \Cref{eq:fixed-inv-commute} to get an equivalence
\[
X^{h\Gamma} \simeq \holim_n (P_nX)^{h\Gamma} \simeq \holim_n F(\Sigma^\infty_+ B\Gamma, P_nX) 
\simeq F(\Sigma^\infty_+ B\Gamma, X).
\]
Thus, for the rest of the argument, assume $\pi_t X=0$ if $t > n$.

We first show $g$ is an equivalence. 
The basic fact we will use is that if $A$ is a finite discrete abelian group with the property that $p^kA=0$ for some $k$ with
trivial $\Gamma$-action, then
\begin{align}\label{eq:key-fact-again}
H^\ast (\Gamma,A) &= \colim H^\ast (\Gamma/\Gamma_k,A)\nonumber \\
&\cong \colim H^\ast (\Sigma^\infty_+ B(\Gamma/\Gamma_k),A) \cong H^\ast (\Sigma^\infty_+ B\Gamma,A).
\end{align}
See \Cref{prop:it-doesn't-matter}.

We are asserting that the natural map
\begin{equation}\label{eq:key-fact-again-1} 
\hocolim F(\Sigma^\infty_+ B(\Gamma/\Gamma_k),X) \longr F(\holim \Sigma^\infty_+ B(\Gamma/\Gamma_k),X)
\end{equation}
is an equivalence. This follows from \eqref{eq:key-fact-again} and an Atiyah-Hirzebruch Spectral Sequence argument.
We use that if $Y$ is a space, then
\[
E_2^{p,q}(Y) = H^p(Y,\pi_{-q}X) \Longrightarrow \pi_{-p-q}F(\Sigma^\infty_+Y,X) = X^{p+q}(Y) 
\]
is zero if $p< 0$ or $q < -n$. This implies that for any pair $(p,q)$, there is an $r$ so that $E_\infty^{p,q}(Y)
= E_r^{p,q}(Y)$ independent of $Y$ and, hence, that the spectral sequences converge. 
Combined with part (2) of \Cref{lem:basic-coh} this is sufficient to show \eqref{eq:key-fact-again-1} is an isomorphism.

We now show that $f$ is an equivalence. We have a diagram of spectral sequences
\[
\xymatrix@C=10pt{
\colim H^s(\Gamma/\Gamma_k,\pi_tX)\ar[d]_\cong \ar@{=>}[r] 
& \colim \pi_{t-s}X^{h\Gamma/\Gamma_k} \ar[d]^{f_\ast}
\\
H^s(\Gamma,\pi_tX) \ar@{=>}[r] &
\pi_{t-s}X^{h\Gamma}.
}
\]
The question then remains whether the upper spectral sequence converges. However, for all $k$ the
homotopy fixed point spectral sequence
\[
E_2^{s,t} = H^s(\Gamma/\Gamma_k,\pi_tX) \Longrightarrow \pi_{t-s}X^{h\Gamma/\Gamma_k}
\]
has the property that $E_2^{s,t} = 0$ if $t > n$ so for a fixed pair $(s,t)$ there is an $r$, independent of $k$, such that
$E_\infty^{s,t} = E_r^{s,t}$. This is enough to give convergence. 
\end{proof} 

\begin{lem}\label{lem:baby-step-4} There is a $\GG$-equivariant $\LTK$-local equivalence
\[
\colim_i\holim_j (\LTE^{h\Gamma_i})^{h\Gamma_j} \simeq F(\igg,\LTE).
\]
\end{lem}

\begin{proof} \Cref{lem:baby-step-2} and Frobenius Reciprocity imply that 
\begin{align*}
\holim_j (\LTE^{h\Gamma_i})^{h\Gamma_j} &\simeq \holim_j F(\Sigma^\infty_+ B\Gamma_{j},\LTE^{h\Gamma_i})\\
&\simeq F(\colim_j \Sigma^\infty_+ B\Gamma_j,\LTE^{h\Gamma_i})\\
& \simeq F(\igg,\LTE^{h\Gamma_i})
\end{align*}
since the (co-)limit is taken over transfer maps. Since $\igg$ is a $p$-complete sphere, it is dualizable; hence,
by part (2) of \Cref{lem:lem-colim-fixed}
\[
\hocolim_i F(\igg,\LTE^{h\Gamma_i}) \simeq F(\igg,\LTE). \qedhere
\]
\end{proof}

We now come to the main theorem of the section. If $X$, $Y$ and $Z$ are left $G$-spectra
the diagonal $G$-action on $F(Y,Z)$ is the action for which the standard adjunction between smash product
and function spectra restricts to an adjunction isomorphism
\[
F_G(X \wedge Y, Z) \cong F_G(X,F(Y,Z))
\]
where $F_G$ denotes the spectrum of $G$-maps and $G$ acts diagonally on $X \wedge Y$.
If we were allowed to use functional notation and $\phi \in F(Y,Z)$, then
\[
(g\phi)(y) = g\phi(g^{-1}y).
\]
Recall that $\igg = \hocolim_j \Sigma_+^\infty B\Gamma_j$ where the colimit is in the category of $p$-complete
spectra and taken over the transfer maps. The $\GG$ action is given by conjugation on the subgroups $\Gamma_j$.  See 
\Cref{def:ig-defined}. 

\begin{thm}\label{thm:dual-en} There is a $\GG$-equivariant $\LTK$-local equivalence
\[
D\LTE \simeq F(\igg,\LTE)  \simeq I_\GG^{-1}\wedge \LTE
\]
where $\GG$ acts diagonally on both $F(\igg,\LTE)$ and the smash product. 
\end{thm}

\begin{proof} This follows by combining  \Cref{lem:baby-step-5},
\Cref{lem:baby-step-1}, and \Cref{lem:baby-step-4}.
\end{proof}

\begin{rem} In Theorem 7.3.1 of \cite{BehrensDavis}, Behrens and Davis give an entirely different expression of the 
Spanier-Whitehead dual of $\LTE$ as a homotopy fixed point spectrum. It would be interesting to 
make a detailed comparison of that result with \Cref{thm:dual-en}.
\end{rem}
 
The following is now a consequence of \Cref{thm:a-big-one}.

\begin{cor}\label{cor:what-we-use} Let $F \subseteq \GG$ be a finite subgroup with $p$-Sylow subgroup
$F_0$. Suppose that $F_0/(F_0\cap Z(\GG))$ is an elementary abelian $p$-group. Then there is an 
$F$-equivariant $\LTK$-local equivalence
\[
D\LTE \simeq S^{-\gg} \wedge \LTE.
\]
\end{cor}

  
\section{The Spanier-Whitehead duals of $\LTE^{hF}$: some theory}\label{sec:PicStuff}

As an application of the theory developed so far, we will give some calculations of $D(\LTE^{hF})$ where
$F \subseteq \GG$ is a finite subgroup whose $p$-Sylow subgroup $F_0$  has the property that 
$F_0/F_0 \cap Z(\GG)$ is an elementary abelian $p$-group. Recall that $Z(\GG)$ is the center of
$\GG$. This recovers in a coherent
way all the known calculations of this kind in the literature. We also produce a new example. This section provides background and set-up. The precise
calculations, which involve the representation theory of specific groups, are in \Cref{sec:44} and \Cref{sec:newxample}. 

\subsection{Generalities on Picard groups}

Because Tate spectra vanish in the $K(n)$-local category (see \cite{GreenSad}), the 
norm map $\LTE_{hF} \longr \LTE^{hF}$ is a weak equivalence; it follows immediately that
\[
D(\LTE^{hF}) = F(\LTE^{hF},L_{\LTK}S^0) \simeq  F(\LTE_{hF},L_{\LTK}S^0) \simeq (D\LTE)^{hF}.
\]
By \Cref{cor:what-we-use} we have, for $F$ satisfying our hypothesis, that there is an $F$-equivariant
equivalence
\[
D\LTE \simeq S^{-\gg} \wedge \LTE
\]
where $S^\gg$ is the $\GG$-sphere obtained from the adjoint representation. We will assume that 
the $F$ action on $S^\gg$ satisfies the hypotheses of \Cref{prop:recog-princ}, so that we have
a finite dimensional real representation $V$ of $F$ and an $F$-equivariant equivalence of $p$-complete
$F$-spheres $ S^V  \simeq S^\gg$. The project then is to calculate the homotopy
type of $(S^{-V} \wedge \LTE)^{hF}$. 

We begin with some nomenclature. If $F$ is a finite group, let $RO(F)$ be the real representation
ring of $F$. If $V \in RO(F)$, we write $S^V$ for the stable one-point compactification of $V$ as an
$F$-sphere, $|V|$ for the dimension $V$, and $S^{|V|}$ for the stable sphere of dimension $|V|$.
Then $S^{|V|}$ is the underlying non-equivariant spectrum of $S^V$.

\begin{rem}[{\bf Picard Spectra}]\label{rem:why-pic} We begin with such basic theory as we need. A good summary and
references to the classical literature can be found in Section 2 of \cite{AkhSto}. We work in the category
of local spectra for some ambient homology theory, but leave this assumption implicit. 

Let $R$ be an $E_\infty$-ring spectrum. The space $\Gl_1(R) \subseteq \Omega^\infty R$
is defined by the pull-back diagram
\[
\xymatrix{
\Gl_1(R) \ar[r]\ar[d] & \Omega^\infty R \ar[d]\\
(\pi_0R)^\times \ar[r] & \pi_0\Omega^\infty R = \pi_0R.
}
\]
The $E_\infty$-multiplication on $R$ gives
$\Gl_1(R)$ the structure of an infinite loop space. 

Let $\sPic(R)$ be the category of invertible $R$-modules and $R$-module equivalences; in a slight abuse of
notation we will also write $\sPic(R)$ for the nerve of the category. Then 
\[
\pi_0\sPic(R) \cong \Pic(R)
\]
where $\Pic(R)$ is now the group of invertible $R$-module spectra up to equivalence. The space $\sPic(R)$  is an
infinite loop space; indeed, there is a connective spectrum $\pic(R)$ with 
\[
\Omega^\infty \pic(R) \simeq \sPic(R) \simeq \Pic(R) \times B\Gl_1(R). 
\]
\end{rem}

\begin{rem}[{\bf Group actions on Picard Spectra}]\label{rem:why-pic-2}
Let $R$ be an algebra in $G$-spectra over the terminal $N_{\infty}$-operad (as in \cite{BlumbergHill2015}). In particular, $R$ has all multiplicative norms. We assume that $R$ is cofree, so that $R \to F(EG_+,R)$  is an equivalence in $G$-spectra.
We also assume that the natural map $R^{hG} \to R$ is a faithful Galois
extension in the sense of Rognes \cite{rognes}. 
With appropriate care and using \cite{HillMeier17, BlumbergHill2015, BH}, in particular, \cite[Theorem 6.23]{BlumbergHill2015}, we may assume these
assumptions hold for $\LTE$ with its action of a finite subgroup of $\GG$; see  Remark 2.1 of \cite{BBHSC4}.

Let $H \subseteq G$ be a subgroup. Define $\sPic_H(R)$ to be the category of invertible $R$-modules $P$
with a compatible $H$-action; that is, the module multiplication map $R \wedge P \to P$ is an $H$-map, where
$H$ acts diagonally on the smash product. Let $\pic_H(R)$ be the associated spectrum. There is a functor
$\sPic(R^{hH}) \to \sPic_H(R)$ sending $Q$ to $R \wedge_{R^{hH}} Q$. 
Under our assumptions, this is an equivalence of categories. 
Hence $\Pic(R^{hH}) = \Pic_H(R)$ or, more generally, there is
an equivalence of spectra
\begin{equation}\label{eq:pic-equivs-sp}
\pic(R^{hH}) \simeq \pic_H(R).
\end{equation}
See Proposition 3.1 of \cite{BBHSC4}.

If $K \subseteq H$ is a subgroup, there is a restriction map $\sPic_H(R) \to \sPic_K(R)$. There is also a transfer
$\tr: \sPic_K(R) \to \sPic_H(R)$ defined using the Hill-Hopkins-Ravenel norm functor $N_K^H$. The assignment
\[
G/H \longmapsto \Pic_H(R) = \pi_0\pic_H(R)
\]
is then a Mackey functor, which we write $\uPic(R)$, with $G$ understood.  See Corollary 3.12 of \cite{BBHSC4}.
All of this uses technology developed in \cite{BH}. 
\end{rem}

\begin{rem}[{\bf Homotopy Fixed Point Spectral Sequences}]\label{rem:compare-to-hfpss}
Let $R$ and $G$ be as in \Cref{rem:why-pic-2}. Then $G$ acts on the category of invertible
$R$-modules: if $g \in G$, and $P$ is an invertible $R$-module spectrum, then
\[
gP = R \wedge_R^g P
\]
where we have extended scalars along the map $g\colon R \to R$. If $P \in \sPic_G(R)$, then multiplication by $g$
defines an $R$-module equivalence
\[
\xymatrix{
R \wedge_R^g P \ar[r]^-{1 \wedge g} & R \wedge_R P \ar[r]^-\simeq & P.
}
\]
This yields a map of spectra
\begin{align}\label{eq:pic-equivs-sp-1}
\pic_G(R) \longr \pic(R)^{hG}
\end{align}
which is an equivalence on $(-1)$-connected covers. See Theorem 3.3.1 of \cite{AkhSto} and the further references there. 
In particular we have isomorphisms
\[
\Pic(R^{hG}) \cong \Pic_G(R) \cong \pi_0\pic_G(R) \cong \pi_0\pic(R)^{hG}.
\]
There is then a homotopy
fixed point spectral spectral sequence
\begin{align}\label{eq:desc-pic-big}
E_2^{s,t}(R,G) = H^s(G,\pi_t\pic(R)) \Longrightarrow \pi_{t-s}\pic(R)^{hG}.
\end{align}
\end{rem}

\begin{rem}[{\bf The $J$-homomorphism}]\label{rem:J-om} Let $R$ and $G$ be as in \Cref{rem:why-pic-2} and
$H \subseteq G$ a subgroup.  If $V$ is a real $H$-representation, then $R \wedge S^V \in \Pic_H(R)$,
where we give $R \wedge S^V$ the diagonal $H$ action. As in Proposition 3.13 of \cite{BBHSC4} this
extends to a morphism of Mackey functors
\[
J_R: \underline{RO}\to \uPic(R) 
\]
where $\underline{RO}$ is the Mackey functor $G/H \mapsto RO(H)$. 

In fact, more is true. If we let $\Rep(H)$ be the category of real representations and isomorphisms; this is a symmetric monoidal
category under direct sum. Then we have a symmetric monoidal functor $J_R^H: \Rep(H) \to \sPic_H(R)$ and hence
a map of spectra $\ko_H \to \pic_H(R)$. Here $\ko_H$ is the spectrum of $H$-equivariant real $K$-theory. This is surely part of a 
morphism between spectral Mackey functors in sense of \cite{Barwick}, but we won't need that much structure. We will use that we 
have a commutative diagram for all $H \subseteq G$
\[
\xymatrix{
\ko_H \ar[r] \ar[d] & \pic_H(R) \ar[d]\\
\ko^{hH} \ar[r] & \pic(R)^{hH} \ .
}
\]
The action of $H$ on $\ko$ is trivial, so there is a weak equivalence
\[
F(\Sigma^\infty_+ BH,\ko) \simeq \ko^{hH}
\]
and the homotopy fixed point spectral sequence for $\ko^{hH}$ is the Atiyah-Hirzebruch Spectral Sequence
for $\ko^\ast BH$. 
\end{rem}

\begin{prop}\label{prop:norming-up-1} 
Let $R$ be as in \Cref{rem:why-pic-2}.

(1) Let $f\colon S^V \to R$ be a $G$-equivariant map so that the
underlying map of spectra $S^{|V|} \to R$ is a unit in $\pi_\ast R$. Then $J_R^G(V) = R^{hG} \in \Pic(R^{hG})$.

(2) 
Let $K\subset H$ be a subgroup. If $V \in RO(K)$ is in the kernel of $J_{R}^{K}$
then $W=\mathrm{ind}_{K}^{G}V$ is in the kernel of $J_{R}^{G}$.

\end{prop}

\begin{proof} For part (1) we extend $f$ to a map $g:R \wedge S^V \to R$ of $G$-equivariant $R$-modules.
This an underlying equivalence by our assumption on $f$, and so a $G$-equivalence since $R$ is cofree. (Note that if $R$ is cofree, so is the $R$-module $R \smsh S^V$). The claim follows.

Part (2) follows from the fact that $J_R: \underline{RO}\to \uPic(R)$ is a morphism of Mackey functors and induction of representations is the transfer in $\underline{RO}$.
\end{proof} 

When $R = \LTE$, the Lubin-Tate spectrum and $G$ is a finite subgroup of the Morava stabilizer group,
the fixed point spectral sequence of \eqref{eq:desc-pic-big} vanishes in high degree at $E_\infty$.
This inspires the following result. If $X$ is a spectrum, let $f: X\langle n \rangle \to X$ denote the $(n-1)$-connected
cover of $X$; thus $\pi_k f$ is an isomorphism for $k \geq n$ and $\pi_k X\langle n \rangle = 0$ if $k < n$. 

\begin{prop}\label{prop:guesstrue:2}Let $R$ be as in \Cref{rem:why-pic-2}.
Let $\ell$ be an integer such that $E_{\infty}^{s,s}(R)=0$ for
$s\geq \ell$ in the homotopy fixed point spectral sequence
\[
E_2^{s,t}(R)=H^s(G,\pi_t\pic(R)) \Longrightarrow \pi_{t-s}\pic(R)^{hG}
\]
Then the composite mapping 
\[
[\Sigma^\infty_+BG, \kon{\ell}] \to [\Sigma^\infty_+BG, \ko] \to \pi_{0}\pic(R)^{hG}
\]
is zero.
\end{prop}

\begin{proof}
The $G$-equivariant maps $\kon{\ell} \to \ko \to \pic(R)$ induces a diagram of homotopy fixed point spectral sequences
\[
\xymatrix{
E_2^{s,t}(\kon{\ell}) = H^s(G, \pi_t\kon{\ell}) \ar[d] \ar@{=>}[r] & \pi_{t-s}F(\Sigma_+^\infty BG,\kon{\ell}) \ar[d] \\
E_2^{s,t}(\ko) = H^s(G, \pi_t\ko) \ar[d] \ar@{=>}[r] & \pi_{t-s}F(\Sigma_+^\infty BG,\ko) \ar[d] \\
E_2^{s,t}(R) = H^s(G, \pi_t\pic R) \ar@{=>}[r]  & \pi_{t-s}(\pic R)^{hG}.
}
\]
By construction $E_2^{s,s}(\kon{\ell}) = 0$ if $s < \ell$ and by hypothesis $E_\infty^{s,s}(R)= 0$ is $s \geq \ell$.
The result follows. 
\end{proof}

\begin{rem}\label{rem-for-prop:guesstrue:2} The question now is how to check the hypotheses of 
\Cref{prop:guesstrue:2}. One technique, which we will employ below, is to access the ideas and techniques of
\cite{AkhSto} to relate the differentials in the homotopy fixed
point spectral sequence for $\pic(R)$ to the differentials in the homotopy fixed point spectral sequence for
$R$ itself. 
\end{rem}

\subsection{Applications to the Lubin-Tate spectrum} We now consider the case where $R = \LTE $ for
some prime $p$ and some height $n$ formal group. We established conventions and notation at
the beginning of \Cref{sec:e-theory}. Let $F \subseteq \GG$ be a finite subgroup. We are interested
in calculating the image of specific representations under the map $J_\LTE^F:RO(F) \to \Pic(\LTE^{hF})$. We have
the following useful preliminary result. 

\begin{prop} \label{prop:reg-goes-to-1} Let $F \subset \mathbb{S}$ be a finite subgroup
that contains the central subgroup $C_2 = \{\pm 1\}$. The regular representation $\rho_{F} \in RO(F)$ maps to the
trivial element of $\Pic(\LTE^{hF})$ under $J_\LTE^F$. 
\end{prop}

\begin{proof}
As describe in \Cref{rem:why-pic-2}, we can replace $\LTE$ by a cofree genuine $N_{\infty}$ ring $G$-spectrum. In particular, it admits all norm maps.
The key for this argument is to refine a unit $x \in \pi_{2}i^*_e\LTE$ to an equivariant map $\bar x \colon S^{  \rho_2} \to i^*_{C_2}\LTE$ where $i^*_{H}\LTE$ is the restriction of $\LTE$ to an $H$-spectrum. With this established, one applies \Cref{prop:norming-up-1}. 

At the prime $2$
the question is very subtle, but has been accomplished in 
\cite{HahnShi}.

If $p$ is odd this is much easier. 
Since $\LTE$ is $p$-local and $p$ is odd, we have that
\[\pi_{\rho_2}^{C_2} i^*_{C_2}\LTE \cong \pi_{0}(S^{-\rho_2} \smsh \LTE)^{hC_2} \cong \left(\pi_{0}i^*_e(S^{-\rho_2} \smsh \LTE)\right)^{C_2} \cong \left(\pi_{2}i^*_{C_2}(S^{1-\sigma} \smsh \LTE)\right)^{C_2} .\]
As a $C_2$-module, 
\[\pi_* i^*_e(S^{1-\sigma} \smsh \LTE) \cong \pi_*i^*_e \LTE \otimes \Z(-1)\]
where $\Z(-1)$ is the sign representation of $C_2$ on $\Z$. Choose a complex orientation $x\in \pi_{2}i^*_e\LTE$, and note that $\gamma x=-x$ for $\gamma$ a generator of $C_2$. Therefore, $x \otimes 1$ gives an element of
$(\pi_2i^*_e \LTE \otimes \Z(-1))^{C_2}$
which corresponds to a $C_2$-equivariant map
$\bar x \colon S^{\rho_2} \to i^*_{C_2}\LTE$ that refines $x$.
\end{proof}

\begin{rem}\label{rem:the-strategy} We can now outline the general strategy we use below. Fix a finite subgroup
$F \subseteq \GG$ with $p$-Sylow subgroup $F_0$. Suppose $F$ contains the central $C_2$ and that
$F_0/F_0 \cap Z(\GG)$ is an elementary abelian $p$-group, so that \Cref{thm:whod-have-believed-it} applies.
Finally suppose we have identified a finite dimensional real representation $V$ of $F$ so we can apply \Cref{prop:recog-princ}
and and write $S^\gg \cong S^V$ as an $F$-equivariant sphere. 

It follows from \cite{DH} that the homotopy fixed point spectral sequence
\[
H^s(F,\LTE_t) \Longrightarrow \pi_{t-s}\LTE^{hF}
\]
has a horizontal vanishing line at $E_\infty$. Applying the ideas from \cite{AkhSto} one can then conclude that 
the spectral sequence \eqref{eq:desc-pic-big}
\[
E_2^{s,t}(\LTE)=H^s(F,\pi_t\pic(\LTE)) \Longrightarrow \pi_{t-s}\pic(\LTE)^{hF}
\]
has the property that $E_\infty^{s,s} = 0$ for large $s$.  So we can apply \Cref{prop:guesstrue:2};
that is, there will be an integer $\ell$ so that the composition
\[
[\Sigma^\infty_+BF, \kon{\ell}] \to [\Sigma^\infty_+BF, \ko] \to \pi_{0}\pic(\LTE)^{hF} = \pi_0\pic(\LTE^{hF})
\]
is zero. In specific examples, we can be very explicit about the integer $\ell$. For example, if $n=2$ and $p=2$ and
$F$ is the automorphism groups of our supersingular curve, then we take $\ell = 8$. See \cite{AkhSto}, especially
Figure 9 and the surrounding narrative. 

Now let
\[
\ion{\ell} = RO(F) \cap \mathrm{Im}\{[\Sigma^\infty_+BF, \kon{\ell}] \to [\Sigma^\infty_+BF, \ko]\}.
\]
Since $C_2\subseteq F$ \Cref{prop:reg-goes-to-1} implies that we have a map
\[
RO(F)/(\ion{\ell} + \rho_F) \to \Pic(\LTE^{hF}).
\]
The source here is the quotient group of $RO(F)$ by the subgroups 
generated by $\ion{\ell}$ and the regular representation $\rho_F$. Since $F$ is a finite group,
\[
\mathrm{Im}\{[\Sigma^\infty_+BF, \kon{\ell}] \to [\Sigma^\infty_+BF, \ko]\} \subseteq  [\Sigma^\infty_+BF, \ko]
\]
is of finite index, so $RO(F)/(\ion{\ell} + \rho_F)$ is a finite abelian group. Our project is then to calculate the
image of a given representation $W$ in this group as well as its image in $ \Pic(\LTE^{hF})$.
\end{rem}

\begin{rem}\label{rem:more-char-classes}
We will use classical characteristic class arguments to calculate $RO(F)/\ion{\ell}$. Let $P_nX$ denote
the $n$th Postnikov section of $X$. Then there is an injection (which is often an isomorphism)
\[
RO(F)/\ion{\ell} \longr [\Sigma_+^\infty BF,P_{\ell-1}\ko\ ] \cong [BF,\ZZ \times P_{\ell-1}BO]. 
\]
For example, suppose $\ell = 8$. We have a tower of fibrations
\[
\xymatrix{
BO \langle 8 \rangle \ar[r] & BSpin \ar[r] \ar[d]_\lambda & BSO \ar[r] \ar[d]^{w_2} & BO \ar[d]^{w_1}\\
&K(\ZZ,4) & K(\ZZ/2,2) & K(\ZZ/2,1)
}
\]
with $w_1$ and $w_2$ the first and second Stiefel-Whitney classes of the universal bundles 
and $\lambda \in H^4(BSpin,\ZZ)$ a class so that that $2\lambda$ is the first Pontrjagin class. In particular,
$P_7BO\simeq P_7BSpin$ is a three stage Postnikov tower. Thus, we have
a filtration of $RO(F)/\ion{8}$ 
\begin{equation*}\label{eq:example-post}
\xymatrix@C=20pt{
0 \ar[r] & A_4 \ar[d]_\lambda \ar[r]^-\subseteq & A_2 \ar[r]^-\subseteq \ar[d]^{w_2}
& A_1 \ar[d]^{w_1} \ar[r]^-\subseteq & RO(F)/\ion{8}\ar[d]^{\mathrm{dim}}\\
&H^4(BF,\ZZ) & H^2(BF,\ZZ/2) & H^1(BF,\ZZ/2) & \ZZ
}
\end{equation*}
where $\mathrm{dim}$ assigns to any virtual representation its rank. If the Atiyah-Hirzebruch Spectral
Sequence for $\ko^\ast(\Sigma^\infty_+ BF)$ collapses, the vertical maps in this filtration will
be surjective.

The class $\lambda$ lies outside the standard list of characteristic classes, but we do
have the following result. Let $c_i$ denote the Chern classes.
\end{rem}

\begin{lem}\label{lem:ev-lambda} Let $\xi$ be a stable complex bundle over $X$ with the property that
$c_1(\xi) \equiv 0 \in H^2(X,\FF_2)$. Then we can choose a $Spin$ structure on $\xi$ and a fixed choice
of $Spin$ structure determines a class $d(\xi) \in H^2(X,\ZZ)$ with the property that $2d(\xi) = c_1(\xi)$.
For this Spin structure on $\xi$ we have
\begin{align*}
\lambda(\xi) = d(\xi)c_1(\xi) - c_2(\xi).
\end{align*}
Furthermore, for such bundles, the characteristic class $\lambda$ is additive; that is,
\[
\lambda(\xi_1 \oplus \xi_2) = \lambda(\xi_1) + \lambda(\xi_2).
\]
\end{lem}

\begin{proof} The first statement follows from examining what happens in integral cohomology
in the fiber sequence
\[
\xymatrix{
B \ar[r] & BU \ar[r]^-{c_1} & K(\ZZ/2,2).
}
\]
We can obtain the  equation for $\lambda$ by studying the universal example $\xi_0$ over $B$. Note that
$H^4(B,\ZZ) \cong \ZZ^2$ generated by $d(\xi_0)^2$ and $c_2(\xi_0)$.  Now use that
if $\xi$ is any complex bundle, then
\begin{equation*}
p_1(\xi) = -c_2(\xi \otimes \mathbb{C}) = -c_2(\xi \oplus \overline{\xi}) = c_1(\xi)^2 - 2c_2(\xi)
\end{equation*}
where $\overline{\xi}$ is $\xi$ with its conjugate complex structure. Additivity follows from 
this same formula by considering the universal example over $B \times B$.
\end{proof} 

\begin{rem}\label{rem:the-strategy-p} There is a $p$-complete variant on the constructions of 
\Cref{rem:the-strategy}. The unit map $S^0 \to \LTE$ extends to a unit map $S^0_p \to \LTE$,
so the map $F(\Sigma^\infty_+BF,\ko) \to \pic(\LTE^{hF})$ factors as a map
\[
F(\Sigma^\infty_+BF,\ko) \to F(\Sigma^\infty_+BF,\pic(S^0_p)) \to \pic(\LTE^{hF}).
\]
Note that by combining \eqref{eq:pic-equivs-sp} and \eqref{eq:pic-equivs-sp-1} we have that the natural
map
\[
\pic(\LTE^{hF}) \to \pic(\LTE)^{hF}
\]
induces an equivalence on $(-1)$-connected covers. 
If $X$ is any spectrum, define $L_p^{\geq 2} X$ by the homotopy push-out diagram
\[
\xymatrix{
X\langle 2 \rangle \ar[r] \ar[d] & X\langle 2 \rangle_p \ar[d]\\
X \ar[r] & L_p^{\geq 2} X
}
\]
where $Y_p$ is the $p$-completion of $Y$.  Note that $L_p^{\geq 2} X$ has homotopy groups
\[\pi_t L_p^{\geq 2} X \cong \begin{cases} \pi_t X & t=0,1 \\
(\pi_t X)_p & t\geq 2.
\end{cases}
\]
Since the homotopy groups of $\pic(S^0_p)$ are $p$-complete above dimension $1$, then
$L_p^{\geq 2} \pic(S^0_p) \simeq \pic(S^0_p)$ and the
map $\ko \to \pic(S^0_p)$ factors through a map $L_p^{\geq 2}\ko \to \pic(S^0_p)$.
The direct analog of \Cref{prop:guesstrue:2} is still true, with the same proof, and with $\ko$ and $\kon{\ell}$
replaced by $L_p^{\geq 2}\ko$ and $L_p^{\geq 2} \kon{\ell}$ as needed.
Thus there will be an integer integer $\ell$ so that the composition
\[
[\Sigma^\infty_+BF, L_p^{\geq 2} \kon{\ell}] \to [\Sigma^\infty_+BF, L_p^{\geq 2} \ko] \to
\pi_{0}\pic(\LTE)^{hF} = \pi_0\pic(\LTE^{hF})
\]
is zero. Indeed, we can choose the same integer $\ell$ as in the uncompleted case. In our examples, we will have
$\ell \geq 2$ and, in that case, $L_p^{\geq 2} \kon{\ell} = \kon{\ell}_p$.

If we let
\[
\ionp{\ell} = RO(F) \cap \mathrm{Im}\{[\Sigma^\infty_+BF, L_p^{\geq 2}\kon{\ell}] \to
[\Sigma^\infty_+BF, L_p^{\geq 2}\ko]\},
\]
then our variant of \Cref{prop:guesstrue:2} gives a map
\[
RO(F)/(\ionp{\ell} + \rho_F) \to \pi_0\pic(E^{hF})
\]
and we have an injection 
\[
RO(F)/\ionp{\ell} \longr [\Sigma_+^\infty BF,P_{\ell-1}L_p^{\geq 2}\ko].
\]
We again get a filtration of $RO(F)/\ionp{\ell}$. For example, if $\ell=8$ and $p=2$ we have only
a slight change:
\[
\xymatrix@C=20pt{
0 \ar[r] & A_4 \ar[d]_\lambda \ar[r]^-\subseteq & A_2 \ar[r]^-\subseteq \ar[d]^{w_2}
& A_1 \ar[d]^{w_1} \ar[r]^-\subseteq & RO(F)/\iontwo{8}\ar[d]^{\mathrm{dim}}\\
&H^4(BF,\ZZ_2) & H^2(BF,\ZZ/2) & H^1(BF,\ZZ/2) & \ZZ
}
\]
where now $H^4(BF,\ZZ_2) = \lim H^4(BF,\ZZ/2^n)$. If $p > 2$ and $\ell=8$ we have
\[
\xymatrix@C=20pt{
0 \ar[r] & A_4 \ar[d]_\lambda \ar[r]^-\subseteq 
& A_1 \ar[d]^{w_1} \ar[r]^-\subseteq & RO(F)/\ionp{8}\ar[d]^{\mathrm{dim}}\ .\\
&H^4(BF,\ZZ_p)  & H^1(BF,\ZZ/2) & \ZZ
}
\]
In all of these examples, the vertical maps will be surjective if the Atiyah-Hirzebruch Spectral Sequence
for $(L_p^{\geq 2}\ko)^\ast(BF)$ collapses. 
\end{rem}


\section{The Spanier-Whitehead duals of $\LTE^{hF}$: examples from elliptic curves}\label{sec:44}

We now focus our attention at height $n=2$ and the primes $p=2$ and $p=3$. In both cases we take a formal
group of height $2$ obtained from a supersingular elliptic curve. We wish to give a concrete
calculation of the Spanier-Whitehead dual $D(\LTE^{hF})$  where $F$ is a finite subgroup of $\GG_2$. 
Our main interest is when $F$ is actually the automorphisms of the chosen elliptic curve.
The results are in \Cref{thm:dualn=p=3} and \Cref{thm:dualn=p=2}.
 
At either prime, the basic case will be when $F \subseteq \SS_2 = \cO_2^\times$ is a subgroup containing
a maximal finite $p$-torsion subgroup $F_0$. At $p=2$, $F_0$ is isomorphic to the quaternion group of order
$8$ and at $p=3$ we have $F_0$ is cyclic of order $3$. In both cases \Cref{cor:what-we-use} applies and we have
an $F$-equivariant $\LTK$-equivalence $D\LTE = S^{-\gg} \wedge \LTE$ where $S^\gg$ is the linear
dualizing sphere. By Tate vanishing \cite{GreenSad} we know that 
\[
D(\LTE^{hF}) \simeq D(\LTE)^{hF} \simeq (S^{-\gg} \wedge \LTE)^{hF}. 
\]
We will be able to use \Cref{prop:recog-princ} to write $S^\gg$ as the $p$-completion of
a representation sphere $S^V$ and hence we have
\begin{equation}\label{eq:master-dual}
D(\LTE^{hF}) \simeq D(\LTE)^{hF} \simeq (S^{-V} \wedge \LTE)^{hF}. 
\end{equation}
The strategy then developed in \Cref{rem:the-strategy-p} applies to complete the calculation. In our examples, the subgroups 
$F$ are such that $ \Pic(\LTE^{hF})$ is cyclic generated by $\Sigma \LTE^{hF}$; therefore,
there is an integer $k$ so that 
\[
(S^{-V} \wedge \LTE)^{hF} \simeq \Sigma^k \LTE^{hF}.
\]
In both cases we will prove $k=44$.

We can extend these results to subgroups of the larger Morava stabilizer
group $\GG_2 = \SS_2 \rtimes \Gal(\FF_{p^2}/\FF_p)$. At both
the prime $2$ and $3$ there is a finite subgroup $G \subseteq \GG_2$ so that
\[
\LTE^{hG} \simeq L_{\LTK}\mathbf{tmf}
\]
where $\mathbf{tmf}$ is the Hopkins-Miller spectrum of topological modular forms. We will then have
\[
D(L_\LTK \mathbf{tmf}) \simeq \Sigma^{44} L_\LTK \mathbf{tmf}.
\]
This recovers results of Behrens \cite{Beh44} and Bobkova \cite{BobkovaSW} at $p=3$ and $p=2$ respectively.

\subsection{The case $n=2$ and $p=3$} 
We first consider the case $p=3$. There is a supersingular elliptic curve $C$ with Weierstrass equation
\begin{equation}\label{eq:ellat3}
y^2 = x^3-x. 
\end{equation}
While defined over $\FF_3$, we work over $\FF_9$, and the formal group $F_C$ of this curve is a formal group of height $2$ 
over that field. Because $C$ is supersingular, the endomorphism
ring $\cE$ of $C$ over $\FF_9$ is a maximal order in a quaternion algebra ramified only at $3$ and $\infty$. The
completion of $\cE$ at $p=3$ is the endomorphism ring $\cO_2$ of $F_C$. Thus $\cE \subseteq \cO_2$ is
a lattice and $\cE/3\cE \cong \cO_2/3\cO_2$. 

Let $i \in \FF_9$ be a fourth root of unity, so $i^2 = -1$. This induces
an automorphism
\[
(x,y) \mapsto (i^2x,i^3y) = (-x,-iy)
\]
of $C$ which we will also call $i$. There are two automorphisms of exact order $3$ given by
\[
(x,y) \mapsto (x\pm 1,y).
\]
We will fix one in a moment, after giving a bit more of the structure of $\cE$.

The Frobenius $\phi$ given by
\[
(x,y) \mapsto (x^3,y^3)
\]
defines an endomorphism of $C$ as well. Since $C$ has four points over $\FF_3$ (including the point at $\infty$) 
we have that
\[
\phi^2 = -3
\]
as an endomorphism of $C$. See Theorem 4.10 of \cite{Wash}, for example. Note that $\phi i = -i\phi$.
The element of order $3$ in the automorphism group of $C$ can be chosen to be
\[
\sigma = -\frac{1}{2}(1+\phi). 
\]
In fact, if $\sigma$ is any element of exact order $3$ then $\sigma^2 + \sigma + 1 = 0$, so $(1+2\sigma)^2 = -3$. Thus
$1+2\sigma = \pm\phi$.
The automorphisms of $C$ are generated
by $i$ and $\sigma$ and the subgroup $C_3$ of the automorphism group generated by $\sigma$ is normal. The group
\[
G_{12} = \Aut(C) \cong C_3 \rtimes C_4 \subseteq \cE^\times \subseteq \cO_2^\times
\]
defines a maximal finite subgroup of $\cO_2$ which contains $3$ torsion.

For our calculations we will follow the outline of \Cref{rem:the-strategy-p}, and we will use the notation
established there. Let $w_1$ be the first Stiefel-Whitney class and $\lambda$  characteristic class for spin
bundles discussed in \Cref{lem:ev-lambda}; recall that $2\lambda$ is the first Pontrjagin class. 

The inclusion $C_4 \to G_{12}$ induces an isomorphism
\[
E(x) \otimes \FF_2[y] \cong H^\ast (C_4,\FF_2) \cong H^\ast(G_{12},\FF_2)
\]
where $x$ is in degree $1$ and $y$ is the second-order Bockstein on $x$. In addition, the inclusion
$C_3 \to G_{12}$ defines an isomorphism
\[
\ZZ_3[z]/(3z) \cong H^\ast(G_{12},\ZZ_3) \cong H^\ast(C_3,\ZZ_3)^{C_4}. 
\] 
where we chose $z$ in degree $4$ to be the square of either generator of $H^2(C_3,\ZZ_3)$. 

Thus to calculate the characteristic classes $w_1$ and $\lambda$ of representations of $G_{12}$ we can restrict those representations to
the subgroups $C_3$ and $C_4$. Note that every non-trivial irreducible real representation of $C_3$ is the restriction of a complex 
representation and, hence, \Cref{lem:ev-lambda} can be used to calculate $\lambda$. 

\begin{rem}\label{rem:cclasses-at-3-2} We now calculate characteristic classes of the needed
real representations of $C_3$, $C_4$, and  $G_{12}$.

We begin with the regular representation $\rho_{G_{12}}$. Restricted to $C_4$ we have an isomorphism
\[
\rho_{G_{12}} \cong3 \rho_{C_4} = 3(\mathbf{1}_\RR + \sigma_\RR + \gamma_4)
\]
where $\sigma_\RR$ is the real sign representation and $\gamma_4$ is the $2$-dimensional real representation given by
rotation by $90$ degrees. The latter is the restriction of a complex representation so $w_1(\gamma_4)=0$ and we have
\begin{equation}\label{eq:w1pho3}
w_1(\rho_{G_{12}}) = x \in H^1(G_{12},\FF_2).
\end{equation}
Restricted to $C_3$ we have an isomorphism
\[
\rho_{G_{12}} \cong 4 \rho_{C_3} \cong 4(\mathbf{1}_\RR + \gamma_3)
\]
where $\gamma_3$ is the unique non-trivial $2$-dimensional real representation of $C_3$. This is the restriction of a
one-dimensional complex representation with non-zero first Chern class in $H^2(C_3,\ZZ) \cong \ZZ/3$. 
Hence by \Cref{lem:ev-lambda} we have
\begin{equation}\label{eq:lambdapho3}
\lambda(\rho_{G_{12}}) = -z^2 \in H^4(C_3,\ZZ_3) \cong \ZZ/3.
\end{equation}

Next we examine the conjugation action of $G_{12}$ on $\cE$, the endomorphism ring of $C$. Let
\[
\cE_0 = \ZZ \oplus \ZZ i \oplus \ZZ\phi \oplus \ZZ i\phi \subseteq \cE.
\]
This inclusion is not equality, as $\sigma \not\in \cE_0$, but since $\cE$ is of rank $4$ over $\ZZ$ 
we have that $\RR \otimes \cE_0 = \RR \otimes \cE$ and we
can use the conjugation action of $G_{12}$ on $\cE_0$ to determine the real representation $\RR \otimes \cE$.

Restricted to $C_4 \subseteq G_{12}$, the subgroup generated by $i$, there is an isomorphism of
$C_4$-representations
\[
\RR \otimes \cE \cong \mathbf{1}_\CC \oplus \sigma_\CC
\]
where $\sigma_\CC$ is the complex sign representation. Thus
\begin{equation}\label{eq:w1E}
w_1(\RR \otimes \cE) = 0. 
\end{equation}

If we restrict to the subgroup $C_3 \subseteq G_{12}$ generated $\sigma$ then there is an isomorphism of
$C_3$-representations
\[
\RR \otimes \cE \cong \mathbf{1}_\CC \oplus \gamma_3
\]
where $\mathbf{1}_\CC$ is the trivial $2$-dimensional real representation and $\gamma_3$ is the unique non-trivial $2$-dimensional 
real representation. Both are restrictions of complex representations. Thus, again using that $\gamma_3$ has
non-trivial first Chern class and using \Cref{lem:ev-lambda}, we have
\begin{equation}\label{eq:lambdaE}
\lambda(\RR \otimes \cE) = -z^2 \in H^3(C_3,\ZZ_3) \cong \ZZ/3.
\end{equation}
\end{rem}

\begin{prop}\label{prop:RIforn=p=3-1} (1) Let $G_{12} \subseteq \SS_2$ be the automorphism group of the supersingular
elliptic curve $y^2 = x^3-x$ over $\FF_9$. The composite mapping
\[
[\Sigma^\infty_+BG_{12}, L_3^{\geq 2} \kon{8}] \to [\Sigma^\infty_+BG_{12},  L_3^{\geq 2} \ko] \to
\pi_{0}\pic(\LTE)^{hG_{12}}
\]
is zero. 

(2) There is an isomorphism
\[
\xymatrix{
\psi \colon RO(G_{12})/\ionthree{8} \ar[r]^-\cong & \ZZ \oplus \ZZ/2 \oplus \ZZ/3
}
\]
sending a representation $V$ to $(\mathrm{dim}(V),a,b)$ with
\begin{align*}
w_1(V) &= ax \in H^1(C_4,\FF_2)\\
\lambda(V) &= bz \in H^4(C_3,\ZZ_3).
\end{align*} 
\end{prop}

\begin{proof} For Part (1) we use (the evident variant of) \Cref{prop:guesstrue:2}. See \Cref{rem:the-strategy-p}.
By Theorem 8.1.3 (see also Figure 6) of \cite{AkhSto} we have that in the spectral sequence
\[
H^s(G_{12},\pi_t\pic(\LTE)) \Longrightarrow \pi_{t-s}\pic(\LTE)^{hG_{12}}
\]
 $E_\infty^{s,s} = 0$ for $s \geq 6$. Note also that $\kon{8}= \kon{6}$.

Part (2) follows from the filtration of $RO(G_{12})/\ionthree{8}$ given in \Cref{rem:the-strategy-p}. Note that the composite
mappings
\[
\xymatrix@R=10pt{
RO(G_{12}) \ar[r] & RO(C_4) \ar[r]^-{w_1} & H^1(C_4,\ZZ/2) \cong \ZZ/2\\
RO(G_{12}) \ar[r] & RO(C_3) \ar[r]_-{\lambda} & H^1(C_3,\ZZ_3) \cong \ZZ/3
}
\]
are both onto by \eqref{eq:w1E} and \eqref{eq:lambdaE}.
\end{proof}

We now give a more specific calculation. Let $\psi$ be the map of part (2) of \Cref{prop:RIforn=p=3-1}.

\begin{prop}\label{prop:RIforn=p=3-3}  (1) If $\rho_{G_{12}} \in RO(G_{12})$ is the regular representation, then
\[
\psi(\rho_{G_{12}} ) = (12,1,-1).
\] 

(2) The group $ RO(G_{12})/(\ionthree{8} + \rho_{G_{12}})$ is generated by the trivial $1$-dimensional real representation;
this choice of  generator determines an isomorphism
\begin{equation}\label{eq:iso138}
\ZZ/72 \cong RO(G_{12})/(\ionthree{8} + \rho_{G_{12}}).
\end{equation}
Furthermore, the $J$-homomorphism
\[
J_\LTE^{G_{12}}\colon RO(G_{12})/(\ionthree{8} + \rho_{G_{12}}) \longr \Pic(\LTE^{hG_{12}})
\]
is an isomorphism. 
\end{prop}

\begin{proof} Part (1) follows from \eqref{eq:w1pho3} and \eqref{eq:lambdapho3}. The isomorphism
\eqref{eq:iso138} is then immediate. The fact the $J$-homomorphism
is an isomorphism then follows from the fact that $\LTE^{hG_{12}}$ has periodicity $72$; that
is $\Pic(\LTE^{hG_{12}}) \cong \ZZ/72$ generated by $\Sigma \LTE^{hG_{12}}$. See \cite{AkhSto}. 
\end{proof}

\begin{prop}\label{prop:class-of-key-repn=p=3} Let $\cE$ be the endomorphism ring of $C$ and give
$\RR \otimes \cE$ the conjugation action by $G_{12}$. Then 
\[
\RR \otimes \cE = -44\cdot \mathbf{1}_{\mathbb{R}}
\]
in $RO(G_{12})/(\ionthree{8} + \rho_{G_{12}})$.
\end{prop}

\begin{proof} By \eqref{eq:w1E} and \eqref{eq:lambdaE} we have 
\[
\psi(\RR \otimes \cE) = (4,0,-1) \in \ZZ \oplus \ZZ/2 \oplus \ZZ/3.
\]
By \Cref{prop:RIforn=p=3-3}, part (1) we have $\psi(\rho_{G_{12}}) = (12,1,-1)$. Since 
\[
(4,0,-1) - 4(12,1,-1) = (-44,0,0)  \in \ZZ \oplus \ZZ/2 \oplus \ZZ/3
\]
the result follows. 
\end{proof} 

We now have a calculation of the Spanier-Whitehead duals to $\LTE^{hF}$. Let
\[
G_{24} = G_{12} \rtimes \Gal(\FF_9/\FF_3) \subset \cO_2  \rtimes \Gal(\FF_9/\FF_3) = \GG_2.
\]
We have, here at $p=3$,
\begin{align*}
\LTE^{hG_{24}} \simeq L_\LTK \mathbf{tmf}\\
\end{align*}
where $\mathbf{tmf}$ is the spectrum of topological modular forms.

\begin{thm}\label{thm:dualn=p=3} Let $p=3$ and $F \subseteq G_{24} \subseteq \GG_2$ for $\GG_2=\Aut(\F_9, F_C)$ the stabilizer group associated to the formal group law $F_C$ of a supersingular elliptic curve $C$ with Weierstrass equation \eqref{eq:ellat3}. Then
\[
D(\LTE^{hF}) \simeq \Sigma^{44} \LTE^{hF}.
\]
\end{thm} 

\begin{proof} First suppose $F \subseteq G_{12}$. Then, as in \eqref{eq:master-dual}, we have
\[
D(\LTE^{hF}) \simeq (S^{-(\RR \otimes \cE)} \wedge \LTE)^{hF}
\]
and the result follows from \Cref{prop:RIforn=p=3-3} and \Cref{prop:class-of-key-repn=p=3}. 

The other possibility is that the composition $F \to G_{24} \to \Gal(\FF_9/\FF_3)$ is onto. Let $F_0$
be the kernel of this map and write $\Gal$ for $\Gal(\FF_9/\FF_3)$. We know from \cite[Lemma 1.37]{BobkovaGoerss}
that there is a $\Gal(\FF_9/\FF_3)$ equivariant equivalence
\[
\Sigma_+^\infty {\Gal} \wedge \LTE^{hF} \simeq \LTE^{hF_0}. 
\] 
We now have 
\begin{align*}
D(\LTE^{hF}) &\simeq D(\LTE)^{hF}\\
&\simeq [D(\LTE)^{hF_0}]^{h{\Gal}} \simeq (\Sigma^{44}E^{hF_0})^{h{\Gal}}\\
&\simeq [\Sigma_+^\infty {\Gal} \wedge \Sigma^{44} \LTE^{hF}]^{h{\Gal}}\\
&\simeq \Sigma^{44}\LTE^{hF}.\qedhere
\end{align*}
\end{proof}

\subsection{The case $n=2$ and $p=2$}

We proceed exactly as in the case of $n=2$ and $p=3$, working with the endomorphism ring of a supersingular
elliptic curve over $\FF_4$.  This is a recapitulation of ideas already laid out by the first author in \cite{BeaudryTowards};
see in particular Lemma 2.4.3 of that paper.

Over $\FF_4$, there is a supersingular elliptic curve $C$ with Weierstrass equation
\begin{equation}\label{eq:ellat2}
y^2 + y = x^3. 
\end{equation}
The endomorphism ring $\cE$ of $C$ over $\FF_4$ is a maximal order in a quaternion algebra ramified only at $2$ and $\infty$. 
It is possible to be quite explicit.

Choose $\omega \in \FF_4$ with $\omega^2 + \omega + 1 = 0$; that is, $\omega$ is a primitive third root of unity.
Then $C$ has an automorphisms $\omega$ and $i$ with
\begin{align*}
\omega(x,y) &= (\omega x, y)\\
i(x,y) &= (x+1,y+x+\omega).
\end{align*}
There is a slight abuse of notation here with the symbol $\omega$. Set $j = \omega i \omega^2$ and $k = \omega^2 i \omega$. These elements
generate a normal subgroup of $\Aut(C)$ isomorphic to the quaternion group $Q_8$ of order $8$. The automorphism $\omega$
defines a cyclic subgroup $C_3 \subseteq \Aut(C)$ of order $3$ and there is an isomorphism
\[
G_{24}= Q_8 \rtimes C_3 \cong \Aut(C).
\]
The group $C_3$ acts on $Q_8$ by cyclicly permuting $i$, $j$, $k$. The element $\omega \in \cE$ can be written
\[
\omega = -\frac{1+i+j+k}{2}
\]
and we have
\[
\cE=  \ZZ \oplus \ZZ i  \oplus \ZZ j \oplus \ZZ\frac{1+i+j+k}{2}.
\]
The completion of $\cE$ at $2$ is the endomorphism ring $\cO_2$ of the formal group associated to $C$, which is
necessarily of  height $2$. Then $G_{24} \subset \cO_2$ is a choice a maximal finite subgroup containing a 
subgroup isomorphic to $Q_8$. 

\begin{rem}\label{rem:set-upn=p=2} Let  
\[
\HH = \RR\{1,i,j,k\}/(i^2=j^2=k^2 = -1, ij=k = -ji\}
\] 
be the quaternion algebra. Up to isomorphism, this is the unique $4$-dimensional associative division algebra over the real
numbers. The group $Q_8$ is a group of units in $\HH$ and left multiplication by $Q_8$ on $\HH$ gives,
up to isomorphism, the unique irreducible $4$-dimensional representation of $Q_8$. 

We have an evident inclusion
\[
\cE \subseteq \HH
\]
which is closed under both the left action and the conjugation action by $G_{24}$. Furthermore
\[
\RR \otimes \cE = \HH. 
\]
Since $\cE/2\cE \cong \cO_2/2\cO_2$, we find that we are exactly
in the situation of \Cref{prop:recog-princ} with $V = \HH_{ad}$ where $\HH_{ad}$ is $\HH$ with 
its conjugation action by $G_{24}$. Thus the main goal is to analyze the homotopy type of 
\[
(S^{-\HH_{ad}} \wedge \LTE)^{hG_{24}} \in \Pic(\LTE^{hG_{24}}). 
\]  
\end{rem}

\begin{rem}\label{rem:some-cohs-at-2}We will need to know the cohomology of $Q_8$ and $G_{24}$. We have that
$Q_8/[Q_8,Q_8] \cong \ZZ/2 \times \ZZ/2$ where we choose the residue classes of $i$ and $j$ as the generators. Then 
\[
H^\ast (Q_8,\FF_2) \cong A \otimes \FF_2[P]
\]
where $P \in H^4(Q_8,\FF_2)$ and $A$ is the $3$-dimensional Poincar\'e duality algebra 
\[
A = \FF_2[a,b]/(a^2 +ab+b^2,a^2b+ab^2).
\]
generated by classes $a$ and $b$ of degree $1$ dual to $i$ and $j$ respectively. A generator of the group $C_3$ acts on
$H^1(Q_8,\FF_2)$ by sending $a$ to $b$ and $y$ to $a+b$, so there is an ismorphism
\[
H^\ast (G_{24},\FF_2) \cong H^\ast (Q_8,\FF_2)^{C_3} \cong E(Q) \otimes \FF_2[P]
\]
where $Q \in H^3(Q_8,\FF_2)$ is the top class in $A$. We also have
\[
H^\ast (G_{24},\ZZ_{(2)}) \cong \ZZ[P]/(8P)
\]
where, by abuse of notation, $P \in H^4 (G_{24},,\ZZ_{(2)}) $ is an integral class which reduces to $P \in H^4(G_{24},\FF_2)$. 
We will give a more specific generator for $H^4 (G_{24},,\ZZ_{(2)})$ below in \Cref{lem:coho-G24}.
\end{rem}

\begin{rem}\label{rem:rep-quats} We review the representation theory of $Q_8$ and $G_{24}$. 

We have defined two real representations $\HH$ and $\HH_{ad}$. 
There are also three isomorphism classes of non-trivial $1$-dimensional real representations of $Q_8$. Each of the
elements $i$, $j$,  $k$ generates a subgroup of order $4$ in $Q_8$; taking the quotient by these subgroups
in turn defines homomorphisms $Q_8 \to \{ \pm 1\}=C_2$ and representations $\chi_i$, $\chi_j$, and $\chi_k$
by restricting the sign representation of $C_2$. 
If we write $\mathbf{1}_\R$ for the trivial representation, then the regular representation of $Q_8$ decomposes as
\[
\rho_{Q_8} \cong \mathbf{1}_{\mathbb{R}} \oplus \chi_i  \oplus \chi_j  \oplus \chi_k \oplus \HH.
\]
As a real representation of $Q_{8}$ we have
\[
\HH_{ad} \cong  \mathbf{1}_{\mathbb{R}} \oplus \chi_i  \oplus \chi_j \oplus \chi_k.
\]
and hence
\begin{equation}\label{eq:decom-refq8}
\rho_{Q_{8}} \cong \HH_{ad} \oplus \HH. 
\end{equation}

Because of the symmetries in $Q_8$, the representation $\chi_i  \oplus \chi_j  \oplus \chi_k$ and $\HH_{ad}$ can
be given the structure of Spin representations. To see this, let $w(\xi) = 1 + w_1(\xi) + w_2(\xi) + \cdots$ be
the total Stiefel-Whitney class. Then, using the notation of \Cref{rem:some-cohs-at-2}, we have that
\[
w(\HH_{ad}) = w(\chi_i  \oplus \chi_j  \oplus \chi_k) = (1+a)(1+b)(1+(a+b)) = 1 \in H^\ast(Q_8,\FF_2).
\]

The representation $\HH$ of $Q_8$ is the restriction of an irreducible complex representation. If we use
the right action of $\CC$ on $\HH$ to give $\HH$ the structure of a complex vector space, then the action
$Q_8$ on $\HH$ is through complex linear transformations given by the matrices
\begin{align}\label{eq:HSU2}
i &=  \left(\begin{matrix}i & 0 \\ 0 & -i \end{matrix}\right),  & j &=  \left(\begin{matrix}0 & 1 \\ -1 & 0 \end{matrix}\right),
&  k&=  \left(\begin{matrix} 0 & -i \\ -i & 0 \end{matrix}\right).
\end{align}
This defines an inclusion $h\colon Q_8 \to SU(2)$. 

The representation $\HH_{ad}$ is a stabilization of the
restriction of the adjoint representation of $SU(2)$ along the inclusion $h$.
Concretely,  the Lie algebra $\su2$ of  $SU(2)$ is the real vector space of $2\times 2$ skew Hermitian complex
matrices $A$ of trace $0$; thus a $2 \times 2$ complex matrix $A$ is in $\su2$ if
$\overline{A} + A^t=0$ and $\mathrm{trace}(A) = 0$. Thus
\[  \mathfrak{su}(2) \cong \left\{   \left(\begin{matrix}bi & -\overline{z} \\ z & -bi \end{matrix}\right) \ |\ 
b \in \R, \ z \in \CC \right\}.
\]
The group $SU(2)$ acts on the Lie algebra by conjugation; this is the adjoint representation of $SU(2)$.
There is an isomorphism of real representations of $Q_8$ 
\[
\mathbf{1}_{\mathbb{R}} \oplus \mathfrak{su}(2) \cong \HH_{ad} .
\]
\end{rem} 

\begin{rem}\label{rem:clarifiication} Let $V$ be a $G$-representation for some
compact Lie group $G$ and let $\xi_V$ be the bundle 
\[
EG \times_{G} V \longr BG. 
\]
Thus if $V$ is a complex representation, we get Chern classes $c_i(V)=c_i(\xi_V) \in H^\ast (BG,\ZZ)$. For example, if $V = \CC^n$
with its standard left action by $U(n)$ then this bundle is the dual of the tautological bundle
over $BU(n)$. Hence
\[
c_i(\xi_{\CC^n}) = (-1)^i c_i
\]
where $c_i$ denotes the universal Chern class.
Note that if a homomorphism $\varphi\colon G \to U(n)$ defines the representation $V$, then there is an isomorphism
of bundles over $BG$ 
\[
\xi_V \cong B\varphi^\ast \xi_{\CC^n}.
\]
\end{rem}

As preparation for our calculations in \Cref{prop:RIforn=p=2} and
\Cref{prop:class-of-key-repn=p=2} we next calculate some characteristic classes.
Since $BSU(2)$ is $3$-connected every vector bundle over $BSU(2)$ has a unique Spin
structure and the characteristic class $\lambda$ of \Cref{lem:ev-lambda} is unambiguously defined. 

\begin{lem}\label{rem:char-quats} (1) We have an isomorphism
\[
\ZZ[y] \cong H^\ast (BSU(2),\ZZ)
\]
where
\[
y = \lambda(\xi_{{\CC^2}}) = -c_2(\xi_{{\CC^2}}) \in H^4 (BSU(2),\ZZ)
\]
is the characteristic class determined by the unique Spin structure on the bundle $\xi_{\CC^2}$ associated
to the standard left action of $SU(2)$ on $\CC^2$.

(2) Let $\mathfrak{su}(2)$ be the adjoint representation of $SU(2)$. Then
\[
\lambda(\mathfrak{su}(2)) = \lambda(\mathbf{1}_{\mathbb{R}} \oplus \mathfrak{su}(2)) =  2y
\]
\end{lem}

\begin{proof} Since $H^\ast (BSU(2),\ZZ) \cong \ZZ[c_2]$ part (1) follows from \Cref{rem:clarifiication} 
and an application of the formula 
\[
\lambda(\xi) = d(\xi)c_1(\xi) - c_2(\xi). 
\]
of \Cref{lem:ev-lambda}.

For Part (2), let $g: BU(1) \longr BSU(2)$ be the map defined by the inclusion of Lie groups $U(1) \to SU(2)$
\[
z \longmapsto \left(\begin{matrix}z & 0 \\ 0 & \overline{z} \end{matrix}\right).
\]
The map $g$ classifies the bundle $\xi = \overline{\gamma}_1 \oplus \gamma_1$, where $\gamma_1$ is the 
tautological line bundle. Since
$c_2(\xi) = -c_1(\gamma_1)^2$, we have 
\[
g^\ast y = \lambda (\overline{\gamma}_1 \oplus \gamma_1) = c_1^2.
\]
In addition, $g^\ast \colon H^4(BSU(2),\ZZ)  \to H^4(BU(1),\ZZ) $ is an isomorphism.\footnote{The map $U(1) \to SU(2)$
is the inclusion of the maximal torus. The Weyl group is $C_2$ and we have explicitly written
down the canonical isomorphism $H^\ast(BSU(2),\ZZ) \cong H^\ast(BU(1),\ZZ)^{C_2}$.} Thus to calculate
$\lambda(\mathbf{1}_{\mathbb{R}} \oplus \mathfrak{su}(2))$ we can restrict to $U(1)$.

If we identify $\mathbf{1}_{\mathbb{R}}$ with the $2\times 2$ diagonal matrices 
\[
\left\{  \left(\begin{matrix}a & 0 \\ 0 & a \end{matrix}\right)  : a \in \RR \right\}
\]
we can then identify the $U(1)$-representation $\mathbf{1}_{\mathbb{R}} \oplus \mathfrak{su}(2) $ as the
direct sum $\mathcal{X} \oplus \mathcal{Y}$ of two $1$-dimensional complex representations with
\begin{align*} 
\mathcal{X} =\left\{   \left(\begin{matrix}\alpha & 0 \\ 0 & \overline{\alpha} \end{matrix}\right) : \alpha \in \CC \right\}
\qquad \mathrm{and}  \qquad
\mathcal{Y} = \left\{   \left(\begin{matrix}0 & -\overline{\alpha} \\ \alpha & 0 \end{matrix}\right) : \alpha \in \CC \right\} .
\end{align*}
The conjugation action of $U(1)$ on $\mathcal{X}$ is trivial and the conjugation action of $z \in U(1)$ on
$\mathcal{Y}$ is by multiplication by $z^2$. Thus if $\zeta$ is the bundle over $BU(2)$ defined
by $\mathbf{1}_{\mathbb{R}} \oplus \mathfrak{su}(2)$ we have $g^\ast \zeta = \mathbf{1}_{\mathbb{R}} \oplus \gamma_1^{-\otimes 2}$.
 Thus
\[
g^\ast \lambda(\mathbf{1}_{\mathbb{R}} \oplus \mathfrak{su}(2)) = (1/2)(-c_2(\gamma_1))^2 = 2g^\ast y.
\] 
Since $\lambda(\xi \oplus \zeta) = \lambda(\xi) + \lambda(\zeta)$ we also get the formula for $\lambda(\mathfrak{su}(2))$.
\end{proof} 

\begin{lem}\label{lem:coho-G24} Let $\HH$ be the real representation of $G_{24}$ extending the left action of
$Q_8$ on the quaternions. Then 
\[
H^\ast (G_{24},\ZZ_2) \cong \ZZ_2[\lambda(\HH)]/8\lambda(\HH).
\]
Furthermore $\lambda(\HH_{ad}) = 2\lambda(\HH)$. 
\end{lem}

\begin{proof} By \Cref{rem:some-cohs-at-2} we have that $H^\ast (G_{24},\ZZ_2) \cong \ZZ_2[P]/(8P)$, where
$P$ has degree $4$. Thus we need only show $\lambda(\HH)$ generates $H^4(G_{24},\ZZ_2)$.
Let $C_i \subseteq G_{24}$ be the subgroup generated by $i$.
Since the restriction map
\[
r^\ast: H^4(G_{24},\ZZ_2) \longr H^4(C_i,\ZZ_2) \cong \ZZ/4
\]
is onto, we need only check that $\lambda(\HH)$ generates $H^4(C_i,\ZZ_2)$. 

We can use the constructions of \Cref{rem:rep-quats} to produce a commutative diagram
\[
\xymatrix{
BC_i \ar[d]_r \ar[r]^f & BU(1) \ar[d]^g\\
BQ_8 \ar[r]_-h & BSU(2).
}
\]
We have $h^\ast \xi_{\CC^2} = \xi_\HH$, $g^\ast \xi_{\CC^2} = \overline{\gamma}_1 \oplus \gamma_1$,
and $f$ is induced by the representation $\gamma_i$ of a cyclic group of order $4$ on $\CC$ given by multiplication by $i$. 
If $y =\lambda(\xi_{\CC^2}) \in H^4(BSU(2),\ZZ_2)$, then we have
\begin{align*}
\lambda(\HH) &= h^\ast y \in H^4(Q_8,\ZZ_2)\\
g^\ast y &= c_1(\gamma_1)^2 \in H^4(BU(1),\ZZ_2).
\end{align*}
Since $H^2(BC_i,\ZZ_2) \cong \ZZ/4$ generated by $c_1(\gamma_i)$ we have that 
\[
r^\ast \lambda(\HH) = f^\ast g^\ast y = c_1(\gamma_i)^2 \in H^4(C_i,\ZZ_2)
\]
is a generator, as needed. 

The calculation of $\lambda(\HH_{ad})$ follows from Part (2) of \Cref{rem:char-quats}. 
\end{proof} 

We are now ready to give our calculations. We will follow the outline of \Cref{rem:the-strategy-p}, and we will use the notation
established there.

\begin{prop}\label{prop:RIforn=p=2} (1) Let $G_{24} \subseteq \SS_2$ be the automorphism group of the supersingular
elliptic curve $y^2 + y = x^3$ over $\FF_4$. The composite mapping
\[
[\Sigma^\infty_+BG_{24}, L_2^{\geq 2} \kon{8}] \to [\Sigma^\infty_+BG_{24},  L_2^{\geq 2} \ko] \to \pi_{0}\pic(\LTE)^{hG_{24}}
\]
is zero. 

(2) There is an isomorphism
\[
\psi \colon RO(G_{24})/\iontwo{8} \cong \ZZ \oplus \ZZ/8
\]
sending a representation $W$ to $(\mathrm{dim}(W),k)$ where $\lambda(W) = k\lambda(\HH) \in H^4(G_{24},\ZZ_2)$. 
\end{prop}

\begin{proof}  For Part (1) we again use the evident variant of \Cref{prop:guesstrue:2}. See \Cref{rem:the-strategy-p}.
By Theorem 8.2.2 (see also Figure 7) of \cite{AkhSto} we have that in the spectral sequence
\[
H^s(G_{24},\pi_t\pic(\LTE)) \Longrightarrow \pi_{t-s}\pic(\LTE)^{hG_{24}}
\]
$E_\infty^{s,s} = 0$ for $s \geq 8$.

Part (2) follows from the filtration of $RO(G_{24})/\iontwo{8}$ given in \Cref{rem:the-strategy-p}. Note that
by \Cref{lem:coho-G24}
\[
H^1(G_{24},\ZZ/2) = 0 = H^2(G_{24},\ZZ/2) 
\]
and $\lambda \colon RO(G_{24}) \to H^4(G_{24},\ZZ_2) \cong \ZZ/8$ is onto. 
\end{proof}

\begin{prop}\label{prop:RIforn=p=2-2} (1) If $\rho_{G_{24}} \in RO(G_{24})$ is the regular representation, then
\[
\psi(\rho_{G_{24}} ) = (24,1).
\]

(2) The group $ RO(G_{24})/(\iontwo{8} + \rho_{G_{24}})$ is generated by the $1$-dimensional real representation
$\mathbf{1}_{\mathbb{R}}$; this choice of generator determines an isomorphism
\[
\ZZ/192 \cong RO(G_{24})/(\iontwo{8} + \rho_{G_{24}}).
\]
Furthermore, the $J$-homomorphism
\[
J_\LTE^{G_{24}}\colon RO(G_{24})/(\iontwo{8} + \rho_{G_{24}}) \longr \Pic(\LTE^{hG_{24}})
\]
is an isomorphism. 
\end{prop}

\begin{proof}  We must calculate $\lambda(\rho_{G_{24}})$.
The inclusion $Q_8 \subseteq G_{24}$ defines an isomorphism
\[
H^4(G_{24},\ZZ_2) \cong H^4(Q_8,\ZZ_2) \cong \ZZ/8.
\]
Restricted to $Q_8$, we have $\rho_{G_{24}} = \rho_{Q_{8}}^{\oplus 3}$. We have, by \eqref{eq:decom-refq8} and
\Cref{lem:coho-G24}
\[
\lambda(\rho_{Q_{8}}) = \lambda(\HH_{ad}) + \lambda(\HH) = 3\lambda(\HH),
\]
whence
\[
\psi(\rho_{G_{24}}) = (24,9) = (24,1). 
\]
Thus  $\ZZ/192 \cong RO(G_{24})/(\iontwo{8} + \rho_{G_{24}})$ generated by $\mathbf{1}_{\mathbb{R}}$. 

The fact the $J$-homomorphism
is an isomorphism then follows from the fact that $\LTE^{hG_{24}}$ has periodicity $192$; that
is $\Pic(\LTE^{hG_{24}}) \cong \ZZ/192$ generated by $\Sigma \LTE^{hG_{24}}$. See \cite{AkhSto}. 
\end{proof} 

\begin{prop}\label{prop:class-of-key-repn=p=2} 
We have an equation
\[
\HH_{\mathrm{ad}} \equiv -44\cdot \mathbf{1}_{\mathbb{R}}
\]
in $RO(G_{24})/(\iontwo{8} + \rho_{G_{24}})$.
\end{prop}

\begin{proof} By \Cref{lem:coho-G24} we have 
\[
\psi(\HH_{ad}) = (4,2) \in \ZZ \oplus \ZZ/8 .
\]
Since
\[
(4,2) - 2(24,1) = (-44,0)
\]
the result follows from \Cref{prop:RIforn=p=2-2}.
\end{proof} 

We now have a calculation of the Spanier-Whitehead duals to $\LTE^{hF}$. 
If $G_{48} = G_{24} \rtimes \Gal(\FF_4/\FF_2)$, we have
\begin{align*}
\LTE^{hG_{48}} \simeq L_\LTK \mathbf{tmf}.
\end{align*}

\begin{thm}\label{thm:dualn=p=2} Let $F \subseteq G_{48} \subseteq \GG_2$ for $\GG_2=\Aut(\F_4, F_C)$ the stabilizer group associated to the formal group law $F_C$ of a supersingular elliptic curve $C$ with Weierstrass equation \eqref{eq:ellat2}.
Then
\[
D(\LTE^{hF}) \simeq \Sigma^{44} \LTE^{hF}.
\]
\end{thm} 

\begin{proof} First suppose $F \subseteq G_{24}$. Then by \Cref{rem:set-upn=p=2}  we have 
\[
D(\LTE^{hF}) \simeq (S^{-\HH_{ad}} \wedge \LTE)^{hF}
\]
and the result follows from \Cref{prop:class-of-key-repn=p=2}. 

The other possibility is that the composition $F \to G_{48} \to \Gal(\FF_4/\FF_2)$ is onto. Let $F_0$
be the kernel of this map an write $\Gal$ for $\Gal(\FF_4/\FF_2)$. We know from Lemma 1.37 of \cite{BobkovaGoerss}
that there is a $\Gal(\FF_4/\FF_2)$ equivariant equivalence
\[
\Sigma_+^\infty {\Gal} \wedge \LTE^{hF} \simeq \LTE^{hF_0}. 
\] 
We now have 
\begin{align*}
D(\LTE^{hF}) &\simeq D(\LTE)^{hF}\\
&\simeq [D(\LTE)^{hF_0}]^{h{\Gal}} \simeq (\Sigma^{44}E^{hF_0})^{h{\Gal}}\\
&\simeq [\Sigma_+^\infty {\Gal} \wedge \Sigma^{44} \LTE^{hF}]^{h{\Gal}}\\
&\simeq \Sigma^{44}\LTE^{hF}.\qedhere
\end{align*}
\end{proof}

\begin{rem}[{\bf Using the String orientation}]\label{rem:tmf-orientation} In \Cref{prop:RIforn=p=3-1}
and \Cref{prop:RIforn=p=2} we showed that (roughly) the restriction of the $J$-homomorphism
\[
J:[\Sigma^\infty_+ BF, \ko\langle 8\rangle] \to \pi_0\pic(E^{hF}) =  \pi_0\pic_F(E)
\]
is the zero map for various finite subgroups $F$ of the Morava Stabilizer Group $\GG_2$.
We used fixed point spectral sequence technology, but this can also be deduced from
the existence of the String orientation of $\mathbf{tmf}$ given in \cite{AndoHopRezk}. 

Suppose $V: BF \to B\Gl_1(S^0)$ defines an action of a finite group on the $0$-sphere $S^0$. Write $S^V$ for
this $F$-sphere. For any ring spectrum $R$, the composition 
\[
\xymatrix{
BF \ar[r]^-V & B\Gl_1(S^0) \ar[r] & B\Gl_1(R)
}
\]
defines an action of $F$ on $R$; this $F$-spectrum is equivalent to $R \wedge S^V$ with the trivial
action on $R$.

Now let $F$ be a finite subgroup of $\GG_2$ and let $R = \LTE^{hF}$. 
Then we can can write the $J(V) \in \sPic_F(\LTE)$ as
\[
J(V) = \LTE \wedge S^V \simeq \LTE \wedge_{\LTE^{hF}} \LTE^{hF} \wedge S^V.
\]
We have the diagonal action on both sides of this equation, although $F$ acts trivially on $\LTE^{hF}$. 

Next suppose the map $BF \to B\Gl_1(\LTE^{hF})$ is null-homotopic. Then we have an equivalence of $F$-spectra
$\LTE^{hF} \wedge S^V \simeq \LTE^{hF}$, and hence of elements 
\[
J(V) = \LTE \wedge S^V \simeq \LTE
\]
in $\sPic(R)$. Put another way, $J(V) = 0 \in \Pic(R) = \pi_0\pic_F(E)$. 

The existence of a String orientation $MO\langle 8 \rangle \to \mathbf{tmf}$ is proved by showing that the
composition 
\[
\ko\langle 8 \rangle \to \ko \to b\gl_1(S^0) \to b\gl_1(\mathbf{tmf})
\]
is null-homotopic. Now let $n=p=2$ and $F \subseteq G_{48}$. Then we have a map of ring spectra 
\[
L_\LTK \mathbf{tmf}\simeq \LTE^{hG_{48}} \longr \LTE^{hF},
\]
so we may conclude the composition
\[
\ko\langle 8 \rangle \to \ko \to b\gl_1(S^0) \to b\gl_1(\LTE^{hF})
\]
is null-homotopic and that the map
\[
[\Sigma_+^\infty BF,\ko\langle 8\rangle ] \longr \pi_0\pic_F(E)
\]
is zero. A similar statement holds at the prime $3$.

This approach, using the String orientation, works only at height $2$ because, ultimately, it depends on the geometry
of elliptic curves. In \Cref{sec:newxample} we will present a higher height example and, by necessity, return to
homotopy fixed point techniques. 
\end{rem}


\section{The Spanier-Whitehead duals of $\LTE^{hF}$: examples from higher height}\label{sec:newxample}

In this section, we fix a prime $p\geq 3$. We will work with the Honda formal group law $F_n$ of height $n=p-1$
over $\FF_{p^n}$.  Our intention is to use the theory we have developed to calculate $D(\LTE^{hF})$ for certain finite subgroups
of $\GG_n=\Aut(\F_{p^n}, F_{n})$. We will then use that calculation to make some remarks about exotic elements in the Picard group of
the $K(n)$-local category. The main results are \Cref{thm:dualn=p-1} and \Cref{thm:pic-exotic}. 

The group $\mathbb{S}_n = \Aut(F_n/\F_{p^n})$ contains a maximal finite subgroup
\[
G \cong C_p \rtimes C_{n^2}.
\]
There is also an extension of subgroups of $\GG_n $
\begin{align}\label{eq:Fext}
1 \to G \to H \to \Gal \to 1
\end{align}
where  $\Gal=\Gal(\F_{p^n}/\F_p)$. This is discussed, for example, in  \cite[3.6.3.1]{henn_res}.
We review some of these facts here.

Let $\WW = W(\FF_{p^n})$ be the Witt vectors on $\FF_{p^n}$ and let $\cO_n$ be the endomorphism ring of $F_n$; see \Cref{exmp:exam2}. We have an isomorphism
\[
\WW\langle S \rangle /(S^n=p, Sa = a^\sigma S) \cong \mathcal{O}_n
\]
where $a\in \WW$ and $\sigma \in \Gal$ is the Frobenius.  Let 
\[
\omega \in \FF_{p^n}^\times \subset \WW^\times \subset \mathbb{S}_n = \mathcal{O}_n^{\times}
\]
be a primitive $p^n-1$ root of unity. We then define elements of $\mathcal{O}_n$ by
\[
\tau = \omega^{\frac{p^n-1}{n^2}}\qquad\mathrm{and}\qquad X = \omega^{n/2}S.
\]
Then $X^n = -p$ and the element $\tau$ has order $n^2$. In particular, $\tau^n \in \FF_p^\times$ is a primitive $(p-1)$st
root of unity.  By Lemma 19 of \cite{henn_res} the subfield 
\[\Q_p(X) \subseteq \mathbb{D}_n = \QQ_p \otimes_{\ZZ_p} \cO_n\]
contains a primitive $p$th root of unity; we choose one such and call it $\zeta_p$. Since any root of unity must have
norm $1$, $\zeta_p \in \SS_n$. 

Since $X^n=-p$, $\Q_p(X)$ has degree $n=p-1$ over $\Q_p$ and it then follows that as subfields of $\mathbb{D}_n$
\[
\mathbb{Q}_p( \zeta_p) = \Q_p(X) .
\]
Let $C_p =\langle \zeta_p \rangle$ be the subgroup of $\SS_n$ generated by $\zeta_p$.
Conjugation by $\tau$ induces an automorphism of $C_p$  of order $p-1$, so there is a primitive root of unity in
$e\in (\Z/p)^{\times}$ such that
\[
\tau \zeta_p \tau^{-1} = \zeta_p^e.
\]
If we let $G \subseteq \SS_n$ be the subgroup generated by $\zeta_p$ and $\tau$, then we have an isomorphism
\[
G \cong C_p\rtimes C_{n^2}= \langle \zeta_p, \tau \mid \tau \zeta_p \tau^{-1} = \zeta_p^e \rangle .
\]
The extension $H$ is more subtle to describe and we won't need any of the details here. 

Let $ \Z(\zeta_p) \subseteq \cO_n$ be the subring generated $\zeta_p$; there is an isomorphism
$\Z[x]/\Phi_p(x) \cong   \Z(\zeta_p)$ where $\Phi_p(x) =(x^p-1)/(x-1)$ is the $p$th cyclotomic polynomial.\

\begin{lem}\label{lem:repsatp-1} Let
\[
\mathcal{E} = \Z(\zeta_p) \{1,\tau, \tau^2, \ldots, \tau^{n-1} \}  \subseteq \mathcal{O}_n
\]
be the sub-$\Z(\zeta_p)$-module generated by $\tau^i$, $0 \leq i \leq n-1$. 

(1) $\cE$ is stable under the conjugation action of $G$, and

(2) $\QQ_p \otimes \cE \cong \QQ_p \otimes_{\ZZ_p} \cO_n$.
\end{lem}

\begin{proof} Part (1) follows from the facts that $ \tau \zeta_p \tau^{-1} = \zeta_p^e$
and $\zeta_{p}\tau^{j}\zeta_{p}^{-1} = \zeta_{p}^{1-e^{j}} \tau^{j}$.

For (2) let $K = \QQ_p \otimes \cE$. Since $\tau^n \in \Z_p^{\times}$, $K$ is the sub-algebra of $\mathbb{D}_n$ generated by
$\zeta_p$ and $\tau$. By construction, $\tau$ is an $n^2$ root of unity over $\ZZ_p$. We examine $\tau$ in the extension
$\ZZ_p \subseteq W(\FF_{p^d})$. 

If the field $\F_{p^d}$ contains an $n^2$ root of 
unity, then $n \mid (p^d-1)/n$, which implies that $n|d$. Consider the field extensions
\[
\QQ_p \subset \QQ_p(\tau) \subseteq \QQ(\omega)
\]
where $\omega \in W(\FF_{p^n})$ is our chosen $(p^n-1)$st root of unity. Both
$\QQ_p(\tau)$ and $\QQ(\omega)$ are unramified extensions of $\QQ_p$ of degree $n$ and, hence,
$\Q_p(\tau) = \Q_p(\omega)$  and it follows that $\omega \in K$. However, since $X \in K$ and $\omega \in K$, $S \in K$, we have 
that $K = \mathbb{D}_n = \QQ_p \otimes_{\ZZ_p} \cO_n$. 
\end{proof}

\begin{rem}\label{rem:set-upn=p-1}
We find that we are exactly
in the situation of \Cref{prop:recog-princ-bis}. We let
\[V= \R \otimes \mathcal{E} \]
with action induced by the conjugation action of $G$ on $\cE$.  
 Thus our goal is to analyze the homotopy type of 
\[
(S^{-V} \wedge \LTE)^{hG} \in \Pic(\LTE^{hG}). 
\]  
\end{rem}

\begin{rem}\label{rem:reps-G-p-1} Here we describe the additive structure of the real orthogonal representation ring
$RO(G)$. 

Let $k$ be an integer, $C_k$ the cyclic of order $k$, and $\gamma = e^{2\pi i/k} \in \CC$. Multiplication by $\gamma^m$ on $\CC$
defines a $1$-dimensional complex representations $\CC(m)$ of $C_k$. As $\gamma^m$ and $\gamma^{k-m}$ are
complex conjugate, the representations $\CC(m)$ and $\CC(k-m)$ become isomorphic as $2$-dimensional
real representations.

Suppose $k$ is odd. Let $m \ne 0$ and let $\alpha_m$ be $\CC(m)$ regarded as a real representation. Then 
$\alpha_m$ is irreducible as a real representation and $RO(C_k)$ is generated  by the trivial one-dimensional representation 
$\mathbf{1}_\R$ and $\alpha_m$, $1 \leq m \leq (k-1)/2$.
If $k$ is even and $1\leq m <k/2$, write $\lambda_m$  for $\CC(m)$ regarded
as a real representation.
These are again irreducible. We also have the $1$-dimensional
sign representation $\sigma$ obtained by restriction along the unique quotient map $C_k \to C_2$.  Then
$RO(C_k)$ is generated by $\mathbf{1}_\R$, $\sigma$, and $\lambda_m$, $1 \leq m \leq (k-2)/2$.

Let $G = C_p \rtimes C_{n^2}$, with $n= p-1$.  We induce $\alpha_m$ along the inclusion $C_p \subseteq G$ to obtain
$n/2$ irreducible $2n^2$-dimensional real representations
\[
\Lambda_m := \mathrm{Ind}_{C_p}^{G} (\alpha_m),  \qquad  1\leq m\leq n/2.
\]
The representations $\Lambda_m$ are also restrictions of complex representations. 

There is an embedding $RO(C_{n^2}) \to RO(G)$ given by restriction along the quotient map $G \to C_{n^2}$. From this,
we conclude that
\[
RO(G) \cong RO(C_{n^2}) \oplus \Z\{ \Lambda_1, \ldots, \Lambda_{n/2}\}.
\]
This gives
\[
RO(G) \cong \Z\{ \mathbf{1}_\R , \sigma , \lambda_1, \ldots, \lambda_{(n^2-2)/2},  \Lambda_1, \ldots, \Lambda_{n/2}\}.
\]
Except for $\mathbf{1}_\R $ and $\sigma$, the listed generators of $RO(G)$ are all restrictions of complex representations.
A dimension count shows that we have a decomposition of the regular representation
\begin{equation}\label{eq:reg-rep-g-np}
\rho_G \cong \mathbf{1}_\R \oplus \sigma \oplus  \lambda_1\oplus  \cdots \oplus
\lambda_{(n^2-2)/2} \oplus \Lambda_1 \oplus \cdots \oplus  \Lambda_{n/2}.
\end{equation}
Note there are no repeated summands. 
\end{rem}

\begin{rem}\label{rem:groupcoh} We now fix some generators for the relevant cohomology groups.
Choose $z_0\in H^2(C_p, \Z_{(p)})$ be such that $z_0 = i^*c_1$ where $c_1 \in H^1(BU, \Z_{(p)})$ is the first Chern class and
$i \colon BC_p \to BU$ is the canonical map induced by the inclusion of $C_p \subseteq U(1)$ which maps the generator
$\zeta_p$ to $e^{2\pi i/p}$

Let $z= z_0^{p-1} \in H^{\ast}(C_p,\ZZ_{(p)}) = \Z_{(p)}[z_0]/pz_0$.
Since $\tau \zeta_p \tau^{-1} = \zeta_p^{e}$ where $e$ generates $\Z/p^{\times}$, the action of $\tau$ on $z_0$ is by multiplication
by a generator of $\Z/p^{\times}$ and so
\[
H^*(G, \Z_{(p)}) \cong H^*(C_p,\Z_{(p)})^{C_{n^2}} \cong \Z_{(p)}[z]/pz.
\]
We also let $y$ be a generator of (recall $n=p-1$)
\[
H^1(G, \Z_p^{\times}) \cong \Z/n\{y\}.
\]
\end{rem}

Next, we gather information about some characteristic classes which will be used below. 

\begin{rem}\label{rem:chernchar}
If $X$ is a space, let
$K^0(X)$ denote the complex $K$-theory of $X$ in degree $0$. Recall that the Chern
character
\[
ch \colon K^0(X) \to H^\ast(X,\QQ)
\]
is the unique ring homomorphism defined using the splitting principle and
the formula $ch(L) = \exp(c_1(L))$ when $L$ is a line bundle. If we
write $ch_k \in H^k(X,\QQ)$ for the $k$th homogeneous component of $ch$,  then if $L$ is a line bundle
\[
ch_k(L) = \frac{c_1(L)^k}{k!}
\]
and in general
\[
ch_k = \frac{s_k(c_1, \ldots, c_k )}{k!}
\]
where $s_k(c_1, \ldots, c_k )$ is $k$th Newton polynomial in the Chern classes. In particular, modulo decomposables,
\[
ch_k \equiv \alpha c_k,  \ \ \alpha \in \Q, \ \ \alpha \neq 0.
\]

We can evaluate $ch_k$ on the universal bundle over $BU$ and obtain a cohomology class $ch_k \in H^{2k}(BU,\QQ)$. The 
classes $ch_k$ are algebraic generators and primitives for the Hopf algebra structure on $H^\ast(BU,\QQ)$. 
If we choose the Bott class $v \in \pi_2BU = H_2(BU,\ZZ)$ so that $\langle c_{1},v \rangle = 1$, then the multiplicative
properties of the Chern character imply
\begin{equation}\label{eq:ch-detects-bott}
\langle ch_{k},v^{k}\rangle = 1
\end{equation}
and $ch_k$ is the unique primitive with this property. 

The map $BU \to BO$ classifying the underlying real bundle of the universal complex line bundle defines an isomorphism
\[
\xymatrix{
H^\ast (BO,\QQ) \ar[r]^-\cong & H^\ast(BU,\QQ)^{C_2}
}
\]
where $C_2$ acts via complex conjugation. If $L$ is a line bundle with conjugate $\overline{L}$, then $c_1(\overline{L}) =
-c_1(L)$, and it then follows that for any bundle $ch_k(\overline{\xi}) = (-1)^kch_k(\xi)$. 
Hence $ch_{2k} \in H^\ast (BO,\QQ)$
and $ch_{2k}(\xi)$ is defined for any real bundle $\xi$. Note that if $\xi_\RR$ is the real bundle underlying some complex vector
bundle $\xi$, then $ch_{2k}(\xi_\RR) = ch_{2k}(\xi)$.

If $1\leq k < p$, the defining expression for $ch_k$ makes sense over $\Z_{(p)}$ and in fact there is a unique lift of $ch_k$
to a class $ch_k \in H^{2k}(BU,\ZZ_{(p)})$. This gives a characteristic class
$ch_k(\xi) \in H^{2k}(X;\Z_{(p)})$ for any
complex vector bundle $\xi$ over $X$. Furthermore, the additivity of $ch_k$ over $\Q$ and the fact that 
\[
H^{2k}(BU \times BU;\Z_{(p)}) \to H^{2k}(BU \times BU ;\Q)
 \] 
is injective implies that $ch_k$ is additive, that is,
\[
ch_k(\xi_1 \oplus \xi_2) =ch_k(\xi_1)+ch_k(\xi_2) \in H^{2k}(X,\ZZ_{(p)})
\]
for any bundles $\xi_1,\xi_2$ over $X$.

If $p$ is odd, $k$ is even, and $k < p$, then we have 
\[ch_k \in H^\ast(BO,\ZZ_{(p)}) = H^\ast(BU,\ZZ_{(p)})^{C_2}\]
and we can define characteristic classes $ch_k(\xi)$ for any virtual real bundle as well. 
\end{rem}

\begin{rem}\label{rem:post-chern-char} Suppose $1\leq k<p$.  Let $\xi$ is a stable complex bundle
of virtual dimension $0$ and suppose the classifying map $\xi \colon X \to BU$ lifts to a map
\[
\xi:X\to BU\langle 2k\rangle.
\]
If $\ku$ is the connective complex $K$-theory spectrum, then $\xi$ is detected in the Atiyah-Hirzebruch spectral sequence
for $\ku_{(p)}^\ast (X)$ by the cohomology class $ch_{k}(\xi) \in H^{2k}(X, \pi_{2k}\ku_{(p)})$; that is, the class given by the
composition
\[
X\to BU\langle 2k\rangle \xrightarrow{ch_k} K(\Z_{(p)},2k).
\]
This follows from \eqref{eq:ch-detects-bott}.

If $p$ is odd, $k$ is even and $1 < k < p$ we can make a similar observation about a stable real bundle $\xi$ 
of virtual dimension $0$. Suppose the classifying  map $\xi \colon X \to BO$ lifts to $BO\langle 2k \rangle$. If $\ko$
is the connective real $K$-theory spectrum, then $\xi$ is detected in the Atiyah-Hirzebruch spectral sequence
for $\ko_{(p)}^\ast (X)$ by the cohomology class  $ch_{k}(\xi) \in H^{2k}(X, \pi_{2k}\ko_{(p)})$.

Note that since $p$ is odd $\ko_{(p)} = \ku^{hC_2}_{(p)}$ and
\[
\pi_\ast \ko_{(p)} = \ZZ_{(p)}[v^2] = \left(\pi_\ast \ku_{(p)}\right)^{C_2}.
\]
\end{rem}

\begin{rem}\label{rem:extchern} We can now relate the characteristic classes of \Cref{rem:chernchar} to the representations of
\Cref{rem:reps-G-p-1}. Suppose $p$ is odd, $n=p-1$ and $G = C_p \rtimes C_{n^2}$. Since $n$ is even, $ch_n$ is defined
for real vector bundles and we get a homomorphism
\[
ch_{n} \colon RO(G) \to H^{2n}(G, \Z_{(p)}).
\]
We will be most interested in representations $W$ which are the restriction of a complex representation and $ch_n(W)$
can then be computed using complex characteristic classes. Note that
\[
2ch_n(\mathbf{1}_\R) =  ch_n(2 \cdot\mathbf{1}_\R) = ch_n( \mathbf{1}_\CC) = 0,
\]
so $ch_n(\mathbf{1}_\R) = 0$. If $\sigma \in RO(G)$ is the sign representation, then $\sigma$ is obtained by
restriction along the unique quotient map $q:G \to \ZZ/2$, so $ch_n(\sigma) = 0$ as $H^{2n}(\ZZ/2,\ZZ_{(p)}) = 0$.
\end{rem}

We are ready to work with the $J$-homomorphism. 

\begin{prop}\label{prop:first-j-pn}
(1) The composite mapping
\[
[\Sigma^\infty_+BG, L_p^{\geq 2} \kon{2p}] \to [\Sigma^\infty_+BG,  L_p^{\geq 2} \ko] \to \pi_{0}\pic(\LTE)^{hG}
\]
is zero. 

(2) There is a homomorphism
\[\psi \colon RO(G) \to \Z \oplus \Z/2 \oplus \Z/p\]
which maps $W$ to $(\dim W, a,b)$ where
\begin{align*}
w_1(W) &= a y \\
ch_{n}(W) &= b z
\end{align*}
where  $y,z$ are as in \Cref{rem:groupcoh}.
\end{prop}
\begin{proof}
It is shown in \cite{matstohea_piceo} that
in the spectral sequence
\[
H^s(G,\pi_t\pic(\LTE)) \Longrightarrow \pi_{t-s}\pic(\LTE)^{hG}
\]
$E_\infty^{s,s} = 0$ for $s \geq 2p$. This gives (1).

For (2), we use the techniques of \Cref{rem:the-strategy-p}. From \Cref{rem:groupcoh} we know
$H^s(G,\ZZ_{(p)}) = 0$ for $0 < s < 2n$. Then \Cref{rem:post-chern-char} give us a filtration
\[
\xymatrix@C=10pt{
0 \ar[r] & A_{2n} \ar[d]_{ch_{2n}}  \ar[r]^-\subseteq & A_1 \ar[d]^{w_1} \ar[r]^-\subseteq & RO(G)/\ionp{2p}\ar[d]^{\mathrm{dim}}\\
&H^{n}(BG,\ZZ_{(p)})   & H^1(BG,\ZZ/2) & \ZZ
}
\]
which we use to define the desired homomorphism.
\end{proof}

\begin{prop}
(1) If $\rho_G$ is the regular representation of $G$, then 
\[
\psi(\rho_G) = \left(p n^2, 1, -n/2 \right).
\]

(2) The group $ RO(G)/(\ionp{2p} + \rho_{G})$ is generated by the $1$-dimensional real representation
$\mathbf{1}_{\mathbb{R}}$; this choice of generator determines an isomorphism
\[
\ZZ/2p^2n^2 \cong RO(G)/(\ionp{2p} + \rho_{G}).
\]
Furthermore, the $J$-homomorphism
\[
J_\LTE^{G}\colon RO(G)/(\ionp{2p} + \rho_{G}) \longr \Pic(\LTE^{hG})
\]
is an isomorphism. 
\end{prop}

\begin{proof} To prove (1), we use \eqref{eq:reg-rep-g-np} to see that $\rho_G$ contains a single copy of the sign
representation $\sigma$ and no other non-orientable direct summands. Therefore, $w_1(\rho_G)=1$.

To compute $ch_{n}(\rho_G)$, note that 
\begin{align*}
ch_{n}( \res^*\rho_G) &=n^2 ch_{n} ( \rho_{C_p}) \\
&=n^2 (ch_{n} (\alpha_1) \oplus \ldots \oplus ch_{n} ( \alpha_{n/2} ))
\end{align*}
However,
\[
ch_{n}(\alpha_m) = \frac{c_1(\alpha_m)^n }{n!} = -z.
\]
The first equality follows since $\alpha_m$ is one dimensional. The second equality follows since $c_1(\alpha_m)$ is non-zero, and any 
unit in $\Z/p$ raised to the power $n$ is equal to $1$. We have also used the fact that $n! =-1 \mod p$.
Finally, since $n^2=1 \mod p$, we have
\begin{align*}
ch_{n}(\rho_{G}) &= n^2 \left(ch_{2n}(\alpha_1)+\ldots+ch_{2n}(\alpha_{n/2}) \right)\\
&=-\frac{n}{2} z \ . 
\end{align*}

For (2), the fact that $RO(G)/(\ionp{2p} + \rho_{G}) \cong  \ZZ/2p^2n^2$ generated by $\psi(\mathbf{1}_{\R})=(1,0,0)$ is a 
computation using part (1) and part (2) of \Cref{prop:first-j-pn}. Furthermore, by \cite{matstohea_piceo}, $\Pic(\LTE^{hG}) \cong  \ZZ/2p^2n^2$ generated by $\Sigma \LTE^{hG} = \psi(\mathbf{1}_{\R} )$. 
\end{proof}

Finally, let $V = \RR \otimes \cE$ be as in \Cref{rem:set-upn=p-1}.
We need to identify the image of  $V$ under the $J$-homomorphism.

\begin{prop}\label{prop:class-of-key-repn=p-1} In $ RO(G)/(\ionp{2p} + \rho_{G})$
\[
V \equiv n^2(1+2p) \cdot   \mathbf{1}_{\R}
\]
\end{prop}

\begin{proof}
We show that
\[ \psi(V) = (n^2,0,-(n-1)n/2).\]
which implies that
\[  \psi(V) =( n^3p(n-1)+n^2)\cdot   \mathbf{1}_{\R}  \equiv n^2(1+2p) \mathbf{1}_{\R} \mod 2n^2p^2 \cdot \mathbf{1}_{\R}. \]

First, $\dim V =n^2$ which gives the first coordinate. To determine if $V$ is orientable, it suffices to restrict to the action of $\tau$. 
The vector space underlying $V$ has basis $\{\zeta_p^i \tau^j : 0\leq  i,j \leq n-1 \}$ and
\[\tau \zeta_p^i \tau^j  \tau^{-1} = \zeta_p^{ie} \tau^j.\]
So, as a representation of $C_{n^2}$, $V$ is $n$ copies of the same representation. Since $n$ is even, $V$ is orientable and so $w_1(V)=0$.

To compute $ch_{2n}(V)$, we note that after restricting to $C_p$, there is an isomorphism
\[
\res^*V \cong \mathbf{1}_{\R} \oplus (n-1)\rho_{C_p}.
\]
The action is given by
\[
\zeta_{p}\tau^{j}\zeta_{p}^{-1} = \zeta_{p}^{1-e^{j}} \tau^{j}.
\]
Note that $1-e^j \neq 0$ for $1\leq j\leq n-1$ as $e$ generates $\Z/p^{\times}$ and
$1+\zeta_p +\ldots + \zeta_p^n =0$. 
So, 
\[\R\{\zeta_p^i \tau^j : 0\leq i \leq n-1\} \cong \begin{cases} n \cdot \mathbf{1}_{\R} & j=0 \\
 \bar{\rho}_{C_p} & 1\leq j\leq n-1
\end{cases} \] 
where $ \bar{\rho}_{C_p}$ is the reduce regular representation. Noting that $\rho_{C_p} \cong   \mathbf{1}_{\R} \oplus \bar{\rho}_{C_p}$ proves the claim. From this, we conclude that
\[
ch_{2n}(V) = (n-1)ch_{2n}(\rho_{C_p}) = -(n-1)n/2.\qedhere
\]
\end{proof}

We can now have a calculation of the Spanier-Whitehead duals to $\LTE^{hF}$ for various finite subgroups $F$. 

\begin{thm}\label{thm:dualn=p-1} Let $n=p-1$ and  $F \subseteq H \subseteq \GG_n$ for $\GG_n=\Aut(\FF_{p^n},F_n)$ the stabilizer group of the Honda formal group law $F_n$ and for $H$ is as in \eqref{eq:Fext}. Then
\[
D(\LTE^{hF}) \simeq \Sigma^{-n^2(2p+1)} \LTE^{hF}.
\]
\end{thm} 

\begin{proof} First suppose $F \subseteq G$. By \Cref{rem:set-upn=p-1}, we have that
\[
D(\LTE^{hF}) \simeq (S^{-V} \wedge \LTE)^{hF}
\]
and the result follows from \Cref{prop:class-of-key-repn=p-1}. 

The other possibility is that the composition $F \to H \to \Gal(\FF_{p^n}/\FF_p)$ is non-trivial. Let $F_0$
be the kernel of this map an write $\widetilde{\Gal}$ for the image of the composite in $\Gal=\Gal(\FF_{p^n}/\FF_p)$ so that
\[ 1 \to F_0 \to F \to \widetilde{\Gal}  \to 1\]
is exact.
From Lemma 1.37 of \cite{BobkovaGoerss}, we can deduce that there is 
a $\widetilde{\Gal}$ equivariant equivalence
\[
\Sigma_+^\infty \widetilde{\Gal} \wedge \LTE^{hF} \simeq \LTE^{hF_0}. 
\] 
Letting $r = -n^2(2p+1)$, we now have
\begin{align*}
D(\LTE^{hF}) &\simeq D(\LTE)^{hF}\\
&\simeq [D(\LTE)^{hF_0}]^{h\widetilde{\Gal} } \simeq (\Sigma^{r}E^{hF_0})^{h\widetilde{\Gal}}\\
&\simeq [\Sigma_+^\infty \widetilde{\Gal}  \wedge \Sigma^{r} \LTE^{hF}]^{h\widetilde{\Gal}}\\
&\simeq \Sigma^{r}\LTE^{hF}.\qedhere
\end{align*}
\end{proof}

\begin{rem}\label{rem:various-comps}
We note that $\LTE^{hF}$ is always periodic of period $2p^2n^2$, though depending on $F \subseteq H$, the period could be shorter. 
For example, the period of $G$ and $H$ is exactly  $2p^2n^2$. That of $C_p$ is $2p^2$. 

Also note that if $p=3$, $E^{hF}$ is $72$-periodic, and 
\[
 \Sigma^{-n^2(2p+1)} \LTE^{hF} = \Sigma^{-28} \LTE^{hF} \simeq  \Sigma^{44} \LTE^{hF}
\]
Hence \Cref{thm:dualn=p=3} and \Cref{thm:dualn=p-1} produce the same shift. Note, however, that the formal
group of the supersingular elliptic curve used in \Cref{thm:dualn=p=3} is not isomorphic over $\FF_9$ to
the Honda formal group.
\end{rem}

We end with a simple but interesting application of \Cref{thm:dualn=p-1} to the study of the Picard group of the $\LTK$-local 
category. We refer the reader to \cite{HopkinsGross}, \cite{666}, and Section 2.4 of \cite{GHMR_Pic} for more background,
but recall some of the key ideas here.  We let $\mathrm{Pic}_n$ be the Picard group
of the homotopy category of $\LTK$-local spectra. For $X$ an invertible $\LTK$-local spectrum, $\LTE_*X$ is an invertible Morava 
module. If $\LTE_*X \cong \LTE_*S^0$ as Morava modules, we say that $X$ is \emph{exotic} and denote the subgroup of
exotic elements  in $\mathrm{Pic}_n$ by $\kappa_n$.

For a $\LTK$-local spectrum $X$, let $I_n(X)$ be the Gross-Hopkins dual of $X$. Gross-Hopkins duality
and Spanier-Whitehead duality are related by the equation
\[
I_n(X) \simeq I_n \smsh D(X).
\]
Furthermore, if $I_n = I_n(S^0)$, then the work of Gross and Hopkins implies that there is a $p$-adic $\GG$-sphere
$S\langle{\det} \rangle$
and
an element $P_n \in \kappa_n$  such that
\[
I_n \simeq  S^{n^2-n} \smsh S\langle{\det}\rangle \smsh P_n.
\]
The invertible $\LTK$-local spectrum $S\langle{\det} \rangle$ is described in great detail in \cite{BBGS} and the
spectrum $P_n$ is in fact defined by this equation.

As a consequence of \cite[Theorem 1.1]{BBS} which analyzes $I_n(\LTE^{hF})$ and Theorem~\ref{thm:dualn=p-1} above, we have the following result.
\begin{thm}\label{thm:pic-exotic}
Let $n=p-1$ and  $F=C_p$. Then,
\[P_n \smsh \LTE^{hF} \simeq \Sigma^{p^2+p} \LTE^{hF}  \]
In particular, $P_n$ is a non-trival element of $\kappa_n$.
\end{thm}
\begin{proof}
By the discussion above, we have a $\LTK$-local equivalence
\begin{align}\label{eq:firstequiv}
I_n(\LTE^{hF}) \simeq I_n \smsh D(\LTE^{hF}) \simeq  S^{n^2-n} \smsh S\langle{\det}\rangle \smsh P_n \smsh D(\LTE^{hF}) .
\end{align}
The spectrum $\LTE^{hF}$ is periodic with minimal periodicity $2p^2$. Using this, \Cref{thm:dualn=p-1} simplifies to
\[D(\LTE^{hF})\simeq   \Sigma^{-n^2(1+2p)} \LTE^{hF} \simeq   \Sigma^{-(p^2+1)} \LTE^{hF}. \]
Note that $F  \subseteq \ker(\det)$. In \cite{BBGS}, it is shown that this implies that 
\[
\LTE^{hF} \smsh S\langle{\det}\rangle \simeq \LTE^{hF}.
\] Furthermore, \cite[Theorem 1.1]{BBS} states that
$I_n(  \LTE^{hF}) \simeq \Sigma^{n^2}\LTE^{hF}$.

These facts together with \eqref{eq:firstequiv} imply that
\begin{align*}
\Sigma^{n^2}\LTE^{hF} & \simeq  S^{n^2-n} \smsh P_n \smsh \Sigma^{-(p^2+1)} \LTE^{hF} ,
\end{align*}
from which the first claim follows. Furthermore, since $p^2+p<2p^2$, $P_n$ cannot be equivalent to $L_{\LTK}S^0$ otherwise $\LTE^{hF}$ would have a shorter periodicity. 
\end{proof}

\bibliographystyle{alphaurl}
\bibliography{bib-dsphere}

\end{document}